%% file: jcp.tex
\documentclass[preprint]{elsarticle} %,review,,12pt,3p
\usepackage{graphicx}
\usepackage{amsmath}
\usepackage{amssymb}
\usepackage{amsfonts}	%pour utiliser \mathbb 
\usepackage{amsthm}
\usepackage{subfigure}
\usepackage{paralist}
\usepackage{bm}
\usepackage{url}
\usepackage{float}
\usepackage{mathrsfs}

\usepackage{algorithmic}
\usepackage{algorithm}
\usepackage{tikz}
\usepackage{booktabs}
\usepackage{textcomp}
\usepackage{pifont}
\usepackage{calc}
\usepackage{xspace}
\usepackage{rotating}
\usepackage{pgf,pgfplots}
\usepackage{stmaryrd}   %llbracket....
\usepackage[]{comment}
\usepackage[]{cancel}
\usepackage[]{hyperref}
\usepackage{array} %centrer le texte verticalement dans les cases des tab.
\usepackage{float}
 \newcommand{\dd}{\mathrm{d}}
\newcommand{\RR}{\mathbb{R}}
\newcommand{\NN}{\mathbb{N}}

\newcommand{\ZZ}{\mathbb{Z}}
\newcommand{\bx}{\boldsymbol{x}}

\newcommand{\bu}{\boldsymbol{u}}

\newcommand{\V}{\textbf{V}}
\newcommand{\Uc}{\textbf{U}}
\newcommand{\Fc}{\textbf{F}}

\newcommand{\Flux}[1]  {\llbracket #1 \rrbracket} 

\newcommand{\Lrho}  	{\widetilde{\rho}}

\newcommand{\Lrhoz}     {\widetilde{\rho_k {\cal Z}_k}}
\newcommand{\Lrhou}     {\widetilde{\rho u}}
\newcommand{\Lrhoe}     {\widetilde{\rho E\,}\!}

\newcommand{\Lz} 	{\widetilde{{\cal Z}_k}}
\newcommand{\rk}	{{\rho_k}}
\newcommand{\zk}	{{{\cal Z}_k}}
\newcommand{\z}		{{\cal Z}}
\newcommand{\y}		{{\cal Y}}
\newcommand{\rhoz}	{{\rho_k {\cal Z}_k}}

\newcommand{\yk}	{{{\cal Y}_k}}
\newtheorem{rem}{Remark}
\newtheorem{prop}{Proposition}

\newtheorem{thm}{Theorem}

\specialcomment{addSK}{ \color{red!50!black} } { \color{black} \ignorespacesafterend }
\specialcomment{removeSK}{ \color{blue!50} } { \color{black} \ignorespacesafterend }

\journal{J. Comput. Phys.}
\begin{document}

\begin{frontmatter}

\title{Simulation of sharp interface multi-material flows involving an arbitrary number 
of components through an extended five-equation model}

\author[ecn]{Marie Billaud Friess \corref{cor1}}
\ead{marie.billaud-friess@ec-nantes.fr}
\author[mdls,cea]{Samuel Kokh}
\ead{samuel.kokh@cea.fr}
\cortext[cor1]{Corresponding author}

\address[ecn]{LUNAM Universit\'e, GeM UMR CNRS 6183, Ecole Centrale Nantes, Universit\'e de Nantes,
1 rue de la No\"e, BP 92101,44321 Nantes Cedex 3, France }
\address[mdls]{Maison de la Simulation USR 3441, Digiteo Labs - b\^at. 565 - PC 190, CEA Saclay, 91191 Gif-sur-Yvette, France}
\address[cea]{DEN/DANS/DM2S/STMF, CEA Saclay, 91191 Gif-sur-Yvette, France}

\begin{abstract}
In this paper, we present an anti-diffusive method dedicated to the simulation of interface flows on Cartesian grids involving an arbitrary number $m$ of compressible components. Our work is two-fold: first, we introduce a $m$-component flow model that generalizes a classic two material five-equation model. In that way, interfaces are localized thanks to color function  discontinuities and a pressure equilibrium closure law is used to complete this new model. The resulting model is demonstrated to be hyperbolic under simple assumptions and consistent. Second, we present a discretization strategy for this model relying on an Lagrange-Remap scheme. Here, the projection step involves an anti-dissipative mechanism allowing to prevent numerical diffusion of the material interfaces. The proposed solver is built ensuring consistency and stability properties but also that the sum of the color functions remains equal to one. The resulting scheme is first order accurate and conservative for the mass, momentum, energy and partial masses. Furthermore, the obtained discretization preserves Riemann invariants like pressure and velocity at the interfaces. 
Finally, validation computations of this numerical method are performed on several tests in one and two dimensions. The accuracy of the method is also compared to results obtained with the upwind Lagrange-Remap scheme.
\end{abstract}


\end{frontmatter}

%%%%%%%%%%% INTRO 

\section{Introduction}

 The present paper deals with the simulation of compressible flows that involve $m$ distinct materials separated by sharp interfaces on a fixed Cartesian mesh. In our physical framework, there is no velocity jump across the material front and the interfaces are passively advected by the local velocity. We suppose that all diffusive processes are negligible and that each component is equipped with a specific Equation of State (EOS).
\\

The most straightforward simulation strategy consists in considering $m$ subdomains with free boundaries in such way that each subdomain is occupied by a single component throughout the computation. Such method is usually called front tracking~\cite{chern86, glimm85, juric96,  terashima2009,unverdi92}. It implies implementing a tracking procedure for the interfaces and coping with the boundary conditions for each subdomain thanks to appropriate jump conditions.
\\

We consider here another popular approach that relies on a single-fluid representation of the whole multi-material medium. The material interfaces are represented by loci of discontinuity of the medium physical properties. In our case these discontinuities will produce a switch between the various EOS of each component. Tracking these discontinuities is commonly achieved by introducing additional parameters. For example, in the well-known level set method widely studied in the case $m=2$ \cite{enright2002,Fedkiw1,Liu2003,Mulder1992,nguyen2002,sethian1996,wang2006,Yokoi2009,Zhao1996} the additional parameter is the signed distance function to the material interface. This function is then evolved thanks to an additional Partial Differential Equation (PDE). Few works propose an extension of the level set method to treat physical situations with $m > 2$ components \cite{Yokoi2009,Zhao1996} and ensuring conservation properties like partial masses or momentum conservation can be a complex task.
\\

Another approach relies on introducing discontinuous parameters $\z_k$, $k=1,\ldots,m$, often referred to as \textit{color functions}, such that $\z_k =1$ (resp. $\z_k=0$) in regions occupied (resp. not occupied) by the sole fluid $k$. It is possible to use physical parameters like the mass fractions of the volume fractions as color functions. Then, the flow and the interface locations are governed by a system of PDEs formed by the evolution equation of the physical unknown parameters of the components and the $m$ additional color functions. Unfortunately, standard discretization techniques like Finite Difference or Finite Volume methods tend to spread the discontinuities that represent the interfaces into several-grid cell wide regions due to numerical diffusion as depicted in figure \ref{fig:intro num diff}. This raises two issues: first, these transition zones may not be physically relevant as our primary model was designed for modelling sharp interfaces. Second, the diffused interface regions may expand over an important portion of the computational domain. Both issues can be circumvented by implementing interfaces reconstruction techniques. This strategy has been widely used in the case $m=2$ with the Volume of Fluid (VOF) method~\cite{Hirt1,lafaurie94,scardovelli99} and recent works like the Moment of Fluid (MOF) method have successfully addressed the case of $m>2$ ~\cite{Galera2,Kucharik1}. While these approaches provide a true sharp description of the interfaces in the discrete setting, they often remain complex to develop and implement.
\\

An alternate strategy have been used in~\cite{Allaire1,Allaire2} for the case $m=2$. In this framework, although the diffused interfaces remain \textit{a priori} not physically relevant, they are consistent (up to a numerical truncation error) with the target sharp discontinuities. This approach belong to a family of methods that was popularized through several publications in the past years  in \cite{Abgrall1,Abgrall2,Allaire1,Allaire2,karni1994,karni1996,Massoni1,saurel1999.1,saurel1999.2,shyue1998,shyue1999,shyue2001,quirk96,Wu1}. Concerning the control of the numerical diffusion produced at the material fronts, a special Lagrange-Remap method was proposed in~\cite{Kokh1} that spares the difficulty of interfaces reconstruction. This numerical scheme encompasses an anti-diffusive discretization for updating the color functions following the lines of~\cite{Despres2,despres2001, Despres1,Lagoutiere-PhD}.
\\

The present work addresses two questions: first, it proposes an extension for $m\geq 2$ of the isobaric five-equation interface capture model used in \cite{Allaire1,Allaire2}. Second, it presents  an extension of the anti-diffusive Lagrange-Remap strategy of \cite{Kokh1} for this model. Preliminary results have been announced in \cite{Billaud2}, and we intend to provide here a detailed study of the model within a larger choice of EOS for the material components.
\\

Before closing this introduction, we wish to refer the reader to several other significant efforts in the modelling and the simulation of multi-component flows like \cite{Boyer2006, Boyer2010, Boyer2011} by means of a phase-field description of the material interfaces and \cite{Dellacherie2003, Lagoutiere-PhD} for a study modelling possible mixture laws for of $m$-component flows. Moreover, we would like to mention that a similar extension of the five-equation model of \cite{Allaire1, Allaire2} has been considered by other authors independantly in \cite{Galon2014}.
\\

The paper is structured as follows. In section 2, we present our $m$-component interface capture model. Under simple hypotheses, we show that the isobaric closure is a consistent definition of a generalized EOS for the whole medium and that, granted simple thermodynamical assumptions, the overall PDE system is hyperbolic. We then detail the construction of the numerical scheme that relies on a Lagrange-Remap splitting as described in section 3. In the section 4, we specify the anti-diffusive strategy used for the remap step, following the work of \cite{Jaouen1}. Finally we show numerical results in 1D and 2D involving up to $m=5$ materials in section 5.

 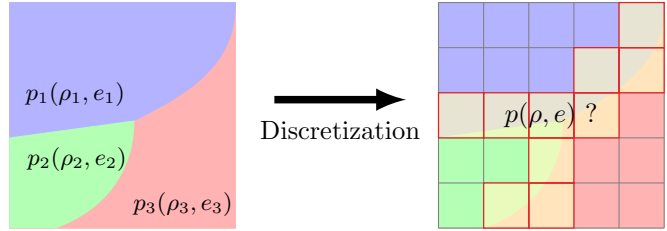
\begin{figure}[h!]
  \centering
  \begin{tikzpicture}[x=0.3cm,y=0.3cm]
\filldraw[fill=red!30,draw=red!30] (0,0) rectangle (10,10);
 \filldraw[fill=blue!30,draw=blue!30] (0,0) -- (0,10) -- (10,10) to[out=-90,in=70] (0,0); 
\filldraw[fill=green!30,draw=green!30] (0,0) -- (0,4) -- (5.5,4.75) to[out=-90,in=20] (2,0); 
\draw (0.3,5) node[above right]{\small $p_1(\rho_1,e_1)$}; 
\draw (5.6,4) node[below left] {\small $p_2(\rho_2,e_2)$}; 
\draw (10.3,2) node[below left] {\small $p_3(\rho_3,e_3)$}; 
\end{tikzpicture} 
\begin{tikzpicture}[x=0.2cm,y=0.4cm]
	\tikzstyle{fleche}=[->,>=latex,line width=1mm]
	\node (D) at ( 5,4){};
	\node (G) at (-5,4){};
	\node (0) at (0,0){};
	\node (texte) at (0,3){Discretization};
	\draw[fleche] (G)--(D);
\end{tikzpicture}
%\huge{$\textrightarrow$}
\begin{tikzpicture}[x=0.3cm,y=0.3cm]
\filldraw[fill=red!30,draw=red!30] (0,0) rectangle (10,10);
 \filldraw[fill=blue!30,draw=blue!30] (0,0) -- (0,10) -- (10,10) to[out=-90,in=70] (0,0); 
\filldraw[fill=green!30,draw=green!30] (0,0) -- (0,4) -- (5.5,4.75) to[out=-90,in=20] (2,0); 
%\draw[gray] (0,0) grid (10,10);
\draw[gray] (0,0) -- (0,10);
\draw[gray] (2,0) -- (2,10);
\draw[gray] (4,0) -- (4,10);
\draw[gray] (6,0) -- (6,10);
\draw[gray] (8,0) -- (8,10);
\draw[gray] (10,0) -- (10,10);
\draw[gray] (0,0) -- (10,0);
\draw[gray] (0,2) -- (10,2);
\draw[gray] (0,4) -- (10,4);
\draw[gray] (0,6) -- (10,6);
\draw[gray] (0,8) -- (10,8);
\draw[gray] (0,10) -- (10,10);
\filldraw[line width=0.5pt,draw=red,fill=yellow!30,opacity=0.7]
   (2,0) rectangle (4,2)
   (2,4) rectangle (4,6)
   (6,4) rectangle (8,6)
   (6,6) rectangle (8,8)
   (8,6) rectangle (10,8)
   (8,8) rectangle (10,10)
   (4,0) rectangle (6,2)
   (4,2) rectangle (6,4)
   (4,4) rectangle (6,6)
   (0,4) rectangle (2,6)
;
\draw (5,5) node {$p(\rho,e)$~?};
\end{tikzpicture}
\caption{Medium with multiple materials separated by interfaces.
Left: the fluids are sharply separated. Right:
Numerical diffusion spreads the interface over few cells. A specific EOS as to be defined in this region. \label{fig:intro num diff}}
\end{figure}

%%%%%%%%%%%%%%%%% MODEL

\section{The $m$-component interface flow model}
In this part, we propose to extend the two-material five-equation model with isobaric closure \cite{Allaire1,Allaire2} to multi-material flows with an arbitrary number of components.

\subsection{Evolution equations}
Let us  consider a medium composed of $m\geq 2$ compressible materials. Each component is equipped with an EOS
that is given as a smooth mapping $(\rho_k,p_k)\mapsto e_k(\rho_k,p_k)$, where $\rho_k$, $p_k$, $e_k$ denote respectively the density, partial pressure and specific internal energy of the fluid $k=1,\dots,m$. Let us note $\xi_k = (\partial \rho_k e_k / \partial p_k)_{\rho_k}$, we assume each pure fluid EOS to satisfy the following hypotheses
\begin{align}
\left(
\frac{\partial p_k}{\partial \rk}
\right)_{e_k}
+
\frac{{p}_k}{\rho_k^2}
\left(
\frac{\partial p_k}{\partial e_k}
\right)_{\rk}
&>0
,
\label{eq: hyp EOS 1}
\\
\xi_k
=
\left(
\frac{\partial \rho_k e_k}{\partial p_k}
\right)_{\rk}
&>0.
\label{eq: hyp EOS 2}
\end{align}
Relation \eqref{eq: hyp EOS 1} enables the classic definition of the sound velocity $c_k>0$ for the fluid $k$ by setting
\begin{equation}
c_k^2 = \left(
\frac{\partial p_k}{\partial \rk}
\right)_{e_k}
+
\frac{{p}_k}{\rho_k^2}
\left(
\frac{\partial p_k}{\partial e_k}
\right)_{\rk}
.
\label{eq:sound_k} 
\end{equation}
Relation~\eqref{eq: hyp EOS 2} implies that for a fixed given $\rho_k>0$, the mapping  $p_k\mapsto e_k(\rho_k,p_k)$
is strictly increasing, thus one-to-one.\\

So as to localize each fluid $k$, we introduce a color function noted $\zk$ that takes the value $\zk=1$ in the regions of the computational domain solely occupied by the fluid $k$ and the value $\zk=0$ otherwise. Finally, this quantity is supposed to remain bounded i.e. 
\begin{equation} 
\zk \in [0,1]
\label{eq:Z max}
\end{equation}
and to satisfy the constraint
\begin{equation}
\sum_{k=1}^m \zk =1.
\label{eq:Z unity}
\end{equation} 
All the flow components share the same velocity $\bu$. The density $\rho$ 
and the specific internal energy $\rho e$ of the $m$-component medium are defined by
\begin{equation}
\begin{aligned}
 \rho  = \sum_{k=1}^m \zk \rk \text{ ~~and~~} \rho e = \sum_{k=1}^m  \zk \rk e_k.
\end{aligned}
\label{eq:rho_e}
\end{equation}
We note $E = e + ||\bu||^2/2$ the specific total energy of the medium and we also set
\begin{equation}
\xi = \sum_{k=1}^{m} \zk \xi_k.
 \label{eq: def of xi}
\end{equation}
If we neglect all dissipative effects and volumic source terms, we consider that the flow is governed by the following system of $2m+2$ equations:
\begin{subequations}
\label{eq:syst1}
\begin{align}
\partial_t (\rho \bu) &+ \nabla \cdot (\rho \bu\otimes \bu) + \nabla p &=0,& 
\label{eq:syst1 momentum}
\\
\partial_t (\rho E)&+ \nabla \cdot (\bu(\rho E+p)) &=0,&
\label{eq:syst1  energy }
\\
\partial_t (\rhoz)&+  \nabla \cdot (\rhoz \bu ) &=0,& \quad  k=1,\dots,m,
\label{eq:syst1 partial mass}
\\
\partial_t {\cal Z}_k &+ \bu  \cdot \nabla {\cal Z}_k &=0,&\quad  k=1,\dots,m.
\label{eq:syst1 color function}
\end{align}
\end{subequations}
The system of equations \eqref{eq:syst1}  will be referred to as the {\it $m$-component model}.  This system is quasi-conservative since the equation~\eqref{eq:syst1 color function} is not conservative. We shall  see in ~\ref{section: eigenstructure analysis} that this is not an issue: the non-conservative product $\bu\cdot\nabla \zk$ is well-defined and one can recast system~\eqref{eq:syst1} into an  equivalent fully-conservative formulation.

%To close the system~\eqref{eq:syst1}, we assume that the multicomponent pressure law may be written as a the general mapping  $(\z_1,\ldots,\z_m,\z_1\rho_1,\ldots,\z_m\rho_m,\rho e)\mapsto p$. The specific case of the so-called isobaric closure \cite{Allaire1,Allaire2} shall be consider thereafter in section~\ref{section: isobaric closure}. 

{
\begin{rem}
System \eqref{eq:syst1} only involves conservation of partial masses $\mathcal{Z}_k\rho_k$. Conservation of global mass $\rho$ can be retrieved by summing over $k$ equations \eqref{eq:syst1 partial mass}.  In addition, using \eqref{eq:Z unity} system \eqref{eq:syst1} can be recast using the evolution equation for $\rho$, $\rho\bu$, $\rho E$ and $\rho_k \mathcal{Z}_k$, $\mathcal{Z}_k$ for $k=1,\ldots m-1$.
%
% is also equivalent to
% \begin{equation}
% \left\{
% \begin{aligned}
% \partial_t \rho &+ \nabla \cdot (\rho \bu)&=0,&
% \\
% \partial_t (\rho \bu) &+ \nabla \cdot (\rho \bu\otimes \bu) + \nabla p &=0,& 
% \\
% \partial_t (\rho E)&+ \nabla \cdot (\bu(\rho E+p)) &=0,&
% \\
% \partial_t (\rhoz)&+  \nabla \cdot (\rhoz \bu ) &=0,& \quad  k=1,\dots,m-1,
% \\
% \partial_t {\cal Z}_k &+ \bu  \cdot \nabla {\cal Z}_k &=0,&\quad  k=1,\dots,m-1.
% \end{aligned}
% \right.
% \label{eq:syst2}
% \end{equation}
In particular for $m=2$, we recover the so-called five-equation model of  \cite{Allaire1,Allaire2}.
%Despite the systems \eqref{eq:syst1} and \eqref{eq:syst2} are totally equivalent, the first form \eqref{eq:syst1} is considered in what follows to build our approximation strategy. 
\end{rem}
}
\subsection{Multicomponent Pressure Law: Isobaric Closure}\label{section: isobaric closure}

Several closure relations that define EOSs for multi-component system have already been examined in the litterature for the case of two-component or multi-component medium (see for example \cite{Allaire1, Allaire2,Dellacherie2003, Galera1,Lagoutiere-PhD,Massoni1}). We propose here a direct extension of the so-called isobaric closure law used in \cite{Allaire1, Allaire2}.
Given the fluid parameters $\zk$, $\rk$ and $\rho e$, we define $p$ by
\begin{equation}
 \rho e = \sum_{k=1}^m \zk \rk e_k(\rk,p).
 \label{eq:pressure}
\end{equation}
The above pressure definition is {\it consistent} for a wide range of EOSs in the sense that it is possible to find a unique solution 
$p$ satisfying \eqref{eq:pressure}. We have the following proposition.
\begin{prop}
\label{prop: pressure consistency}
Suppose that for $k=1,\ldots,m$ and any fixed $\rk>0$
\begin{align}
 \lim_{p_k \to 0} \rk e_k (\rk,p_k) &= 0,
\label{eq: consistency hyp 1}
\\
 \lim_{p_k \to +\infty} \rk e_k (\rk,p_k) &= +\infty,
\label{eq: consistency hyp 2}
\end{align}
then under the assumption \eqref{eq: hyp EOS 2}, there exists a single $p$ verifying 
\eqref{eq:syst1}.
\end{prop}
\begin{proof}
For the given values $\zk,\rk>0$, $k=1,\ldots,m$ and $\rho e > 0$ , we consider the mapping 
$
\Xi: p\mapsto \sum_{k=1}^{m} \zk \rk e_k(\rk,p) - \rho e
$. 
The assumptions \eqref{eq: consistency hyp 1} and \eqref{eq: consistency hyp 2}
imply respectively that 
$\lim_{p\to 0} \Xi = -\rho e < 0$ 
and 
$\lim_{p\to +\infty} \Xi = +\infty$. The theorem of intermediate value then states that there exists $\bar p$ such that
$\Xi(\bar p)=0$.
According to \eqref{eq: hyp EOS 2}, $\dd \Xi / \dd p = \sum_{k=1}^{m} \zk \xi_k > 0 $ and thus 
$\Xi$ is one-to-one. This implies that $\bar p$ is unique.
\end{proof}
The proposition~\ref{prop: pressure consistency} is valid for a wide range of EOSs since it
does not require any specific analytical form for the EOSs and enables the use of tabulated data. Assumptions \eqref{eq: consistency hyp 1}-\eqref{eq: consistency hyp 2} are sufficient but not necessary conditions for obtaining a proper definition of the pressure $p$. Indeed, for a $m$-component flow involving $m$ Mie-Gruneisen materials \eqref{eq: consistency hyp 1} and \eqref{eq: consistency hyp 2} are not  necessarily verified, but simple calculations allow to properly define $p$  as in  proposition~\ref{prop: isobaric pressure mie-gruneisen}.

 \begin{prop}
 \label{prop: isobaric pressure mie-gruneisen}
Suppose that for each $k=1,\dots,m$ the component $k$ is a Mie-Gruneisen material whose EOS reads
$$
\begin{aligned}
(\rho_k,p_k)
&\mapsto
\rho_k e_k = \rho_k e_k^{\mathrm ref}(\rk) 
+
 \frac{p_k-p_k^{\mathrm ref}(\rho_k)}{\Gamma_k(\rho_k)},
\\
(\rho_k,e_k)
&\mapsto
p_k = 
p_k^{\mathrm ref}(\rho_k)
+
{\Gamma_k}(\rk) \rk [e_k - e^\mathrm{ref}_k(\rk)],
\end{aligned}
$$
where $\rk\mapsto {\Gamma_k}(\rk)$, $\rk\mapsto {e_k^\mathrm{ref}}(\rk)$ and $\rk\mapsto p^\mathrm{ref}_k(\rk)$ are real valued functions. Furthermore, we have
{$$
\begin{aligned}
\xi_k &= \frac{1}{\Gamma_k}
,
\\
c_k^2 
&=
\frac{\dd p^\mathrm{ref}_k}{\dd \rk} 
+
\frac{p_k-p_k^\mathrm{ref}}{\Gamma_k}
\frac{\dd \Gamma_k}{\dd \rk} 
-
\rk\Gamma_k
\frac{\dd e_k^\mathrm{ref}}{\dd \rk} 
+(\Gamma_k+1)\frac{p_k - p^\mathrm{ref}_k}{\rk}
.
\end{aligned}
$$
}
For the $m$-component system \eqref{eq:syst1}, the unique solution pressure $p$ defined by the isobaric closure~\eqref{eq:pressure} is
$$
p=
\left[
\rho e 
- \sum_{k=1}^m \zk \rk {e}^\mathrm{ref}_k
+ \sum_{k=1}^m \zk \frac{p^\mathrm{ref}_k}{\Gamma_k}
\right]
\times
\left[
\sum_{k=1}^{m}
\frac{\zk}{\Gamma_k}
\right]^{-1}.
$$
 \end{prop}
Two very common choices of EOS are often found in the literature: the Stiffened Gas and the Perfect Gas which  belong to the class of Mie-Gruneisen materials with
\begin{align}
 \text{Perfect Gas:}&&
\Gamma_k &= \gamma_k - 1,&
p^\mathrm{ref}_k &= 0,&
e^\mathrm{ref}_k &= 0,
\label{eq: eos pg}
\\
 \text{Stiffened Gas:}&&
\Gamma_k &= \gamma_k - 1,&
p^\mathrm{ref}_k &= -\gamma_k\pi_k,&
e^\mathrm{ref}_k &= 0,
\label{eq: eos sg}
\end{align}
where $\gamma_k$ and $\pi_k$ are constants.\\

\subsection{Hyperbolicity and eigenstructure}
{
We now briefly study the hyperbolicity of  the $m$-component model with isobaric closure (more details are given in \ref{section: eigenstructure analysis}). System~\eqref{eq:syst1} in one space dimension for smooth solutions formulated for the primitive variables $\mathbf{V} = (\rho, u, p,\z_1,\dots,\z_m,\y_1,\dots,\y_{m-1})^T$ reads
\begin{equation}
\partial_t \mathbf{V} 
+ 
\mathbf{A}(\mathbf{V})
\partial_x \mathbf{V} 
=0
,
\label{eq:prim}
\end{equation}
where $\mathbf{A}(\mathbf{V})$ is a $(2m+2)\times(2m+2)$ matrix defined by
$$
\mathbf{A}(\mathbf{V})
=
\begin{pmatrix}
 \mathbf{A}_1 & 0
 \\
  0  & u \mathbf{I}_{2m-1}
\end{pmatrix}
,\quad
\mathbf{A}_1(\mathbf{V})
=
\begin{pmatrix}
  u & \rho & 0 
  \\
  0 & u & 1/\rho 
  \\
  0 & \rho c^2 & u 
\end{pmatrix}
$$
and $\mathbf{I}_{2m-1}$ is the $(2m-1)\times(2m-1)$ identity matrix. 
}
{
\begin{prop} 
\label{prop:hyperbolicity}
The sound velocity $c$ of the $m$-component model defined by relation
\begin{equation}
 \rho \xi c^2 = \sum_k \zk \rk \xi_k c_k^2.
 \label{eq: sound velocity}
\end{equation}
is a positive real number and the matrix $\mathbf{A}(\mathbf{V})$ is diagonalizable. The eigenvalues of $A$ are $\lambda_1 = u - c, \lambda_2 = u + c, \lambda_3 = \cdots = \lambda_{2m+2} = u$. The fields associated with the eigenvalues $\lambda_1$ and $\lambda_2$ are genuinely nonlinear, those associated with $\lambda_k$ are linearly degenerate for $k=3,\ldots,2m+2$.
Therefore, the $m$-component system \eqref{eq:syst1} is hyperbolic.
\end{prop}
}

%%%%%%%%%%%% DISCRETISATION 

\section{{Lagrange-Remap scheme}}
{
In this section, we present a general quasi-conservative finite-volume scheme for approximating the solution of system~\eqref{eq:syst1}. Here, we use a two-step Lagrange-Remap strategy \cite{Despres3,Godlewski1}: the first step accounts for acoustic effects while the second step deals with material transport of the components. For the sake of simplicity, we present our numerical scheme in the case of one-dimensional problems. The extension of the scheme to multi-dimensional problems when using a Cartesian discretization of the computational domain is achieved in this work thanks to a simple dimensional splitting detailed in section~\ref{section: multidim}.
}
Here, we consider a regular discretization of the real line into a set of cells $\left([x_{i-1/2}, x_{i+1/2}]\right)_{i\in\ZZ}$, where $x_{i+1/2} = i\Delta x$  for $i\in\ZZ$ and $\Delta x>0$ is the space step. Let $x_{i}= \frac{x_{i+1/2}+x_{i-1/2}}{2}$ be the center of the cell $i$, for $i\in\ZZ$. Let $(t^n)_{n\in\NN}$ be the regular sequence of instants $t^n=n\Delta t$, where $\Delta t$ is the time step. If $(x,t)\mapsto a(x,t)$ is a fluid parameter, we consider the discrete variable $a_i^n$ to be an approximation of 
$$
\frac{1}{\Delta x}
\int_{x_{i-1/2}}^{x_{i+1/2}} a(x,t^n) dx
,\quad
i \in \ZZ, n \in \NN.
$$
Discrete values located at the cell boundary $x=x_{i+1/2}$ are denoted using the subscript $(i+1/2)$; for the sake of readability discrete variables after the Lagrangian step are designed by $\tilde{\cdot}$, such as $\tilde{a}_{i+1/2}$ or $\tilde{a}_i$. {Finally, we denote by $\Flux{a}_i = a_{i+1/2}-a_{i-1/2}$ the flux difference in a cell $i$.  }

\subsection{Lagrangian step}
The Lagrangian step consists in solving the Euler equations in Lagrangian coordinates (see Appendix A. or for more details \cite{Despres3, Godlewski1}). Following \cite{Despres1,Kokh1}, we perform this task using the acoustic scheme \cite{Despres3}. The acoustic scheme can be obtained using a Suliciu-type relaxation approach \cite{Bouchut2004,Coquel2001,Suliciu1990} and it can be considered as a particular HLLC \cite{Batten1997,Toro1999} solver for the system written in Lagrangian coordinates. Let us mention that other approximate Riemann solvers may be used to achieve the approximation of the Lagrangian step like for example a Roe-type linearization \cite{Roe1981} or a Rusanov scheme \cite{Rusanov1961}. If we note  $L_i = 1 +  \frac{\Delta t}{\Delta x} \Flux{u^n}_i $, this numerical scheme provides the update relations
\begin{equation}
\left\{
\begin{array}{rcl}
 L_i (\Lrhou)_i&=& (\rho u)^n_i - \frac{\Delta t}{\Delta x} \Flux{p^n}_i, \\[0.15cm]
 L_i (\Lrhoe)_i &=& (\rho E)^n_i - \frac{\Delta t}{\Delta x} \Flux{p^nu^n}_i, \\[0.15cm]
 L_i (\Lrhoz)_i &=& (\rhoz)^n_i, \\[0.15cm]
\widetilde{\cal Z}_{k,i}           &=&   {\cal Z}_{k,i}^n,
\end{array}
\right.
\label{eq:lag}
\end{equation}
where the fluxes $p^n_{i+1/2}$ and $u^n_{i+1/2}$ are defined by
\begin{equation}
\left\{
\begin{array}{rcl}
p^n_{i+1/2} &=& \dfrac{p_i^n+p^n_{i+1}}{2} - \dfrac{1}{2} (\rho c)^n_{i+1/2} (u^n_{i+1}-u^n_{i}),\\[0.25cm]
u^n_{i+1/2} &=& \dfrac{u_i^n+u^n_{i+1}}{2} - \dfrac{1}{2} \dfrac{1}{(\rho c)^n_{i+1/2}} (p^n_{i+1}-p^n_{i}),
\\
(\rho c)^n_{i+1/2} &=& \sqrt{\max \left( (\rho c^2)^n_i,(\rho c^2)^n_{i+1}\right) 
\min\left(\rho_i^n,\rho_{i+1}^n\right)},
\end{array}
\right.
\label{eq:acou}
\end{equation}
and  
 $(\rho c^2)^n_i$ is computed according to \eqref{eq: sound velocity}.\\ 

As in \cite{Despres1,Kokh1}, $\Delta t$ is chosen so that it verifies the following Courant-Friedrichs-Lewy (CFL) condition
\begin{equation}
\frac{\Delta t}{\Delta x} \max_{i\in\ZZ}\left(|u^n_{i+1/2}|, (\rho c)^n_{i+1/2}/\min(\rho^n_i,\rho^n_{i+1})\right) \leq C_\mathrm{CFL},
\label{eq:CFL}
\end{equation}
where $0\leq C_\mathrm{CFL} \leq 1$. The condition~\eqref{eq:CFL} ensures stability for the Lagrange step~\eqref{eq:lag}. 

\subsection{Remap step}
The remap step accounts for the material transport of the fluid \cite{Despres3,Godlewski1}. During this phase, the updated Lagrangian variables issued from \eqref{eq:lag} are projected onto the Eulerian mesh. Following the classic lines of \cite{Godlewski1} and as suggested by \cite{Billaud1,Despres1,Kokh1} we achieve this task by setting
%\begin{equation}
%\left\{
%\begin{array}{rclclcl}
%(\rho u)^{n+1}_i &=& (\Lrhou)_i &-& \frac{\Delta t}{\Delta x} \Flux{u^n\, \Lrhou}_i &+& \frac{\Delta t}{\Delta x} \Flux{u^n}_i\, (\Lrhou)_i , \\[0.15cm]
%(\rho E)^{n+1}_i&=& (\Lrhoe)_i &-& \frac{\Delta t}{\Delta x} \Flux{u^n\, \Lrhoe}_i &+& \frac{\Delta t}{\Delta x} \Flux{u^n}_i\, (\Lrhoe)_i , \\[0.15cm]
%(\rho_k \zk)^{n+1}_i &=& (\Lrhoz)_i &-& \frac{\Delta t}{\Delta x} \Flux{u^n\, \Lrhoz}_i &+& \frac{\Delta t}{\Delta x} \Flux{u^n}_i\, (\Lrhoz)_{i}, \\[0.15cm]
%{\cal Z}^{n+1}_{k,i}   &=&  \widetilde{\cal Z}_{k,i}    &-& \frac{\Delta t}{\Delta x} \Flux{u^n\, \Lz}_i    &+& \frac{\Delta t}{\Delta x} \Flux{u^n}_i\, \widetilde{\cal Z}_{k,i},
%\end{array}
%\right.
%\label{eq:rem}
%\end{equation}
\begin{equation}
\mathbf{W}^{n+1}_i = \widetilde{\mathbf{W}}_i  - \frac{\Delta t}{\Delta x} \Flux{u^n\, \widetilde{\mathbf{W}}_i} + \frac{\Delta t}{\Delta x} \Flux{u^n}_i\, \widetilde{\mathbf{W}}_i, 
\label{eq:rem}
\end{equation}
where $\mathbf{W}=[\rho u, \rho E, \rho_1 {\cal Z}_1, \dots,  \rho_m {\cal Z}_m, {\cal Z}_1, \dots,  {\cal Z}_m]^T$. Here, the fluxes $\Flux{u^n}_i$ are computed using $u^n_{i+1/2}$ and $u^n_{i-1/2}$ given by  \eqref{eq:acou} and the fluxes $\Lrho_{i+1/2}, (\Lrhou)_{i+1/2}$, $ (\Lrhoe)_{i+1/2}$, $(\Lrhoz)_{i+1/2}$ and $\Lz_{,i+1/2}$ are defined by
\begin{equation}
\left\{
\begin{array}{lcl}
\widetilde{\rho}_{i+1/2}       &=&  \displaystyle \sum_{k=1}^m \widetilde{\cal Z}_{k,i+1/2} \widetilde{\rho}_{k,i+1/2},\\
(\widetilde{\rho e})_{i+1/2}  &=&  \displaystyle \sum_{k=1}^m \widetilde{\cal Z}_{k,i+1/2} (\widetilde{\rho_k e_k})_{i+1/2},\\
(\widetilde{\rho u})_{i+1/2}    &=&  \displaystyle \widetilde{\rho}_{i+1/2} \widetilde{u}_{i+1/2},\\[0.25cm]
(\widetilde{\rho E})_{i+1/2} &=&  \displaystyle  \widetilde{(\rho e)}_{i+1/2} 
+ 
\widetilde{\rho}_{i+1/2}(\widetilde{u}_{i+1/2})^{2}/2, 
\\[0.25cm]
 (\widetilde{\rho_k {\cal Z}_k})_{i+1/2} &=& \displaystyle  \widetilde{\cal Z}_{k,i+1/2} \widetilde{\rho}_{k,i+1/2}.
\end{array}
\right.
\label{eq: remap fluxes structure}
\end{equation}
Then, we choose to define all the phasic values $\widetilde{\rho}_{k,i+1/2}$, $\widetilde{e}_{k,i+1/2}$ and the velocity $\widetilde{u}_{i+1/2}$ by taking their upwind value with respect to the velocity $u^n_{i+1/2}$, namely
\begin{equation}
(\widetilde{\rho}_k, \widetilde{\rho_k e}_k, \widetilde{u})_{i+1/2}=
\begin{cases}
(\widetilde{\rho}_k, \widetilde{\rho_k e}_k, \widetilde{u})_{i},&\text{if $u^n_{i+1/2}>0$,}\\
(\widetilde{\rho}_k, \widetilde{\rho_k e}_k, \widetilde{u})_{i+1},&\text{if $u^n_{i+1/2}\leq0$.}
\end{cases}
\end{equation}

At this step the color function fluxes $\widetilde{\z}_{k,i+1/2}$ can be computed in several manners. A very common way is to use an upwind flux  
\begin{equation}
\widetilde{\z}_{k,i+1/2}={\cal Z}^{up}_{k,i+1/2} = 
 \left\{
\begin{array}{cl}
 {\cal Z}_{k,i}^{n} & \text{ if } {u}_{i+1/2} >0, \\
 {\cal Z}_{k,i+1}^{n} & \text{ otherwise. }
 \end{array}
\right.
\label{eq: upwind flux}
\end{equation}
Nevertheless, despite its simplicity and satisfying good properties: unit constraint \eqref{eq:Z unity},  stability, 
it is well known that the resulting scheme is very diffusive in particular for the transport of $\zk$ (see section \ref{section: simulations}). To overcome this curse, high order schemes as Minmod limiter, Superbee limiter or Ultrabee limiter schemes can be envisaged. Nevertheless, it was demonstrated \cite{Jaouen1} that such schemes present defects for the passive transport of $m>2$ scalar functions since they can not respect at same time stability and unit constraint. Here, we propose and detail an anti-diffusive procedure to appropriately define the fluxes $\widetilde{\z}_{k,i+1/2}$  in section \ref{section : fluxes z} insuring stability and unit constraint \eqref{eq:Z unity}, and that allows accurate computation of $\zk$ during remap step.

\begin{rem}
 The cell-centered quantities $\widetilde{\rho}_{k,i}$ and $\widetilde{e}_{k,i}$ are defined by 
$\widetilde{\rho}_{k,i} = (\widetilde{\rho_k \zk})_i / \widetilde{\z_{k,i}}$, $\widetilde{e_{k,i}}=e_k(\widetilde{\rho_{k,i}},\widetilde{p_i})$. If $|\widetilde{\z}_{k,i}|$ is lower than a user-defined threshold value, we set $\widetilde{\rho}_{k,i}=\widetilde{e}_{k,i}=0$ while evaluating \eqref{eq: remap fluxes structure}.
\end{rem}

\subsection{Overall algorithm}\label{section: overall algorithm}
Let us summarize the overall algorithm detailed in the previous sections. For updating the variables  at each instant $t^n$ to their values at $t^{n+1}$, we proceed following the steps of the algorithm \ref{alg:procedure1}.\\

{ 
\begin{algorithm}[H]
{\bf I - Lagrangian step:}
\begin{algorithmic}[1]
\STATE Compute the acoustic fluxes $p_{i+1/2}^n, u_{i+1/2}^n$ thanks to \eqref{eq:acou}.
\STATE Compute the time step $\Delta t$ according to the CFL constraint \eqref{eq:CFL}.
\STATE Update  $\Lrhou,  \Lrhoe,  \Lrhoz, \Lz$ with \eqref{eq:lag}.
\end{algorithmic}
{\bf II - Remap step:}
\begin{algorithmic}[1]
\STATE {Compute the fluxes $\widetilde{\cal Z}_{k,i+1/2}$ for each interface $i+1/2$.}
%\IF {$u^n_{i+1/2}>0$ and $u^n_{i-1/2}>0$} 
%\STATE
%$
%  \widetilde{\cal Z}_{k,i+1/2} =
% \begin{cases}
%d^n_{k,i+1/2},&
%\text{if $\widetilde{\cal Z}_{k,i+1/2}^\DO <d^n_{k,i+1/2}$,}
%\\
%\widetilde{\cal Z}_{k,i+1/2}^\DO,&
%\text{if $\widetilde{\cal Z}_{k,i+1/2}^\DO \in [d^n_{k,i+1/2},D^n_{k,i+1/2}]$,}
%\\
%D^n_{k,i+1/2},&
%\text{if $\widetilde{\cal Z}_{k,i+1/2}^\DO >D^n_{k,i+1/2}$,}
% \end{cases}
%$\\[0.2cm]
%\ELSIF {$u^n_{i+1/2}<0$ and $u^n_{i+3/2}<0$}
%\STATE 
%$
%  \widetilde{\cal Z}_{k,i+1/2} =
% \begin{cases}
%d^n_{k,i+1/2},&
%\text{if $\widetilde{\cal Z}_{k,i+1/2}^\DO <d^n_{k,i+1/2}$,}\\
%\widetilde{\cal Z}_{k,i+1/2}^\DO,&
%\text{if $\widetilde{\cal Z}_{k,i+1/2}^\DO \in [d^n_{k,i+1/2},D^n_{k,i+1/2}]$,}\\
%D^n_{k,i+1/2},&
%\text{if $\widetilde{\cal Z}_{k,i+1/2}^\DO >D^n_{k,i+1/2}$,}
% \end{cases}
%$ 
%\ELSE
%\STATE $ \widetilde{\cal Z}_{k,i+1/2}$ is the  upwind value.
%\ENDIF
\STATE Compute the remap fluxes of the conservative quantities $\Lrho_{i+1/2}$, $\Lrhou_{i+1/2}$, $\Lrhoe_{i+1/2},$ $(\Lrhoz)_{i+1/2}$  with \eqref{eq: remap fluxes structure}.
\STATE Update $(\rho u)^{n+1}$,  $(\rho E)^{n+1}$,  $(\rho_k \zk)^{n+1}$, $\zk^{n+1}$ with {\eqref{eq:rem}}.
\end{algorithmic}
\caption{Lagrange-Remap procedure}
\label{alg:procedure1}
\end{algorithm} 
 }

Combining the Lagrange \eqref{eq:lag} and Remap \eqref{eq:rem} steps, we obtain the following quasi-conservative scheme
\begin{equation}
\left\{
\begin{array}{rclclclr}
(\rho u)^{n+1}_i &=&  (\rho u)^n_i &-& \frac{\Delta t}{\Delta x} \Flux{u^n\, \Lrhou}_i &-& \frac{\Delta t}{\Delta x} \Flux{p^n}_i, \\[0.15cm]
(\rho E)^{n+1}_i&=&  (\rho E)^n_i &-& \frac{\Delta t}{\Delta x} \Flux{u^n\, \Lrhoe}_i &-& \frac{\Delta t}{\Delta x} \Flux{p^nu^n}_i , \\[0.15cm]
(\rho_k \zk)^{n+1}_i &=& (\rho_k \zk)^n_i &-& \frac{\Delta t}{\Delta x} \Flux{u^n\, \Lrhoz}_i, \\[0.15cm]
{\cal Z}^{n+1}_{k,i}   &=&  {\cal Z}^n_{k,i}    &-& \frac{\Delta t}{\Delta x} \Flux{u^n\, \Lz}_i    &+& \frac{\Delta t}{\Delta x} \Flux{u^n}_i\, \widetilde{\cal Z}_{k,i}. 
\end{array}
\right.
\end{equation}
As in \cite{Kokh1}, the overall discretization strategy is conservative with respect to the momentum, total energy and partial masses. 

\subsection{Extension to the multi-dimensional case}\label{section: multidim}
Without loss of generality, we consider a two-dimensional  problem that is discretized over a Cartesian grid. Noting $\bu=(u_1,u_2)^T$ the velocity, 
the $m$-component system~\eqref{eq:syst1} reads 
\begin{equation}
\left\{
\begin{aligned}
\partial_t \Uc 
+
\partial_{x_1} \Fc_1(\Uc,\z_1,\dots,\z_m) 
+
\partial_{x_2} \Fc_2(\Uc,\z_1,\dots,\z_m) 
&= 0
,
\\
\partial_t \zk 
+
u_1\partial_{x_1} \zk 
+
u_2\partial_{x_2} \zk 
&= 0,  
\quad k=1,\dots,m
,
\end{aligned}
\right.
\label{eq: 2D syst}
\end{equation}
where 
$\Uc=[\rho u_1, \rho u_2, \rho E, \rho_1\z{}_1,\ldots,\rho_m\z{}_m]^T$, and 
$$
\begin{aligned}
 \Fc_1(\Uc,\z_1,\dots,\z_m)
&= 
  [\rho u_1^2  + p , \rho u_1u_2, (\rho E + p)u_1,
   \rho_1 \z_1 u_1 ,\dots,\rho_m \z_m u_1]^T
,\\
 \Fc_2(\Uc,\z_1,\dots,\z_m)
&=
 [\rho u_1 u_2, \rho u_2^2 + p ,(\rho E + p) u_2,\rho_1 \z_1 u_2,\dots,\rho_m \z_m u_2]^T
.
 \end{aligned}
$$
We approximate the solution of \eqref{eq: 2D syst} by means of a dimensional splitting strategy. This boils down to solve successively  in each space direction
$$
\left\{
\begin{array}{rcl}
\partial_t \Uc +\partial_{x_1} \Fc_1(\Uc,\z_1,\dots,\z_m) &=& 0
,\\
\partial_t \zk + u_1\partial_{x_1} \zk &=& 0,  
\quad k=1,\dots,m,
\end{array}
\right. 
$$
then
$$
\left\{
\begin{array}{rcl}
\partial_t \Uc +\partial_{x_2} \Fc_2(\Uc,\z_1,\dots,\z_m) &=& 0
,\\
\partial_t \zk + u_2\partial_{x_2} \zk &=& 0, 
\quad k=1,\dots,m
.
\end{array}
\right. 
$$
In the first (resp. second) step, the velocity component $u_2$ (resp. $u_1$) is a scalar passively advected by the flow. The solution of the first (resp. second) step is simply approximated using the algorithm 1 detailed in section~\ref{section: overall algorithm} supplemented with the update of the variable $\rho u_2$ (resp. $\rho u_1$) achieved thanks to an upwind scheme.

\section{{Anti-diffusive strategy for computing the color functions}} \label{section : fluxes z}

In this part we focus on the construction of the numerical fluxes $\widetilde{\cal Z}_{k,i+1/2}$ required during the projection step of the Lagrange-Remap scheme presented in the previous section. As in \cite{Billaud1,Kokh1}, we perform a specific discretization for this transport stage that involves an anti-dissipative algorithm for the advection of the $m$ color functions $\zk$. This transport algorithm has been proposed in \cite{Jaouen1}: it relies on a recursive construction for the flux of each color function $\zk$ and it guarantees that both relations \eqref{eq:Z max}-\eqref{eq:Z unity} are satisfied. Here, we use the same notation as the previous section, and focus on the case of one-dimensional problems. \\

The choice of the fluxes $\Lz_{,i+1/2}$ relies on the alternative
\begin{itemize}
 \item if $u^n_{i+1/2} > 0$ and $u^n_{i - 1/2} < 0$ (resp. $u^n_{i+3/2} < 0$ and $u^n_{i + 1/2} > 0$), then we set 
  $\Lz_{,i+1/2}$ to the upwind value of $\zk$, with respect to the sign of $u^n_{i+1/2}$,
 \item if $u^n_{i+1/2} > 0$ and $u^n_{i - 1/2} > 0$ (resp. $u^n_{i+3/2} < 0$ and $u^n_{i + 1/2} < 0$), then we choose for $\Lz_{,i+1/2}$ 
 the closest value to the downwind value of $\zk$, with respect to the sign of $u^n_{i+1/2}$ within a trust interval $I^n_{k,i+1/2}$ that provides important features to our numerical scheme.
\end{itemize}
In the next sections, we present a procedure to define a trust interval $I^n_{k,i+1/2}$ that ensures: 
\begin{enumerate}
\item consistency for  $\Lz_{,i+1/2}$, 
\item stability for $\zk$, 
\item  verifies relation~\eqref{eq:Z unity} for the discrete unknowns.
\end{enumerate}

\subsection{Consistency and stability criterions}\label{section: cons stab z}
First we exhibit a sufficient criterion ensuring a consistency constraint for $\Lz_{,i+1/2}$. Let us note 
$$
m^n_{k,i+1/2} = \min\left({\cal Z}_{k,i}^n,{\cal Z}_{k,i+1}^n\right) 
,\quad
M^n_{k,i+1/2} = \max \left({\cal Z}^n_{k,i},{\cal Z}_{k,i+1}^n\right)
.
$$
As in  \cite{Despres1,Despres2,Kokh1}, consistency for the flux  $\Lz_{,i+1/2}$ is imposed by requiring
\begin{equation}
\Lz_{,i+1/2} \in [m^{n}_{k,i+1/2},M^{n}_{k,i+1/2}].
% \quad
% \forall i \in \ZZ.
\label{eq:cons}
\end{equation}

We now seek a sufficient condition on $\widetilde{\cal Z}_{k,i+1/2}$ that guarantees the stability for the color function ${\cal Z}_k$ in a neighborhood cell of the interface $i+1/2$. We assume $u^n_{i+1/2} > 0$ and $u^n_{i- 1/2} > 0$ (resp. $u^n_{i+1/2} < 0$ and $u^n_{i+ 3/2} < 0$): as $\zk$ is evolved by a pure transport equation, stability within cell $i$ (resp. $i+1$) is ensured by { a Harten-Leroux criterion \cite{Harten1, Leroux1977}}
\begin{equation}
{\cal Z}_{k,i}^{n+1} \in [m^{n}_{k,i-1/2},M^{n}_{k,i-1/2}]
\quad
\text{ 
(resp.
${\cal Z}_{k,i+1}^{n+1} \in [m^{n}_{k,i+3/2},M^{n}_{k,i+3/2}]$
)}
\label{eq:stab}
\end{equation}
One should note this enforces a local maximum principle in the cell $i$ (resp. $i+1$). {From relation \eqref{eq:stab} and following same lines as in the proof of \cite [\S~3.3.3, proposition 3.1]{Kokh1}, we are able to deduce a sufficient condition on $\widetilde{\cal Z}_{k,i+1/2}$ for ensuring stability for $\zk$ when $u^n_{i+1/2} > 0$ and $u^n_{i- 1/2} > 0$ (resp. $u^n_{i+1/2} < 0$ and $u^n_{i+ 3/2} < 0$) that is resumed by the following proposition.
}
\begin{prop}
\label{prop:stability}
Assuming that both conditions for consistency \eqref{eq:cons} and CFL \eqref{eq:CFL} hold. 
% Then, for a given $i \in \ZZ$  a sufficient condition for local stability is ${\cal Z}_{i+1/2}^{n} \in [a^{n}_{i+1/2};A^{n}_{i+1/2}]$.
\begin{enumerate}[(a)]
\item If $u^n_{i+1/2} > 0$ and $u^n_{i-1/2}>0$, set
$$
\begin{array}{rcl}
a^{n}_{k,i+1/2} &=& {\cal Z}^n_{k,i}+(M^n_{k,i-1/2}-{\cal Z}^n_{k,i})\left(\dfrac{u_{i-1/2}^n}{u^n_{i+1/2}} 
-
 \dfrac{\Delta x}{\Delta t} \dfrac{1}{ u^n_{i+1/2}}\right)
%+M^n_{k,i-1/2},
\\
A^{n}_{k,i+1/2} &=& {\cal Z}^n_{k,i}+(m^n_{k,i-1/2}-{\cal Z}^n_{k,i})\left(\dfrac{u_{i-1/2}^n}{u^n_{i+1/2}} 
-
\dfrac{\Delta x}{\Delta t} \dfrac{1}{u^n_{i+1/2}}\right)
%+M^n_{k,i-1/2}
,
\end{array}
$$
then   
$$
\widetilde{\cal Z}_{k,i+1/2} \in [a^{n}_{i+1/2},A^{n}_{i+1/2}]
\Rightarrow 
{\cal Z}_{k,i}^{n+1} \in [m^{n}_{k,i+1/2},M^{n}_{k,i+1/2}] 
$$
$$
\Rightarrow \text{stability for $\z_k$ in the cell $i$,}
$$
\item If $u_{i+3/2}^n<0$ et $u_{i+1/2}^n<0$, set
$$
\begin{array}{rcl}
a^{n}_{k,i+1/2} &=&{\cal Z}^n_{k,i+1}+(M^n_{k,i+3/2}-{\cal Z}^n_{k,i+1})\left(\dfrac{u_{i+3/2}^n}{u^n_{i+1/2}} 
+ \dfrac{\Delta x}{\Delta t}\dfrac{1}{ u^n_{i+1/2}}\right)
%+M^n_{k,i+1/2}
\\
A^{n}_{k,i+1/2} &=&{\cal Z}^n_{k,i+1}+(m^n_{k,i+3/2}-{\cal Z}^n_{k,i+1})\left(\dfrac{u_{i+3/2}^n}{u^n_{i+1/2}} 
+ \dfrac{\Delta x}{\Delta t}\dfrac{1}{u^n_{i+1/2}}\right)
%+M^n_{k,i+1/2}
,
\end{array}
$$
then  
%  ${\cal Z}_{i}^{n+1} \in [m^{n}_{k,i+1/2};M^{n}_{k,i+1/2}] \Rightarrow $ stability in the cell $i+1$.
$$
\widetilde{\cal Z}_{k,i+1/2} \in [a^{n}_{i+1/2},A^{n}_{i+1/2}]
\Rightarrow 
{\cal Z}_{k,i+1}^{n+1} \in [m^{n}_{k,i+3/2},M^{n}_{k,i+3/2}] 
$$
$$
\Rightarrow \text{stability for $\z_k$ in the cell $i+1$.}
$$
\end{enumerate}
\end{prop}

Combining both criterion \eqref{eq:cons} and bounds given in proposition \ref{prop:stability} we can state as in \cite{Billaud1,Despres2,Despres1,Kokh1} the following property:
{for} $u^n_{i+1/2} > 0$ and $u^n_{i- 1/2} > 0$ (or. $u^n_{i+1/2} < 0$ and $u^n_{i + 3/2} < 0$) both consistency and stability are satisfied if  
\begin{equation}
\widetilde{\cal Z}_{k,i+1/2}\in[\omega^{n}_{k,i+1/2},\Omega^{n}_{k,i+1/2}] 
\label{eq: Z const stab}
\end{equation}
where 
\begin{equation}
[\omega^{n}_{k,i+1/2},\Omega^{n}_{k,i+1/2}] = [m^{n}_{k,i+1/2},M^{n}_{k,i+1/2}] \cap [a^{n}_{k,i+1/2},A^{n}_{k,i+1/2}]
\label{eq:om}
\end{equation}

\subsection{Recursive construction of the trust interval $I^n_{k,i+1/2}$}\label{section: recurs trust interval}
A first simple choice for constructing the trust interval is to take $I^n_{k,i+1/2}= [\omega^{n}_{k,i+1/2},\Omega^{n}_{k,i+1/2}] $. Although this choice provides stability and consistency for $\zk$, it may unfortunately fail to comply with \eqref{eq:Z unity}. This flaw can indeed be revealed with a pure transport test of an arbitrary number of materials (see \cite{Jaouen1}). To overcome this difficulty, a solution should consists in computing only the first $m-1$ color functions and deduce the last one from the unity constraint~\eqref{eq:Z unity}. In that case, this is the constraint \eqref{eq:Z max} that is violated (see again \cite{Jaouen1}). \\

In consequence, we propose to construct a sequence of intervals $[d^n_{k,i+1/2},D^n_{k,i+1/2}] \subset [\omega^{n}_{i+1/2},\Omega^{n}_{i+1/2}], k=1, \dots, m$  thanks to a recursive process like \cite{Jaouen1} as follows.
\begin{enumerate}[(a)]
\item For $k=1$, set
$$
 d^n_{1,i+1/2}  = \max\left({\omega^n_{1,i+1/2}},{ 1 - \sum_{l=2}^{m-1} \Omega^n_{l,i+1/2} }\right),
 $$
 $$
D^n_{1,i+1/2}  =  \min\left(\Omega^n_{1,i+1/2},1 - \sum_{l=2}^{m-1} \omega^n_{l,i+1/2} \right).
$$
\item For $k=2,\ldots,m-1$, suppose that $\widetilde{\cal Z}_{l,i+1/2}$ are already known for $l\leq k$ and set
$$
d^n_{k,i+1/2}  = 
\max\!\!\left({\omega^n_{k,i+1/2}}, { 1 \!-\! \sum_{l=1}^{k-1} \!\widetilde{\cal Z}^n_{l,i+1/2} \!-\!\!\!\! \sum_{l=k+1}^{m-1} \!\!\!\Omega^n_{l,i+1/2} } \right),
$$
$$
D^n_{k,i+1/2}  = 
\min\!\!\left({\Omega^n_{k,i+1/2}},{  1 \!-\! \sum_{l=1}^{k-1}\!\widetilde{\cal Z}^n_{l,i+1/2} \!-\!\!\!\! \sum_{l=k+1}^{m-1} \!\!\!\omega^n_{l,i+1/2}}\right).
$$
\item For $k=m$, suppose that $\widetilde{\cal Z}_{l,i+1/2}$ are already known for 
$l< m$ and set for the last flux
$$
d^n_{k,i+1/2}  = 
\max\!\!\left({\omega^n_{k,i+1/2}},{ 1 \!-\! \sum_{l=1}^{m-1} \!\widetilde{\cal Z}_{l,i+1/2} } \right),
$$
$$
D^n_{k,i+1/2}  = 
\min\!\!\left({\Omega^n_{k,i+1/2}},{  1 \!-\! \sum_{l=1}^{m-1}\! \widetilde{\cal Z}_{l,i+1/2} \!} \right).
$$
\end{enumerate}
As pointed out in \cite{Jaouen1}, we are ensured that $I_{k,i+1/2}^n \neq \emptyset$ since the fluxes are chosen in their admissibility interval $[\omega^{n}_{i+1/2},\Omega^{n}_{i+1/2}]$. Moreover $[\omega^{n}_{i+1/2},\Omega^{n}_{i+1/2}]\neq \emptyset$, since the upwind choice for $\widetilde{\cal Z}_{k,i+1/2}^{n}$ belongs to $[\omega^{n}_{i+1/2},\Omega^{n}_{i+1/2}]$.\\

Then, the following result proved in \cite{Jaouen1} can be stated.
\begin{thm}
\label{theorem 1}
Suppose that the CFL constraint~\eqref{eq:CFL} and ${\cal Z}^n_{k,i} \in [0,1]$, $\sum_{k=1}^m {\cal Z}_{k,i}^n = 1$ are satisfied. If the fluxes are chosen such that $ \widetilde{\cal Z}_{k,i+1/2} \in I^n_{k,i+1/2}=[d^n_{k,i+1/2},D^n_{k,i+1/2}]$ then for all $k \in \{1,\dots,m\}$ and  for all $i \in \ZZ$
\begin{enumerate}
\item $ {\cal Z}^{n+1}_{k,i} \in [0,1],$
\item $\displaystyle \sum_{k=1}^m {\cal Z}^{n+1}_{k,i}= 1$.
\end{enumerate}
\end{thm}
In that way, we have defined a trust interval $I^n_{k,i+1/2}$
such that the numerical approximation of $\zk$ provides the three following features: stability, consistency and a discrete equivalent of \eqref{eq:Z unity} if 
\begin{equation}
\widetilde{\cal Z}_{k,i+1/2}\in I^n_{k,i+1/2}= [d^n_{k,i+1/2},D^n_{k,i+1/2}]
\label{eq : trust int}
\end{equation}
when $u^n_{i+1/2} > 0$ and $u^n_{i- 1/2} > 0$ (or $u^n_{i+1/2} < 0$ and $u^n_{i + 3/2} < 0$).

{
\begin{rem}
Here, the sequence of intervals $ [d^n_{k,i+1/2},D^n_{k,i+1/2}]$ are computed in the order of the increasing index $k$. Of course, this is not the only possible choice and the intervals construction order does not significantly impact the numerical result. This point will be illustrated in section~\ref{section: 1D transport}. We also refer to \cite[\S~2]{Jaouen1} in the context of passive scalar convection. 
\end{rem}
}

\subsection{Choice of $\Lz_{,i+1/2}$}\label{section: flux z choice}

We detail here the definition of $\Lz_{,i+1/2}$ that was drafted  at the beginning of section~\ref{section : fluxes z} using similar lines as \cite{Billaud1,Despres2,Despres1,Jaouen1,Kokh1}.\\ 
First, when $u^n_{i+1/2} > 0$ and
 $u^n_{i- 1/2} > 0$ (or $u^n_{i+1/2} < 0$ and ${u^n_{i+ 3/2}} < 0$), we construct the trust interval $I^n_{k,i+1/2}$ given by \eqref{eq : trust int}.\\ 
Then $\z_{k,i+1/2}$ is set to the closest value to the downwind value ${\cal Z}^{do}_{k,i+1/2}$ (according to sign of $u^n_{i+1/2}$) within $I^n_{k,i+1/2}$ 
 (see figure \ref{fig:1}) in order to limit the numerical diffusion as in \cite{Despres2, Despres1, Jaouen1, Kokh1}. Namely
\begin{equation}
 \widetilde{\cal Z}_{k,i+1/2} = \min_{Z\in I^n_{k,i+1/2}} |{\cal Z}^{do}_{k,i+1/2}-Z|
\label{eq: anti flux}
\end{equation}
with $ 
\displaystyle 
{\cal Z}^{do}_{k,i+1/2} = 
 \left\{
\begin{array}{cl}
 {\cal Z}_{k,i+1}^{n} & \text{ if } {u}_{i+1/2} >0, \\
 {\cal Z}_{k,i}^{n} & \text{ otherwise. }
 \end{array}
\right.
$
\begin{rem}
The trust interval $I_{k,i+1/2}^n$ is clearly defined for only two cases $(u_{i+1/2} > 0$ and $u_{i-1/2}>0)$ or $(u_{i+1/2} < 0$ and $u_{i+3/2}<0)$. In other cases, the upwind value is selected as default value for $  \widetilde{\cal Z}_{k,i+1/2}$. 
\end{rem}

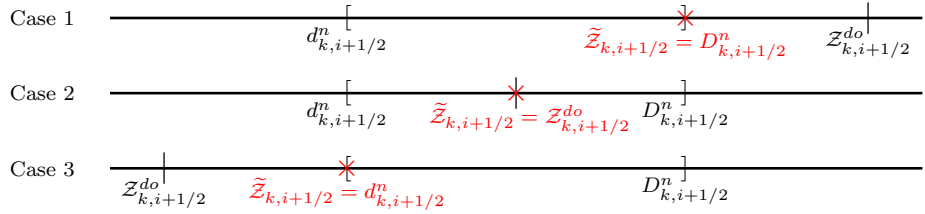
\begin{figure}[ht!]
\centering
\begin{tikzpicture}[xscale=1.8,yscale=0.5]
   \tikzstyle{every node}=[font=\footnotesize]
   \node[] at (-1,4) {Case~1};
   \node[] at (-1,2) {Case~2};
   \node[] at (-1,0) {Case~3};
   \draw[line width=1pt] (-0.5,4)   to[->] (5.5,4);
   \draw[line width=1pt] (-0.5,2)   to[->] (5.5,2);
   \draw[line width=1pt] (-0.5,0)   to[->] (5.5,0);
   \node[] at (1.25,0) {$\big[$};
   \node[below] at (1.25,0) {\textcolor{red}{$\widetilde{\cal Z}_{k,i+1/2}=d_{k,i+1/2}^n$}};
   \node[] at (1.25,0) {\textcolor{red}{\large $\times$}};
   \node[] at (3.75,0) {$\big]$};
   \node[below] at (3.75,0) {$D_{k,i+1/2}^n$};
   \node[] at (1.25,2) {$\big[$};
   \node[below] at (1.25,2) {$d_{k,i+1/2}^n$};
   \node[] at (3.75,2) {$\big]$};
   \node[below] at (3.75,2) {$D_{k,i+1/2}^n$};
   \node[] at (1.25,4) {$\big[$};
   \node[below] at (1.25,4) {$d_{k,i+1/2}^n$};
   \node[] at (3.75,4) {$\big]$};
   \node[below] at (3.75,4) {\textcolor{red}{$\widetilde{\cal Z}_{k,i+1/2}=D_{k,i+1/2}^n$}};
   \node[] at (3.75,4) {\textcolor{red}{\large $\times$}};
   \node[] at (-0.1,0) {\large $|$};
   \node[below] at (-0.1,0) { ${\cal Z}_{k,i+1/2}^{do}$};
   \node[] at (2.5,2) {\large $|$};
   \node[] at (2.5,2) {\textcolor{red}{\large $\times$}};
   \node[below] at (2.6,2) {\textcolor{red}{$\widetilde{\cal Z}_{k,i+1/2}={\cal Z}_{k,i+1/2}^{do}$}};
   \node[] at (5.1,4) {\large $|$};
   \node[below] at (5.1,4) { ${\cal Z}_{k,i+1/2}^{do}$};
   % \foreach \xp in {1.45,1.65,...,3.75}{\node[] at (\xp,0) {/};}
\end{tikzpicture} 
\caption{Choice of the flux $ \widetilde{\cal Z}_{i+1/2}$, three possible cases:  
1) ${\cal Z}^{do}_{k,i+1/2}>D_{k,i+1/2}^n$, 2) $d_{k,i+1/2}^n \leq {\cal Z}^{do}_{k,i+1/2}\leq D_{k,i+1/2}^n$, 3) ${\cal Z}^{do}_{k,i+1/2} < d_{k,i+1/2}^n$.
\label{fig:1}}
\end{figure}

%\subsection{Illustration on scalar passive transport} 

\section{Numerical results}\label{section: simulations}
{In this section, we present numerical results with 
one-dimensional and two-dimensional tests cases for the 
$m$-component interface flow model~\eqref{eq:syst1}. We will compare simulations performed with the proposed Lagrange-Remap  scheme  (see algorithm \ref{alg:procedure1}) with anti-diffusive \eqref{eq: anti flux} and upwind \eqref{eq: upwind flux} construction of the color function fluxes during the Remap step.
}
%%%%%%%%%%%%%% RESNUM 

\subsection{Test 1: One-dimensional five-material passive transport}\label{section: 1D transport}
We consider a one-dimensional periodic domain of length $1\,\mathrm{m}$. At the initial instant six interfaces located at the positions 
%$X_1=0<X_2<\dots<X_5<X_6=1$ 
$$
X_1 = 0,\quad X_2 = 0.1,\quad X_3 =  0.25,\quad X_4 =  0.7,\quad X_5 =  0.9,\quad X_6 =  1,
$$
separate five constant pure fluid states $k=1,\cdots,5$. Due to periodicity, the discontinuity located at $x=X_6$, matches the discontinuity at $x=X_1$. Each region $k=1,\dots,5$ delimited by $X_k<x<X_{k+1}$ is occupied by a different material. 
The fluid in region $k\in\{1,5\}$ is a perfect gas whose EOS is given by \eqref{eq: eos pg}, while region $k\in\{2,4\}$ contains a stiffened gas whose EOS is \eqref{eq: eos sg}. A Van der Waals
fluid that is governed by the EOS 
$$
p_k = \left( \dfrac{\gamma_k -1}{1-b_k \rho_k}\right)(\rho_ke_k+a_k\rho^2_k)-a_k\rho^2_k.
$$
fills region $k=3$. The initial value of both pressure $p$ and velocity $u$ are uniform and set to 
$p= 10^5\,\mathrm{Pa}$ and $u=100\,\mathrm{m}.\mathrm{s}^{-1}$.\\
{The EOS parameters used for each component are displayed in table~\ref{tab: 1D advec test EOS parameters}.
%
%  Note that for these values, the hyperbolicity of the $m$-component system is ensured, especially for the Van der Waals gas for which $c_k^2$ and $c^2$ remain positive.  
Let us underline that these initial data did not produce any phasic value for the Van der Waals fluid out of the region where the is hyperbolic. More precisely $c_k^2$ for $k=1,\ldots,5$ and $c^2$ remained positive values whenever they need to be evaluated by the algorithm throughout the whole computation.} The initial density within each regions are given in table~\ref{tab: 1D advec init data}. The computational domain is discretized over a 100-cell grid and the simulation is performed with a CFL coefficient $C_{\rm CFL} = 0.9$.
\begin{table}[]
\centering
\caption{test 1, one-dimensional five-material passive transport. EOS parameters.} 
\begin{tabular}{ccccc}
\hline\hline
material & $\gamma_k$ & $\pi_k$ \scriptsize{($\mathrm{Pa}$)}& $a_k$ \scriptsize{($\mathrm{Pa}.\mathrm{m}^6/\mathrm{kg}^2$)} & $b_k$ \scriptsize{($\mathrm{m}^{3}/\mathrm{kg}$)} 
\\
$k=1$     & $1.6$ & $\times$                 &$\times$&$\times$
\\ 
$k=2$  & $4.4$ & $6.10^8$  &$\times$&$\times$
\\ 
$k=3$   & $1.4$ & $\times$                 &$5$        &$1.10^{-3}$
\\ 
$k=4$ & $2.4$ & $2.10^8$  &$\times$&$\times$
\\
$k=5$ & $1.6$ & $\times$                 &$\times$&$\times$
\\
 \hline\hline
\end{tabular}
\label{tab: 1D advec test EOS parameters}
\end{table}
\begin{table}[]
\caption{test 1, one-dimensional five-material passive transport. Initial location of each material and initial density values.} 
\centering
\begin{tabular}{ccc}
\hline\hline
location & $k$ & $\rho_k$ \scriptsize{$(\mathrm{kg}.\mathrm{m}^{-3})$}
% \\
%  &  & \footnotesize $(\mathrm{kg}.\mathrm{m}^{-3})$ &  & \footnotesize $(\mathrm{Pa})$ & \footnotesize  $a_k$ & \footnotesize  $(\mathrm{m}^3/\mathrm{kg})$ 
\\
\hline \hline
$X_1=0\le x< X_2$  & $1$ & $50.0$
\\ 
$X_2 \le x< X_3$  & $2$ & $1000.0$
\\ 
$X_3 \le x<X_4$  & $3$ & $500.0$ 
\\ 
$X_4 \le x <X_5$  & $4$ & $1200.0$
\\
$X_5 \le x < X_6=1$  & $5$ & $150.0$
\\
\hline\hline
\end{tabular}
 \label{tab: 1D advec init data}
\end{table}

{Figures~\ref{fig: 1D transport z}, \ref{fig: 1D transport y} and \ref{fig: 1D transport density} respectively display the variable $\z_k$, $\y_k$, $k=1,\dots,5$ and $\rho$ for two instants $t_{end}=0.01\,\mathrm{s}$ corresponding to one turn through the initial position, and $t_{end}=1.5\,\mathrm{s}$ after 150 turns. We can see that the upwind scheme performs very poorly from the first moments  as the approximate solution has been flatten out by numerical diffusion. Despite the use of a coarse mesh, the solution obtained with the anti-diffusive scheme remains sharp and shows a very good agreement with the exact solution, especially for both color functions and mass fractions. To quantify the evolution of the numerical diffusion of the color functions during the computation, we count the number of cells where the variables $\z_k$ are numerically diffused. For a given time $t^n$, a cell $i$ is considered to be a diffusion cell for the material $k$ if  $ \varepsilon \le \z_{k,i} \le 1- \varepsilon$, where $\varepsilon=10^{-6}$. The percent of diffusion cells for both anti-diffusive and upwind scheme thorough the computation is represented in figure \ref{fig: INT diff cells }. We observe that the anti-diffusive scheme performs computation with at most $2\%$ of diffusion cells contrary to the upwind scheme that completely diffuses the profile of the color functions for $t=0.01 \mathrm{s}$.}\\
In addition,we note that the maximum principle $0 \leq \z_k \leq 1$, $k=1,\dots,5$ and the discrete equivalent of \eqref{eq:Z unity}  are verified throughout the computation for $\zk$ as expected. We can observe that the mass fractions $\yk$ comply with $\yk \in [0,1]$ and $\sum_k \yk=1$. Finally, we also verify that the pressure and velocity profiles initially constant are preserved during the computation for each scheme as depicted in figure~\ref{fig: 1D transport pressure velocity} as for the five-equation two-component model (see \cite{Kokh1}).  A proof of this property is given in \ref{section: iso-p iso-u profiles}.\\

{
At last, we test the influence of the material ordering (then the order of the color function fluxes construction) on the numerical solution computed with the anti-diffusive scheme. To this end, we propose to solve again the one-dimensional five-material passive transport test with the anti-diffusive scheme using the following material numbering permutation
$$
\text{ fluid 1} \to \text{ fluid 2},~ 
\text{ fluid 2} \to \text{ fluid 1},~
\text{ fluid 3} \to \text{ fluid 5},
$$
$$
\text{ fluid 4} \to \text{ fluid 3},~
\text{ fluid 5} \to \text{ fluid 4}. 
$$
This permutation can be summarized thanks to the application $\sigma$ over the set of indices $k$ defined such that $ \sigma (1)=2, \sigma (2)=1, \sigma (3)=5, \sigma (4)=3, \sigma (5)=4$. In order to compare the numerical solutions obtained with the anti-diffusive scheme for both numbering of the materials, we have measured the maximal relative difference over the time interval $[0,t_{end}]$ \textit{i.e.} 
$$
 e_1(a)=\max_{t \in [0,t_{end}]}
\max_{x \in [0,1]} \left | \dfrac{a(x,t)- \hat{a}(x,t)}{a(x,t) +\hat{a}(x,t)}  \right |
$$ 
for $a \in \{\rho,p,u\}$ and the absolute difference 
$$\displaystyle e_2(a_k)=\max_{t \in [0,t_{end}]} \max_{x \in [0,1]} \left |a_k(x,t)-\hat{a}_{\sigma(k)}(x,t) \right | $$ for $a_k \in \{\z_k,\y_k\}$. Here, the quantity $a$ computed for the test with modified material order is noted with $\widehat{a}$. The tables \ref{tab: MAT e1}-\ref{tab: MAT e2} summarizes the  $e_1$ and $e_2$ values obtained for the anti-diffusive scheme with $t_{end}=0.01\mathrm{s}$. As we can see, the order of material influence seems to be negligible since both the errors $e_1$ and $e_2$ are at most of magnitude $10^{-11}$. 
}
 
\newlength{\picwidth}
\setlength{\picwidth}{.5\textwidth}

\begin{figure}
\centering
 \begin{tabular}{cc}
 \includegraphics[width=\picwidth, clip, trim =
 12mm 8mm 0mm 0mm]{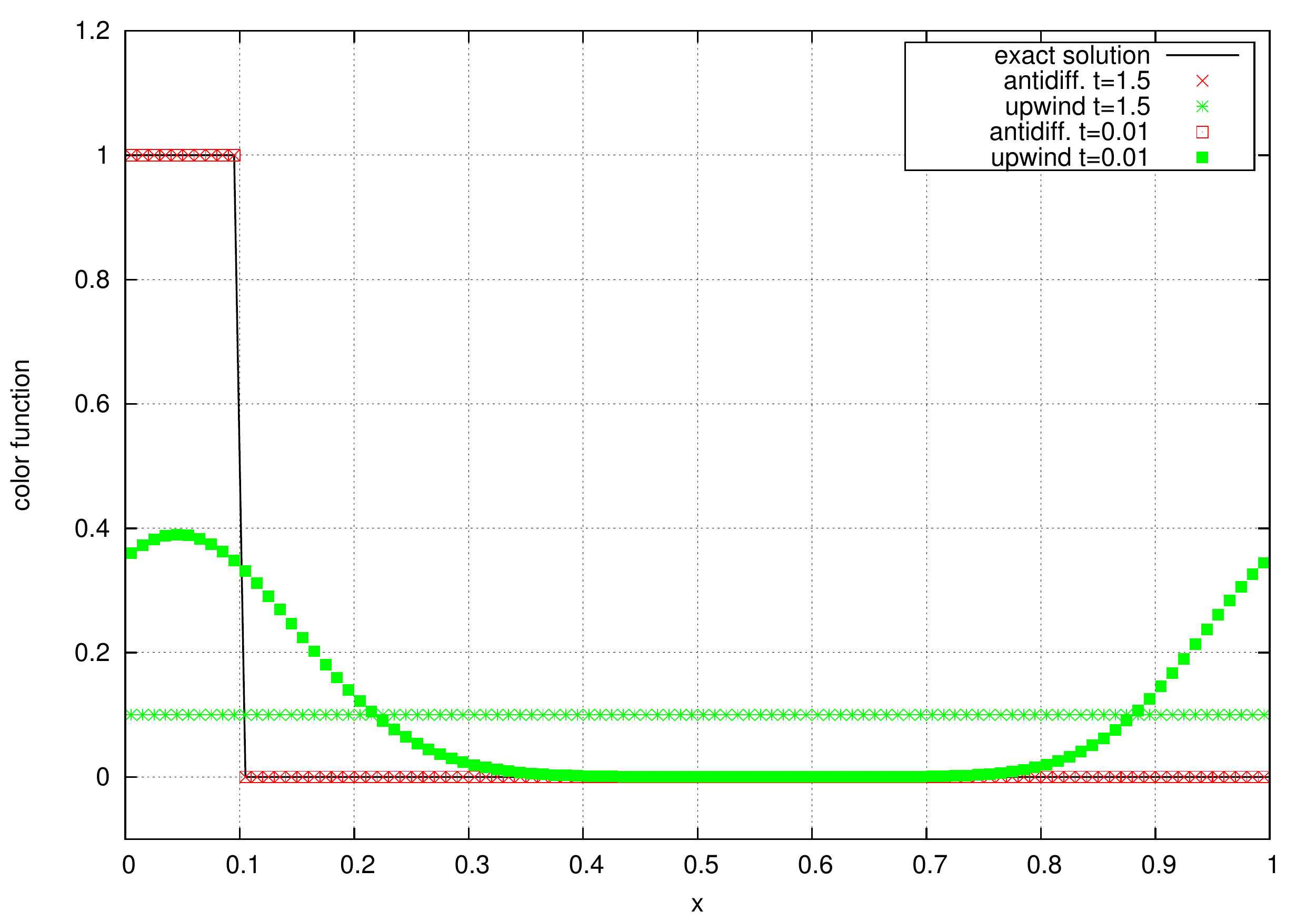} 
&
 \includegraphics[width=\picwidth, clip, trim =
 12mm 8mm 0mm 0mm]{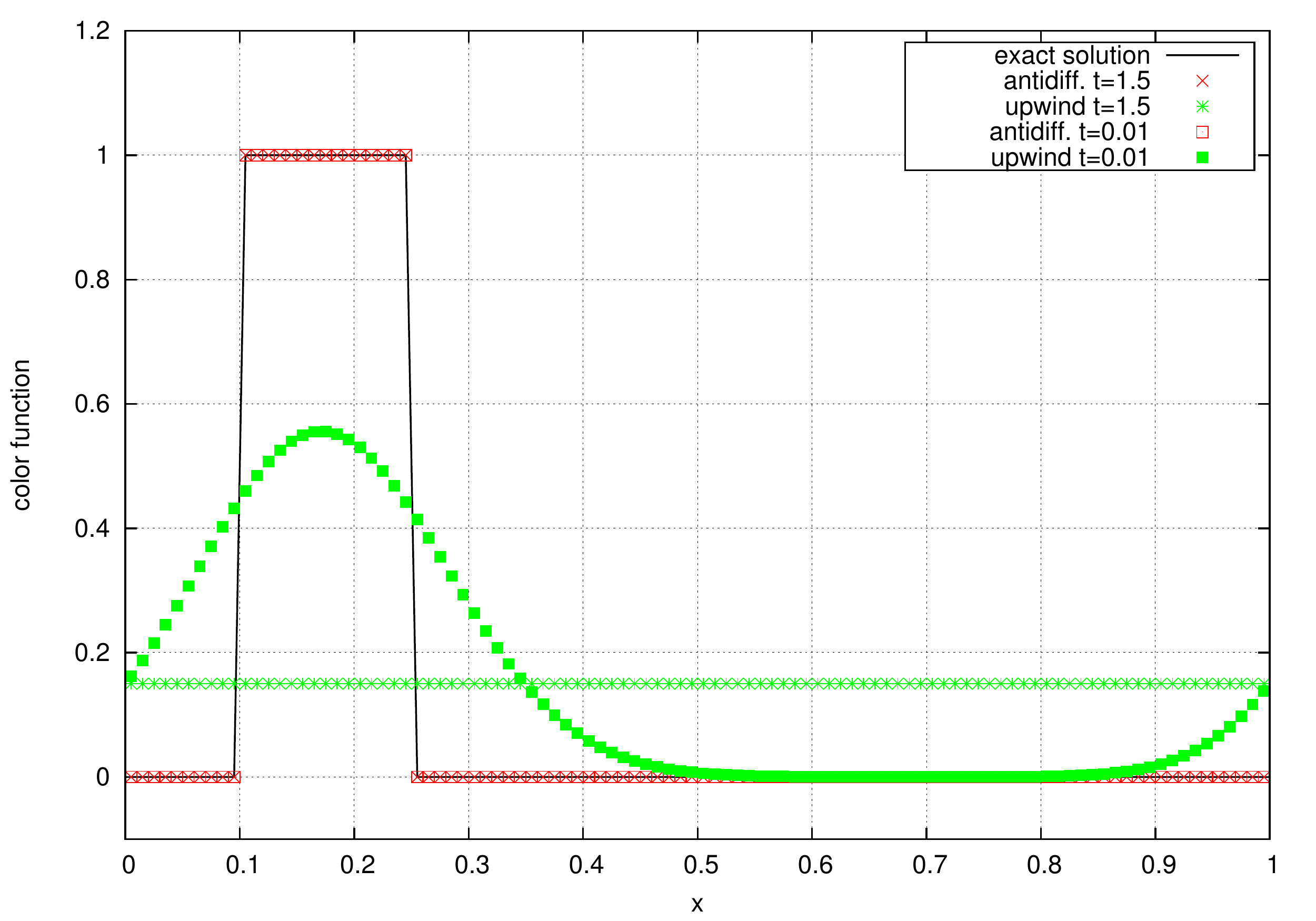} 
\\
Profile of $\z_1$.
&
Profile of $\z_2$.
\\
 \includegraphics[width=\picwidth, clip, trim =
 12mm 8mm 0mm 0mm]{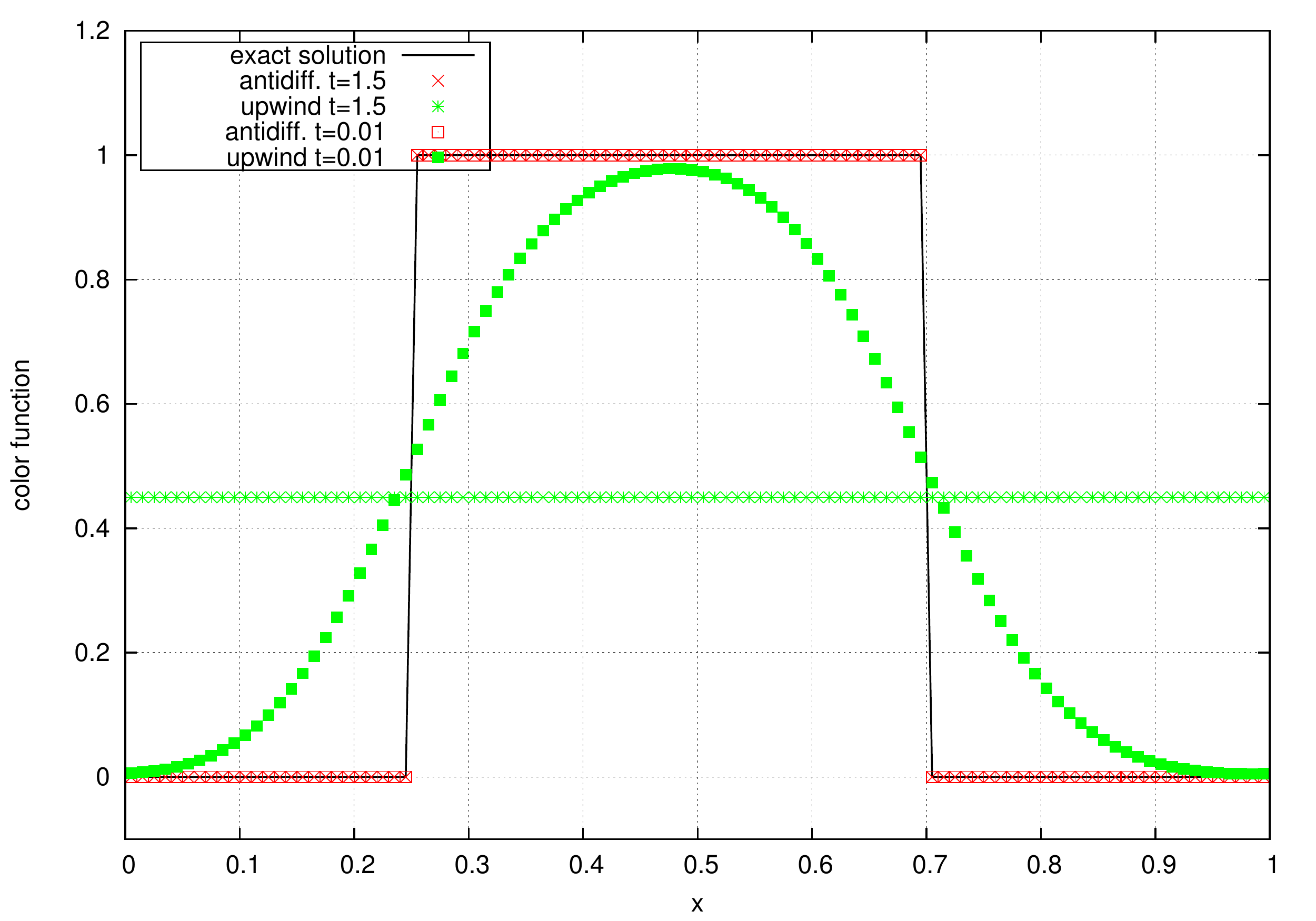} 
&
 \includegraphics[width=\picwidth, clip, trim =
 12mm 8mm 0mm 0mm]{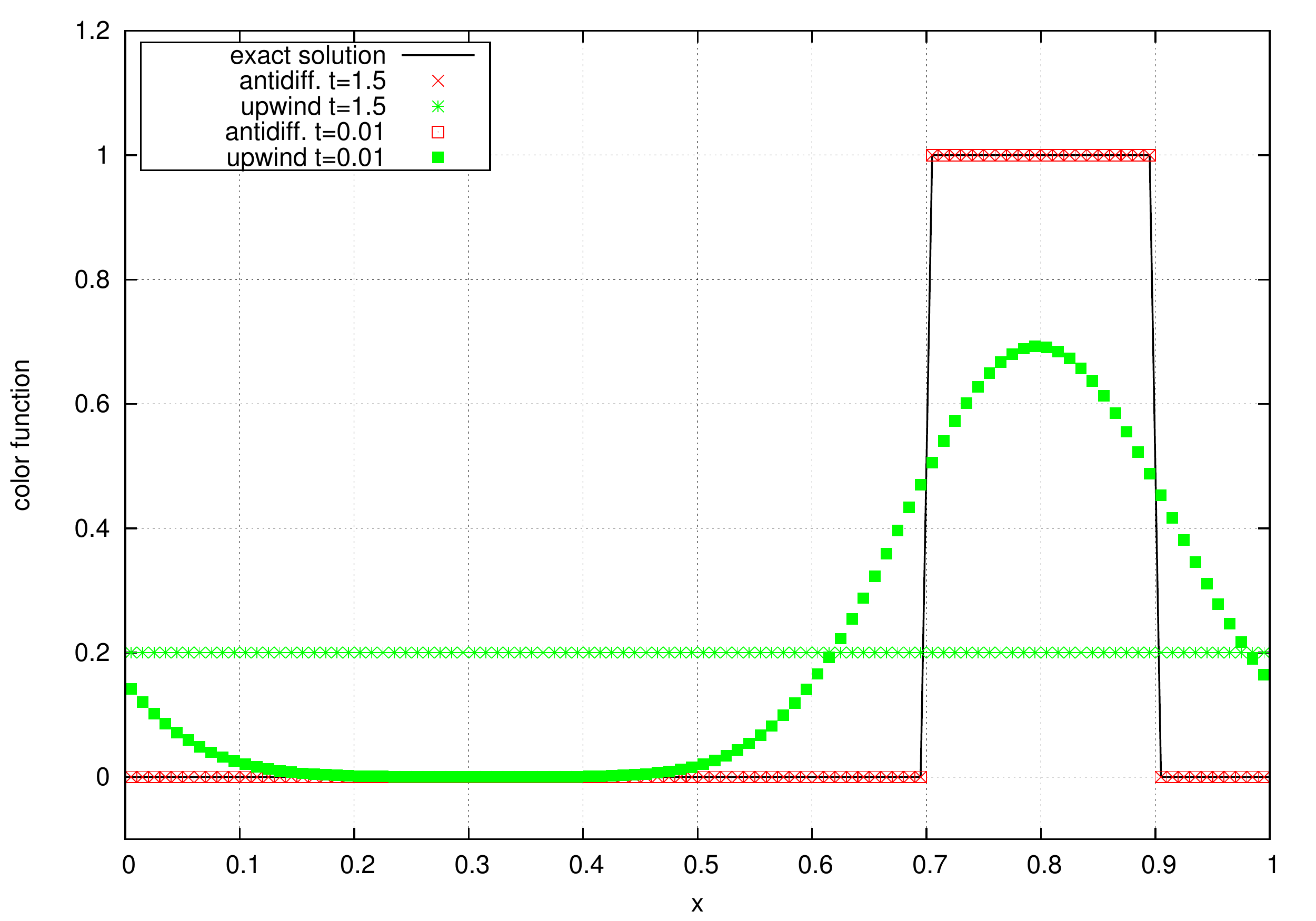} 
\\
Profile of $\z_3$.
&
Profile of $\z_4$.
\\
  \includegraphics[width=\picwidth, clip, trim =
 12mm 8mm 0mm 0mm]{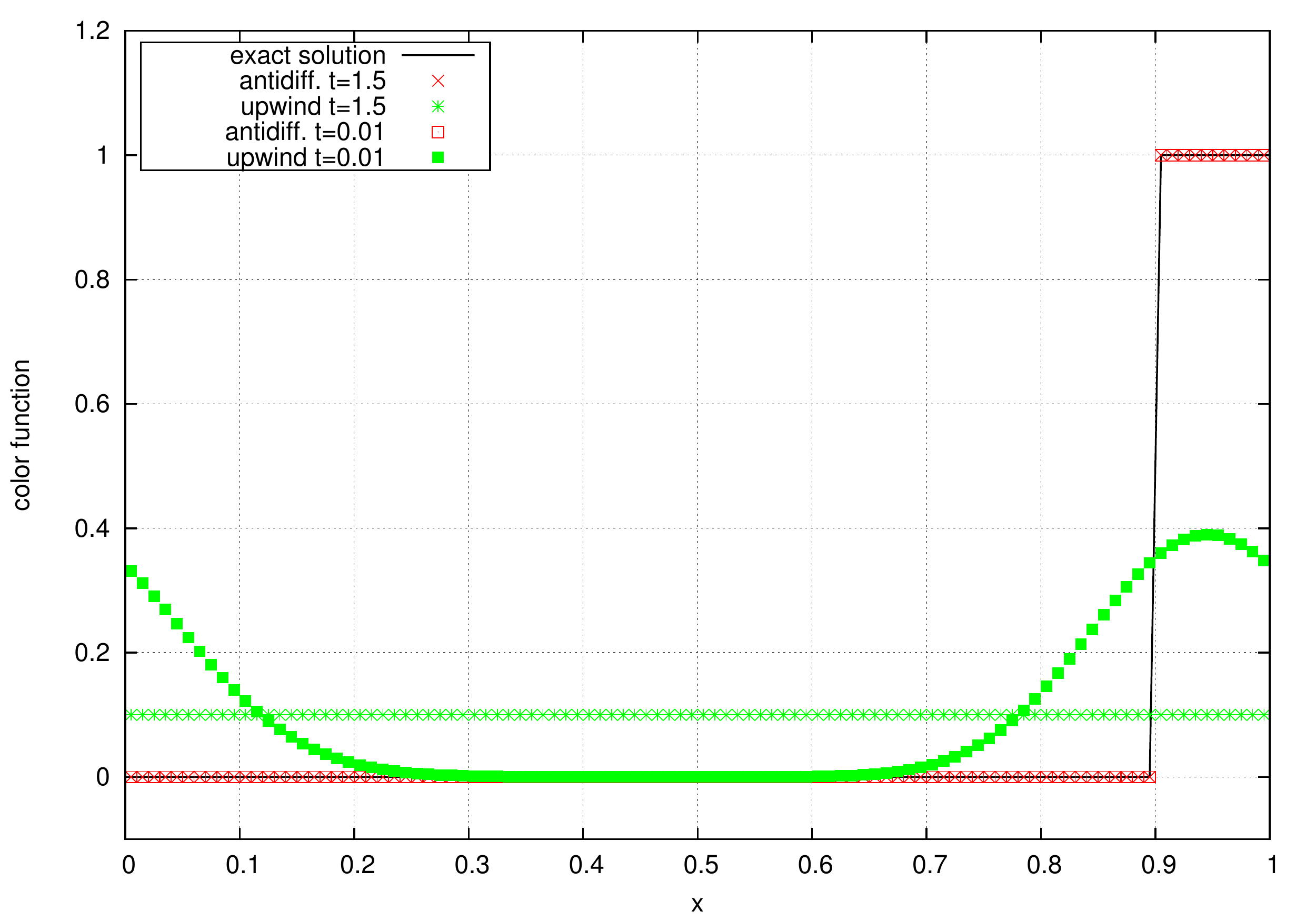}
  &
    \includegraphics[width=\picwidth, clip, trim =
 12mm 8mm 0mm 0mm]{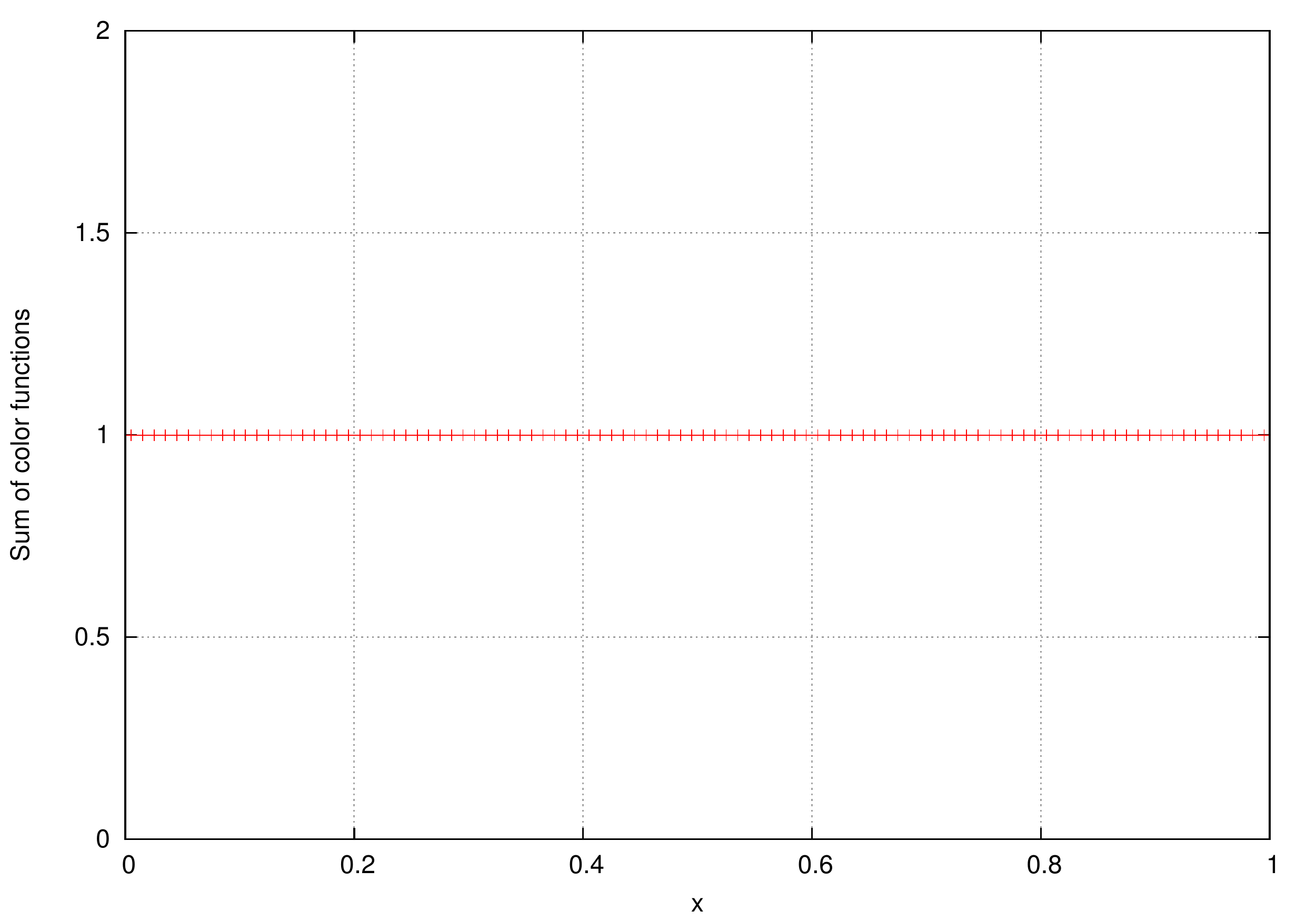} 
\\
Profile of $\z_5$ & $\sum_{k=1}^5 \zk$
 \end{tabular}
\caption{test 1, one-dimensional five-material passive transport. Profile of the color functions $\zk$ and of $\sum_{k=1}^5 \zk$ for $t\in \{0.01\,\mathrm{s},1.5\,\mathrm{s}\}$. Comparison between the upwind scheme and the anti-diffusive solver.}
 \label{fig: 1D transport z}
\end{figure}

\begin{figure}
\centering
 \begin{tabular}{cc}
 \includegraphics[width=\picwidth, clip, trim =
 12mm 8mm 0mm 0mm]{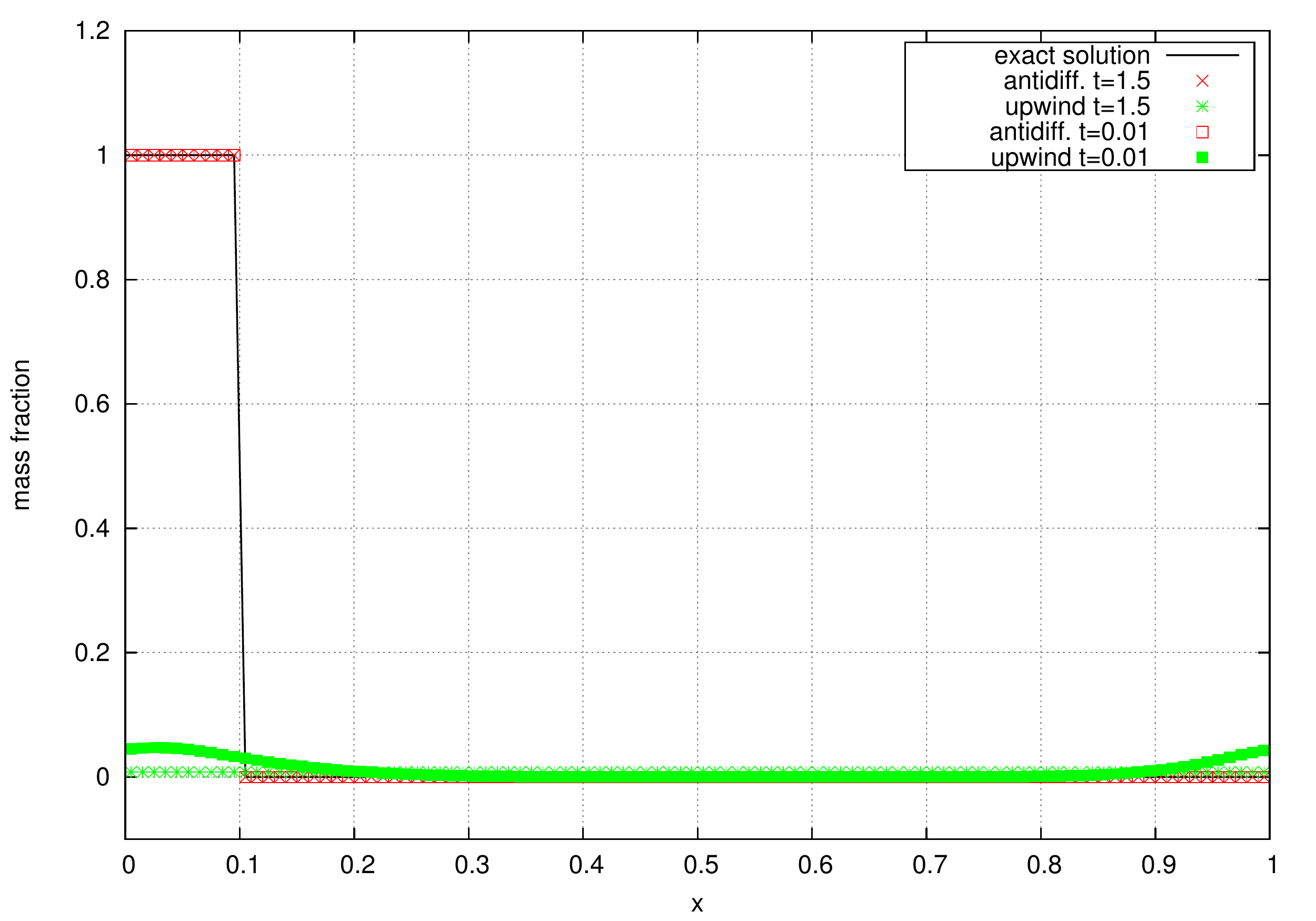} 
&
 \includegraphics[width=\picwidth, clip, trim =
 12mm 8mm 0mm 0mm]{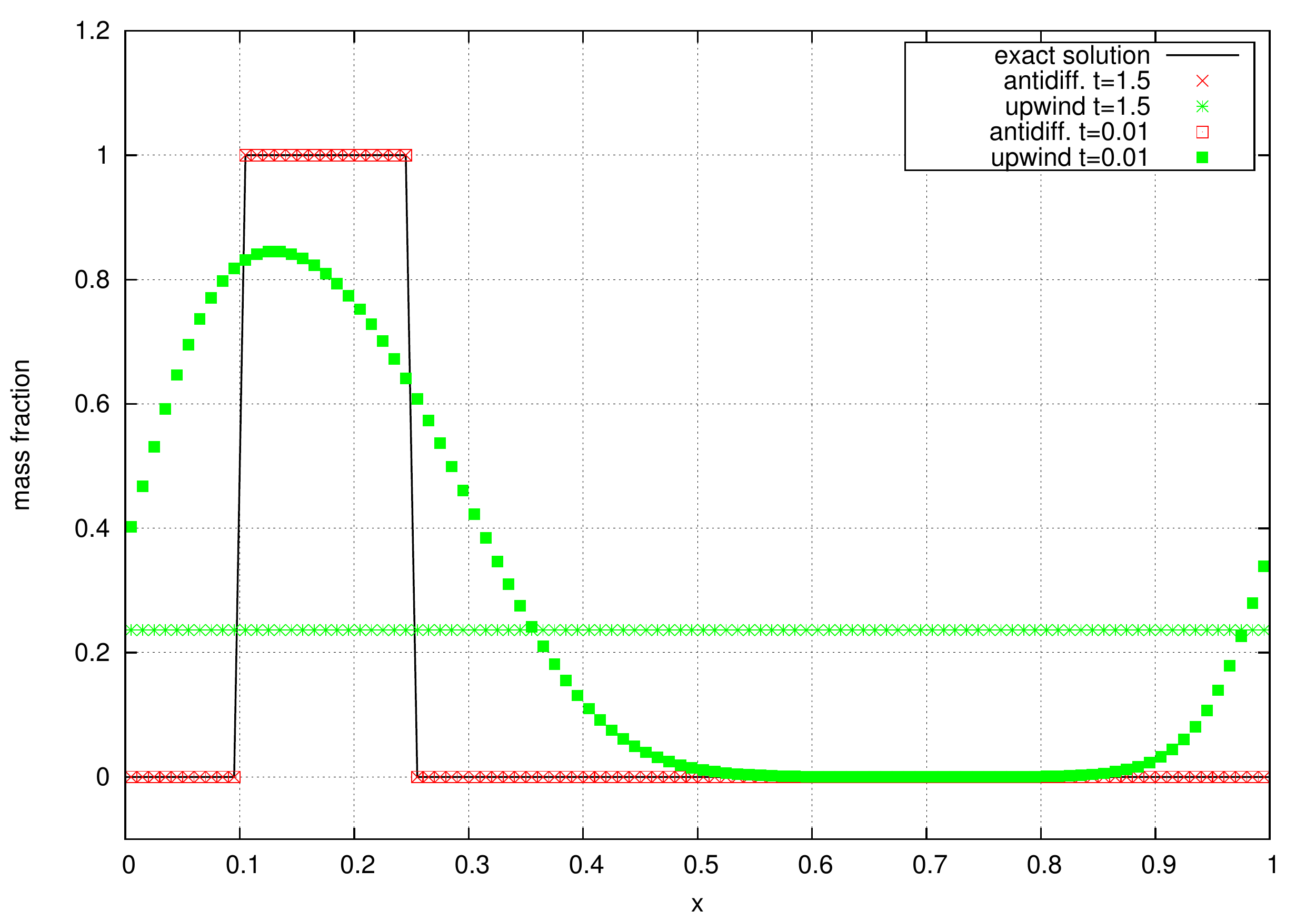} 
\\
Profile of $\y_1$.
&
Profile of $\y_2$.
\\
 \includegraphics[width=\picwidth, clip, trim =
 12mm 8mm 0mm 0mm]{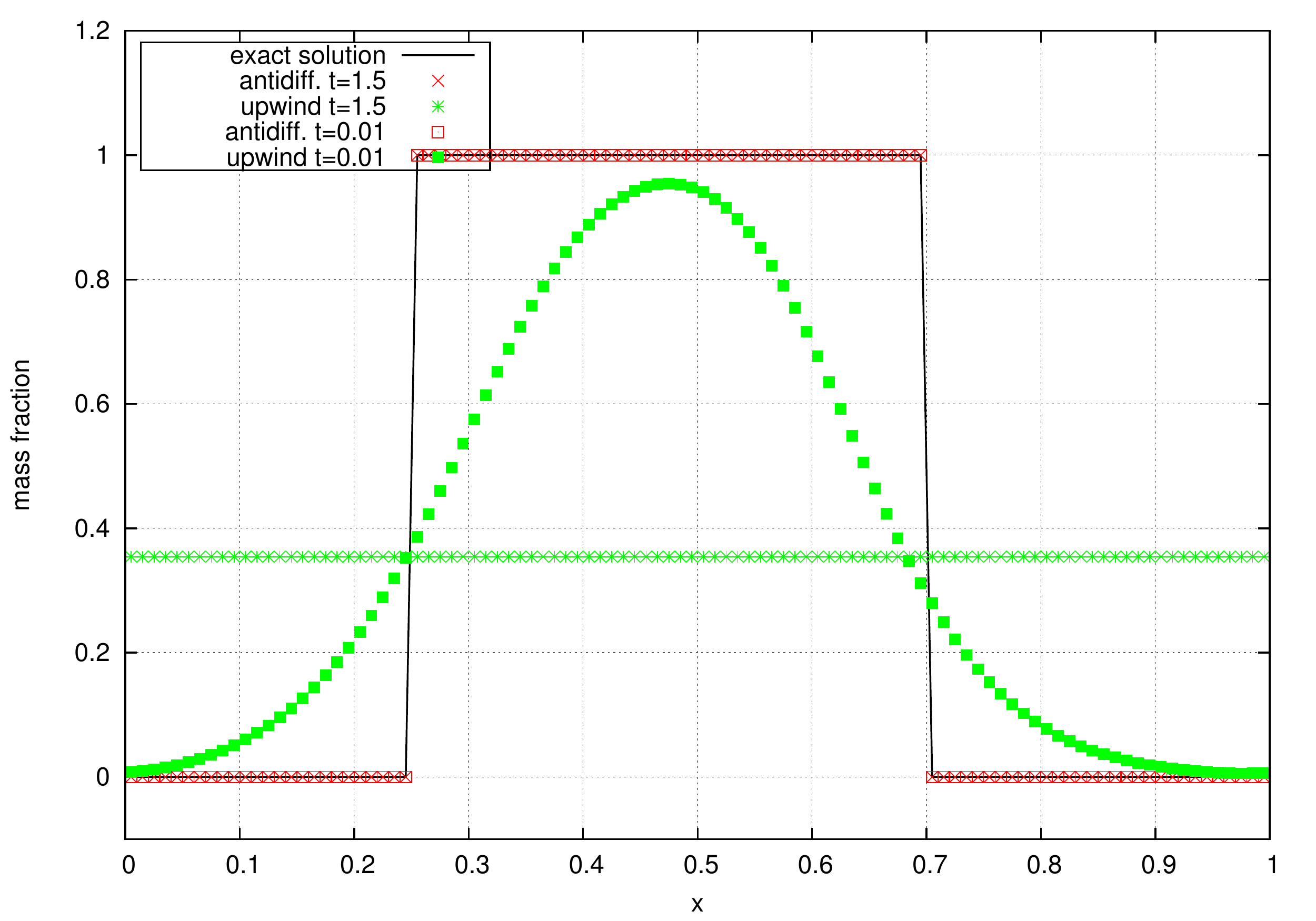} 
&
 \includegraphics[width=\picwidth, clip, trim =
 12mm 8mm 0mm 0mm]{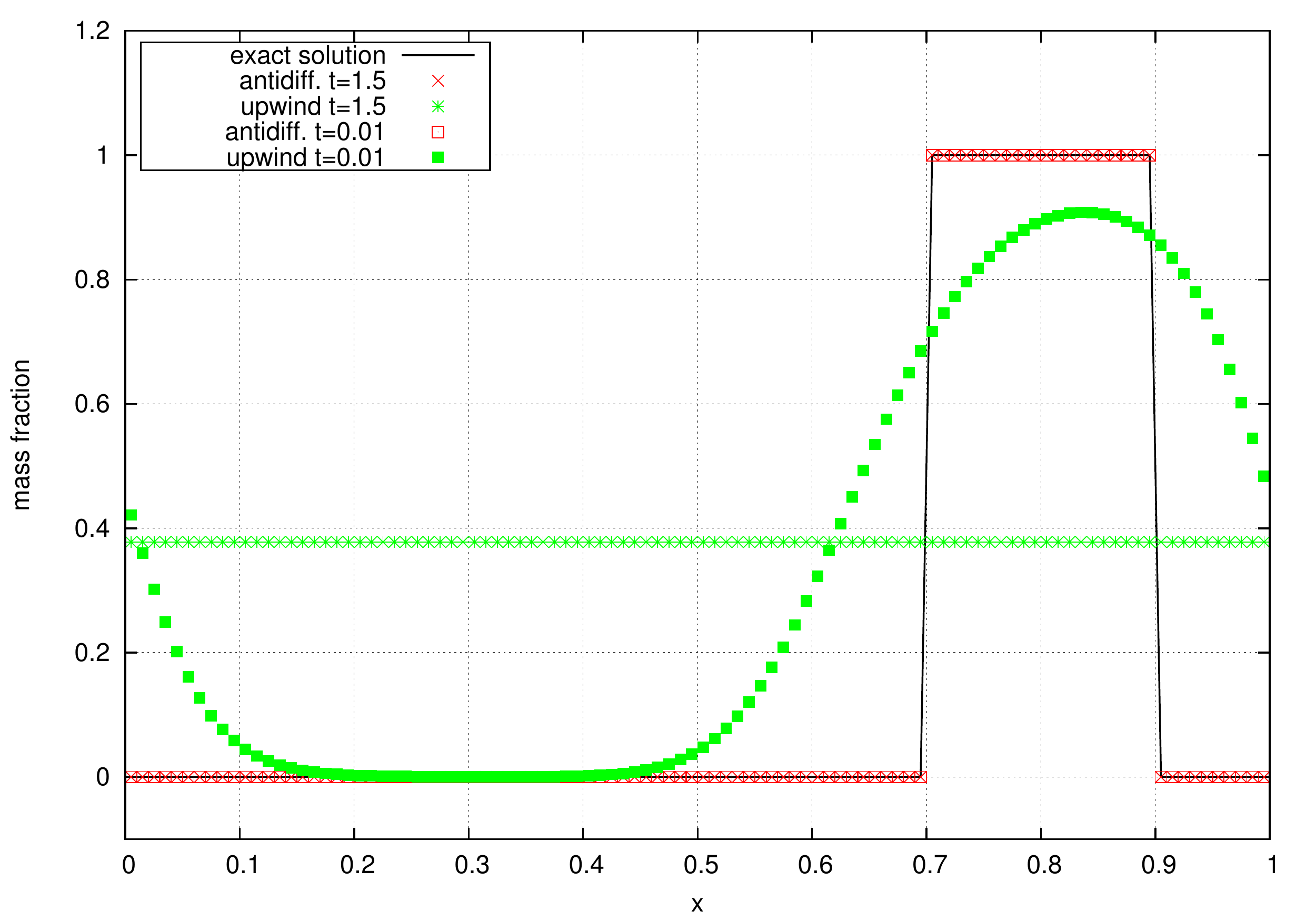} 
\\
Profile of $\y_3$.
&
Profile of $\y_4$.
\\
  \includegraphics[width=\picwidth, clip, trim =
 12mm 8mm 0mm 0mm]{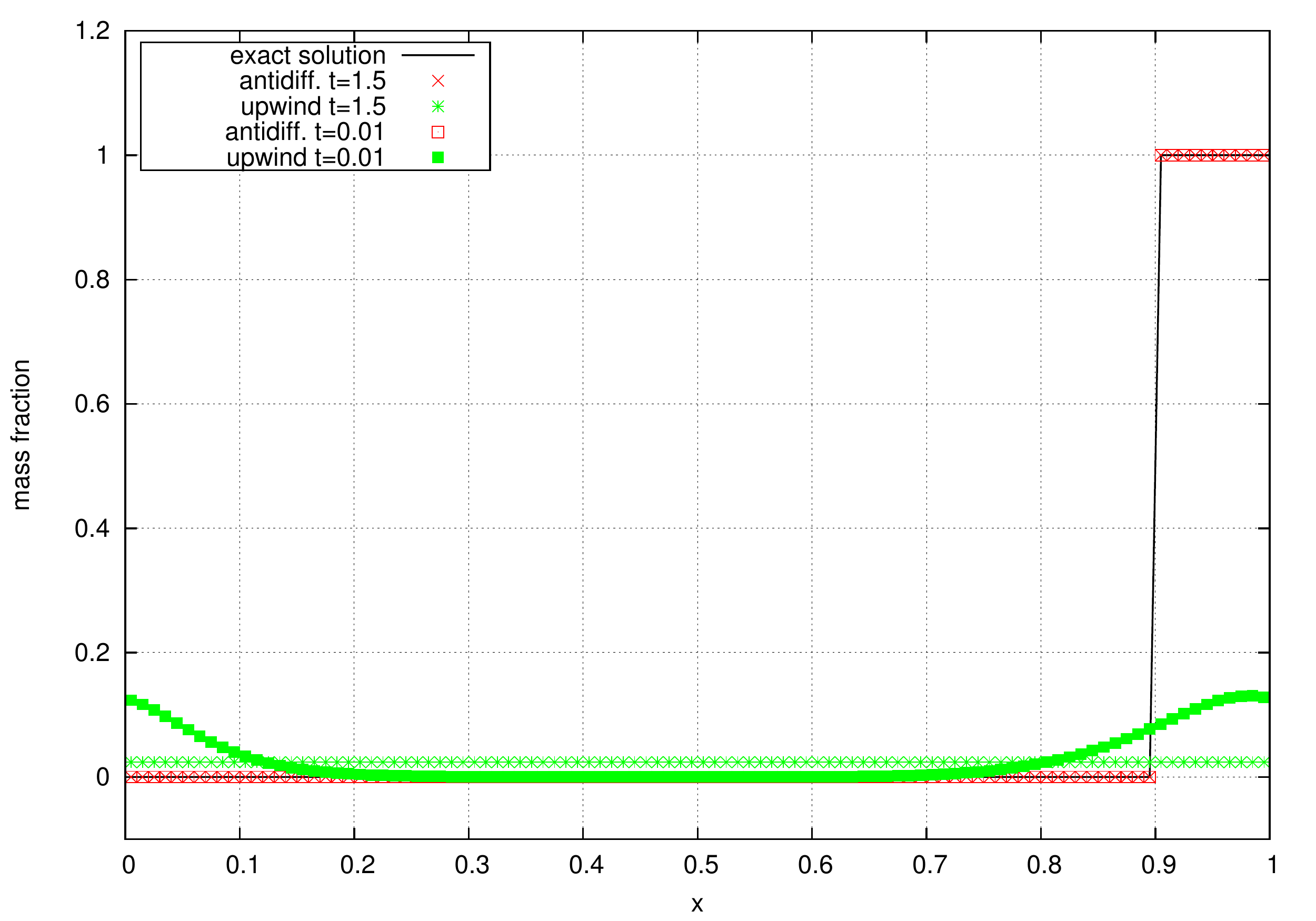}
&
  \includegraphics[width=\picwidth, clip, trim =
 12mm 8mm 0mm 0mm]{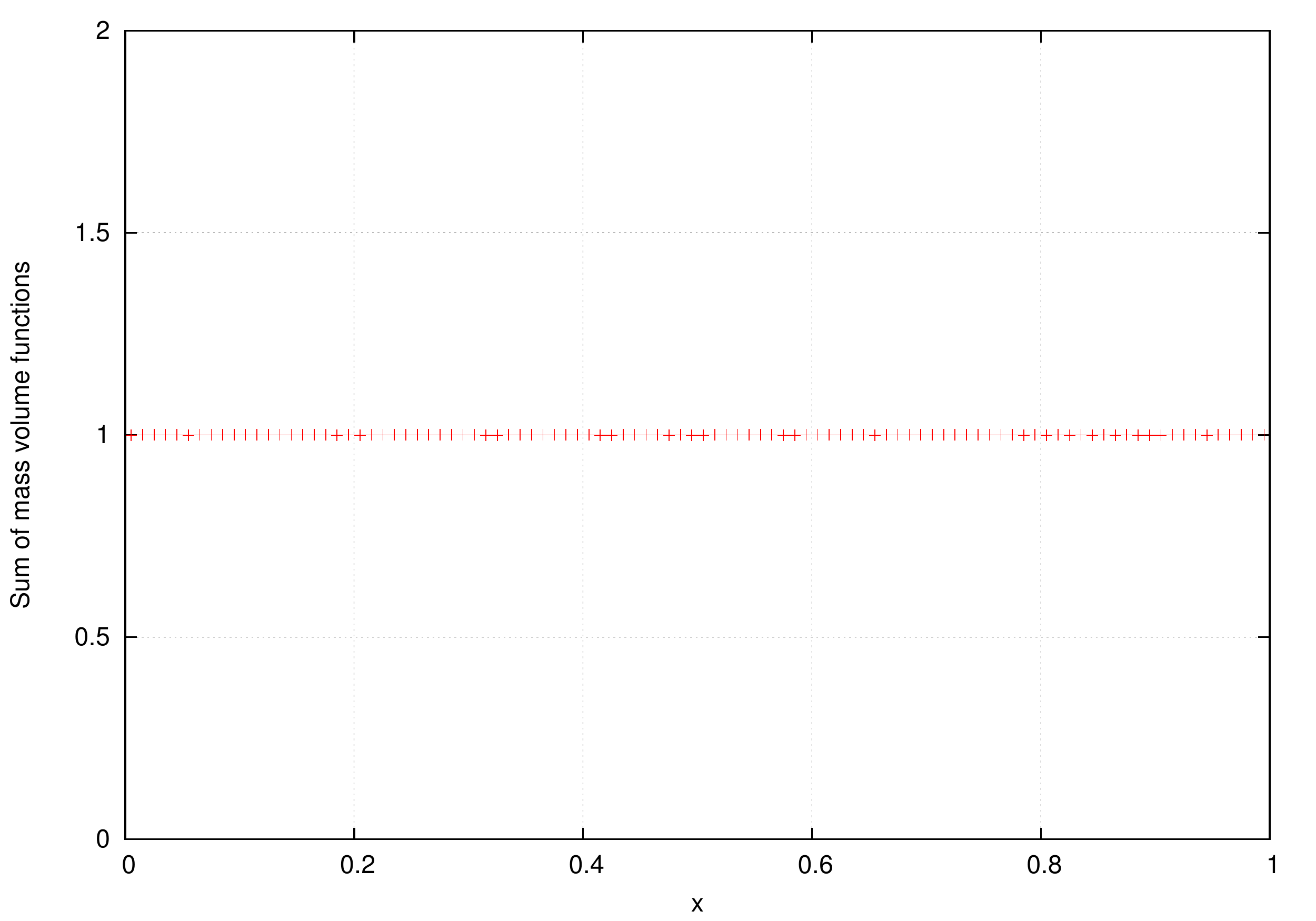}  
\\
Profile of $\y_5$ & $\sum_{k=1}^5 \yk$
 \end{tabular}
\caption{test 1, one-dimensional five-material passive transport. Profile of $\yk$ and of $\sum_{k=1}^5 \yk$ for $t\in \{0.01\,\mathrm{s},1.5\,\mathrm{s}\}$. Comparison between the upwind scheme and the anti-diffusive solver.}
 \label{fig: 1D transport y}
\end{figure}

\begin{figure}
\centering
  \includegraphics[width=1.2\picwidth, clip, trim =
 12mm 8mm 0mm 0mm]{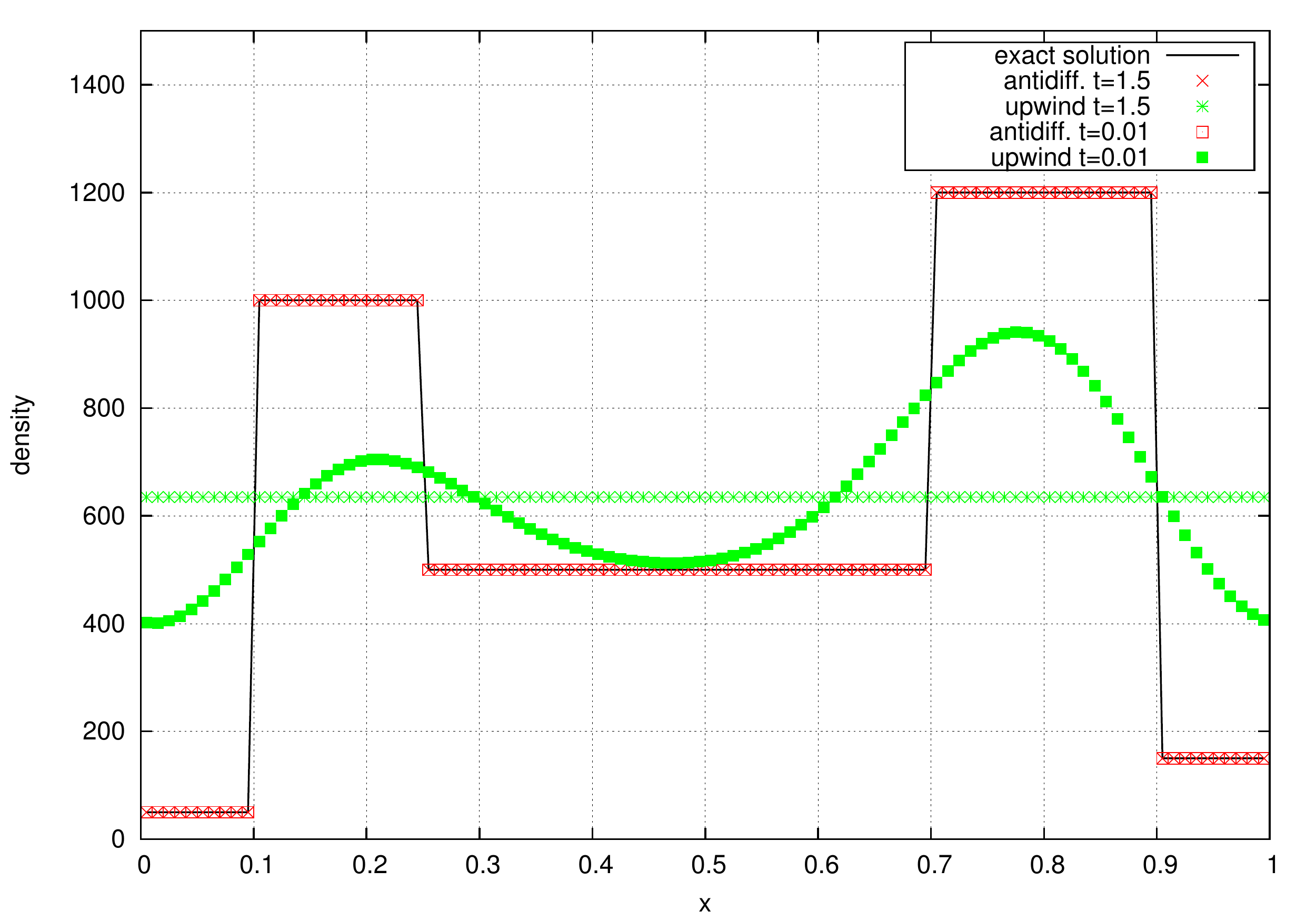} 
\caption{test 1, one-dimensional five-material passive transport. Profile of density for $t\in \{0.01\,\mathrm{s},1.5\,\mathrm{s}\}$. Comparison between the upwind scheme, the anti-diffusive solver and the exact solution.
}
 \label{fig: 1D transport density}
\end{figure}

\begin{figure}
\centering
  \includegraphics[width=0.9\picwidth, clip, trim =
 12mm 8mm 0mm 0mm]{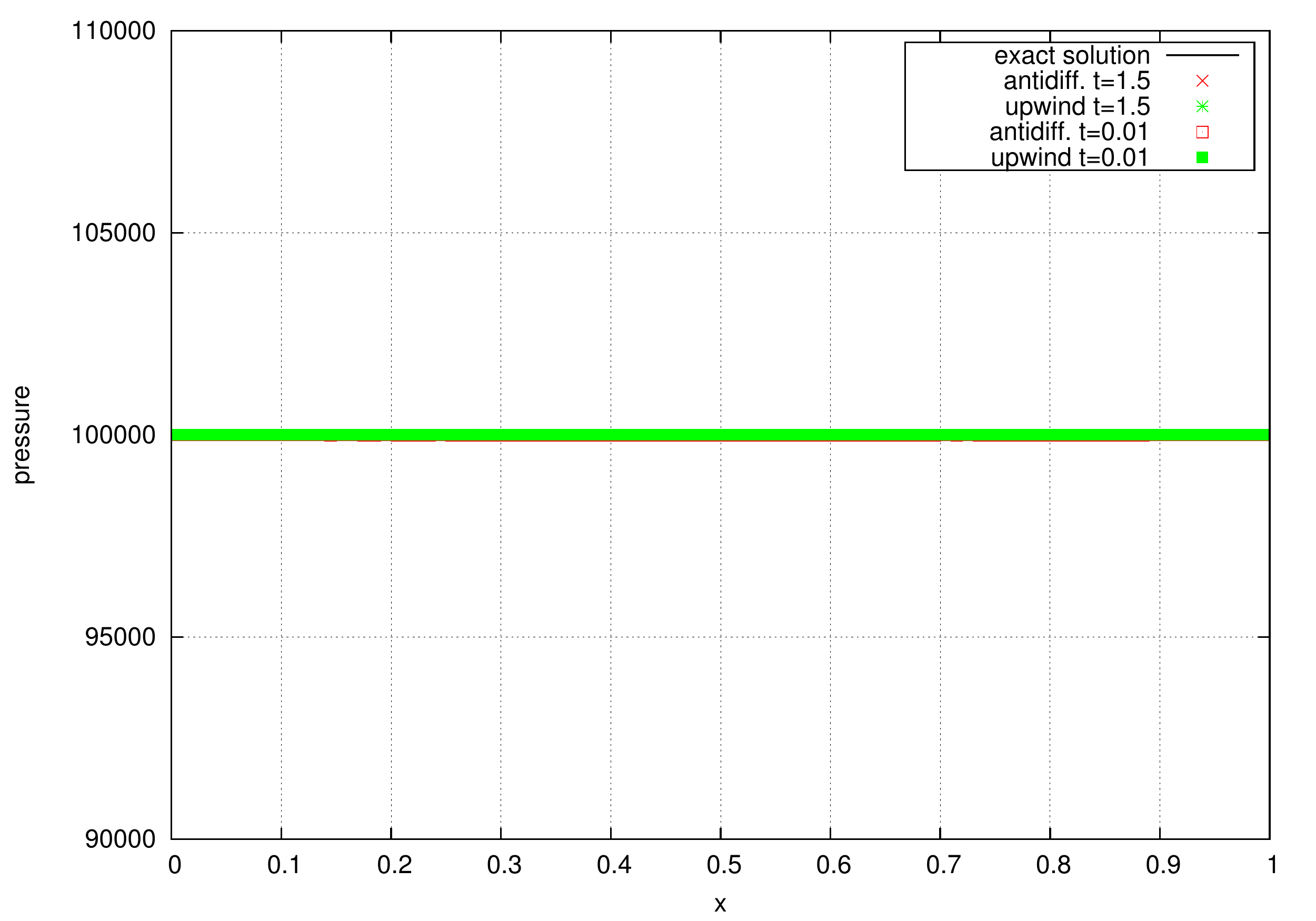} 
  \includegraphics[width=0.9\picwidth, clip, trim =
 12mm 8mm 0mm 0mm]{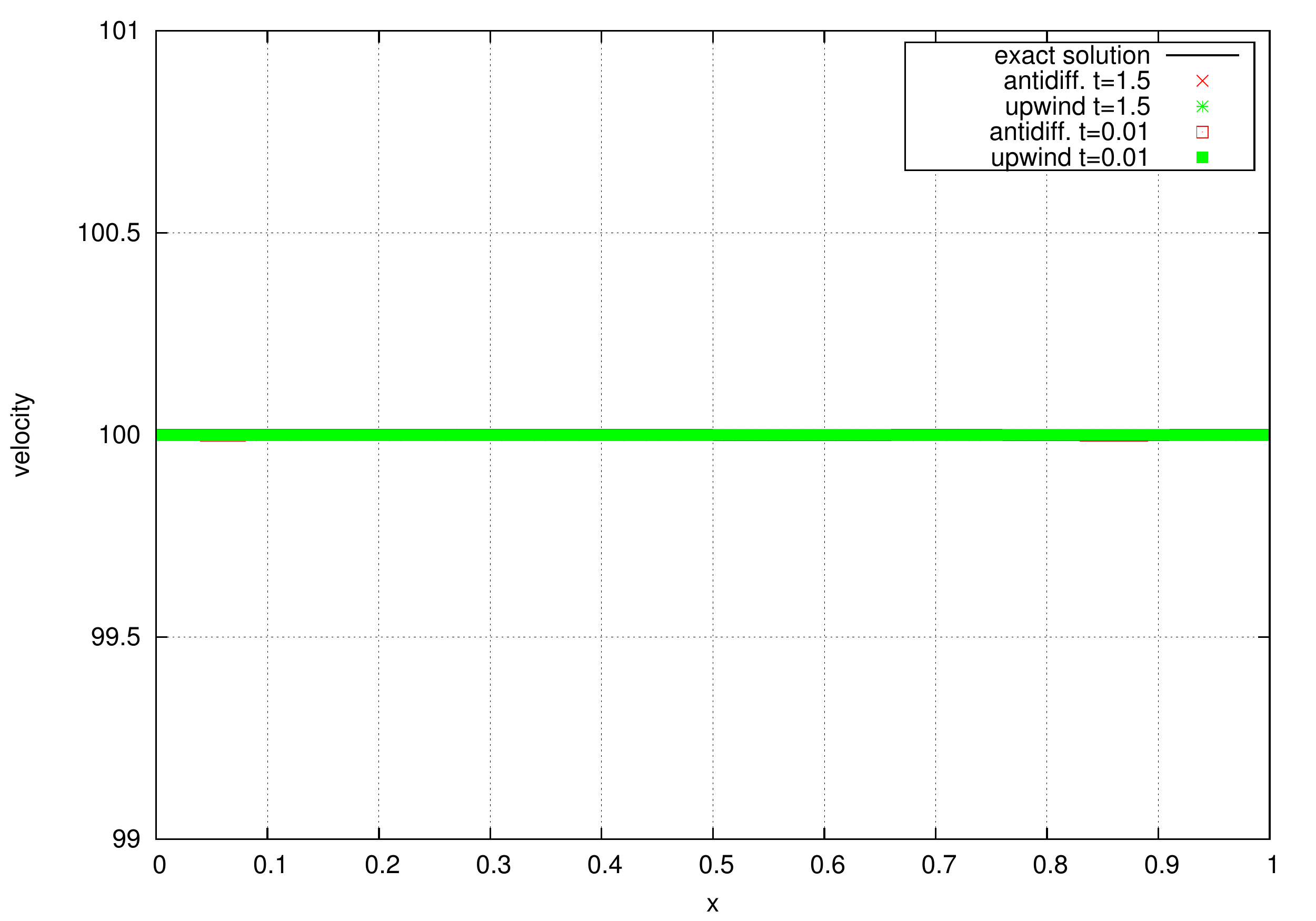}  
\caption{
test 1, one-dimensional five-material passive transport. Profile of the pressure (top) and velocity (bottom) for $t\in \{0.01\,\mathrm{s},1.5\,\mathrm{s}\}$. Comparison between the upwind scheme, the anti-diffusive solver and the exact solution.
}
 \label{fig: 1D transport pressure velocity}
\end{figure}

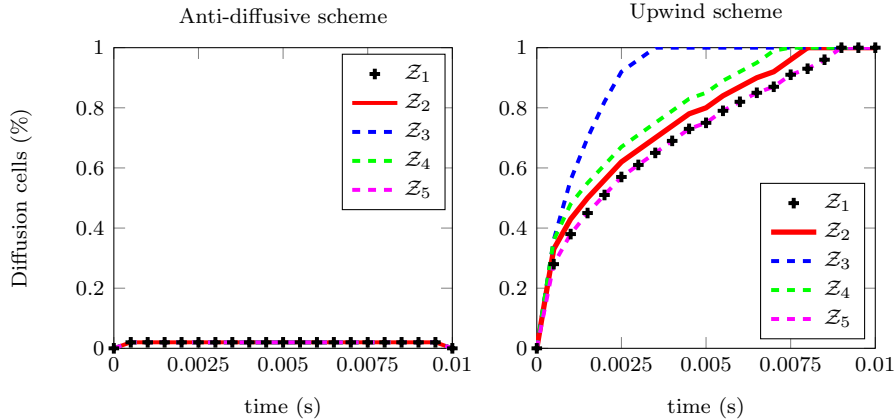
\begin{figure}
\centering 
\input{RES/1D/TRANSPORT/diffusion/antidiff_INT.tikz}
\input{RES/1D/TRANSPORT/diffusion/upwind_INT.tikz}
\caption{test 1, one-dimensional five-material passive transport. Percent of cells in the domain where the color functions $\z_k$ are numerically diffused for $t \in \{0s,0.01s\}$ for both anti-diffusive (left) and upwind (right) schemes.\label{fig: INT diff cells }}
\end{figure}

%\begin{figure}
%\centering 
%\input{RES/1D/TRANSPORT/flux/erreur_1_INT.tikz}
%\input{RES/1D/TRANSPORT/flux/erreur_2z_INT.tikz}
%\input{RES/1D/TRANSPORT/flux/erreur_2y_INT.tikz}
%\caption{One-dimensional transport test with five materials. Evolution of the errors $e_1(t),e_2(t)$  for $t \in [0s,0.01s]$. 
%\label{fig: erreur flux}}
%\end{figure}

\begin{table}[]
\centering
\caption{test 1, one-dimensional five-material passive transport.  Difference $e_1$ for $t_{end}=0.01\,\mathrm{s}$} 
{\small
\begin{tabular}{c|ccc}
\hline\hline
$a$ & $\rho$ & $p$ & $u$  \\
\hline
$e_1(a)$ &  $8.85~10^{-12}$ & $1.68~10^{-11}$ & $6.27~10^{-14}$ \\
\hline\hline
\end{tabular}
}
 \label{tab: MAT e1}
\end{table}

\begin{table}[]
\centering
\caption{test 1, one-dimensional five-material passive transport.  Difference $e_2$ for $t_{end}=0.01s$} 
{\small
\begin{tabular}{c|cccccccccc}
\hline\hline
$a_k$         & ${\cal Z}_1$ & ${\cal Z}_2$ & ${\cal Z}_3$ & ${\cal Z}_4$ & ${\cal Z}_5$\\
\hline
$e_2(a_k)$& $6.12~10^{-12}$&  $3.43~10^{-12}$ &  $ 3.45~10^{-12}$ & $3.12~10^{-12}$ &$ 6.15~10^{-12}$\\
\hline\hline
$a_k$ &  ${\cal Y}_1$ & ${\cal Y}_2$ & ${\cal Y}_3$ & ${\cal Y}_4$ & ${\cal Y}_5$ \\
\hline
$e_2(a_k)$& $  1.79~10^{-11}$
& $6.84~10^{-12}$ & $7.49~10^{-12}$ & $2.03~10^{-11}$ & $2.02~10^{-11}$\\
\hline\hline

\end{tabular}
}
 \label{tab: MAT e2}
\end{table}

%%%%%%%%%%%%%%%%%%%%%%%%%%%%%%%%%%%%%%%%%%%%%%%%%%%%%%%%%%%%%%%%%%%%%%
%%%%%%%%%%%%%%%%%%%%%%%%%%%%%%%%%%%%%%%%%%%%%%%%%%%%%%%%%%%%%%%%%%%%%%%%%
%%%%%%%%%%%%%%%%%%%%%%%%%%%%%%%%%%%%%%%%%%%%%%%%%%%%%%%%%%%%%%%%%%%%%%%%%

\subsection{Test 2: One-dimensional three-material Riemann problems juxtaposition}
%
% {\subsubsection{Test 1}}

We consider now a test that consists of two side-by-side Riemann problems within a one-dimensional $1\,\mathrm{m}$ long domain.
At $t=0$ three perfect gases are separated by two interfaces located at $x=0.4\,\mathrm{m}$ and $x=0.6\,\mathrm{m}$. Initially the fluid is 
at rest in the whole domain and both pressure and density are constant in each region $0<x<0.4$, $0.4\le x \le 0.6$ and $0.6<x<1$. These values are displayed in table~\ref{tab: shock tube init data} along with the specific heat ratio value of each fluid. The initial discontinuity at $x=0.6$ is a stationnary material discontinuity.
\begin{table}
\centering
\caption{test 2, one-dimensional three-material Riemann problems juxtaposition. Initial values for pressure and density in the domain.}
\begin{tabular}{ccccc}
\hline\hline
 location & $k$ & $\rho_k$ \scriptsize{$(\mathrm{kg}.\mathrm{m}^{-3})$} & $p_k$ \scriptsize{$(\mathrm{Pa})$} & $\gamma_k$  \\
\hline \hline
$0<x<0.4$ & $k=1$ & $1.0$ & $1.0$ & $1.6$\\ 
$0.4\le x \le 0.6$ & $k=2$ & $ 0.125$ & $0.1$ & $2.4$\\ 
$0.6<x<1$ & $k=3$ & $ 0.1   $ & $0.1$ & $1.4$\\ 
\hline\hline
\end{tabular}
 \label{tab: shock tube init data}
\end{table}
The exact solution of this problem is composed by a set of waves that are depicted in figure~\ref{fig:struct_sol}: the Riemann problem centered at $x=0.4$ produces a leftward going rarefaction wave and two waves that travels towards the right end of the domain. The fastest of these waves is a shock that propagates at speed $D_1 \simeq 2.2780\,\mathrm{m}.\mathrm{s}^{-1}$. This shock hits at $t_\mathrm{shock}\simeq 0.0878\,\mathrm{s}$ the 
stationnary discontinuity located at $x=0.6$ and triggers another two-component Riemann problem centered at $x=0.6$. This Riemann problem is also composed by a rarefaction wave that travels leftwards and two waves moving towards the right end of the domain: a material discontinuity and a shock wave of speed $D_2 \simeq 2.034\,\mathrm{m}.\mathrm{s}^{-1}$.
\begin{figure}[h!]
\centering
 \begin{tikzpicture}[scale=0.5]
    %les diff�rents gazs
    \fill[fill=white] (0,0)--(5,0)--(7,8)--(0,8)--cycle;
    \fill[fill=gray!10] (5,0)--(7,8)--(11,8)--(10,4)--(10,0)--cycle;
    \fill[fill=gray!30] (15,0)--(15,8)--(11,8)--(10,4)--(10,0)--cycle;
    % les axes
     \tikzstyle{fleche}=[->,>=latex,line width=0.5mm]
      \draw[fleche] (0,0)--(0,8) node[below left]{$t$};
     \draw[fleche] (0,0)--(15,0) node[below left]{$x$};
     % ddc
     \draw[dashed,line width=0.7mm] (5,0) -- (7,8);
     \draw[dashed,line width=0.7mm] (10,0) -- (10,4);
     \draw[dashed,line width=0.7mm] (10,4)  -- (11,8);
     %interfaces
     \node[black,below] at (5,0) { $0.4$};
     \node[black,below] at (10,0) { $0.6$};
      \node[black,below] at (0,0) { $0$};
       \node[black,left] at (0,0) { $0$};
     % R1
     \draw[line width=0.5mm] (5,0) -- (2.5,8);
     \draw[line width=0.5mm] (5,0) -- (2.25,8); 
     \draw[line width=0.5mm] (5,0) -- (2.0,8); 
     % R2
     \draw[line width=0.5mm] (10,4) -- (8.5,8);
     \draw[line width=0.5mm] (10,4) -- (8.25,8); 
     \draw[line width=0.5mm] (10,4) -- (8.0,8);  
          %chocs
     \draw[line width=1mm] (5,0) -- (10,4);
     \draw[line width=1mm] (10,4) -- (13,8);
     % temps
     \draw[dotted,line width=0.5mm] (0,4) node[left]{$t_{\rm shock}$} -- (15,4) ;
     \draw[dotted,line width=0.5mm] (0,7) node[left]{$t_{\rm end}$} -- (15,7);
     % wave legend
      \draw[dashed,line width=0.7mm] (-1,-2) -- (0,-2);
      \node[black,right] at (0,-2) { \footnotesize contact-discontinuity};
      \draw[line width=1mm] (6,-2) -- (7,-2);
      \node[black,right] at (7,-2) { \footnotesize shock};
      \draw[line width=0.5mm] (9,-2) -- (10,-1.8);
      \draw[line width=0.5mm] (9,-2) -- (10,-2);
      \draw[line width=0.5mm] (9,-2) -- (10,-2.2);
      \node[black,right] at (10,-2) { \footnotesize rarefaction};
       % ${\boldsymbol{ -}}$ shock, $\boldsymbol{=}$ rarefaction
\end{tikzpicture}
\caption{test 2, one-dimensional three-material Riemann problems juxtaposition. Wave structure of the solution in the $(x,t)$-plane.  \label{fig:struct_sol}}
\end{figure}
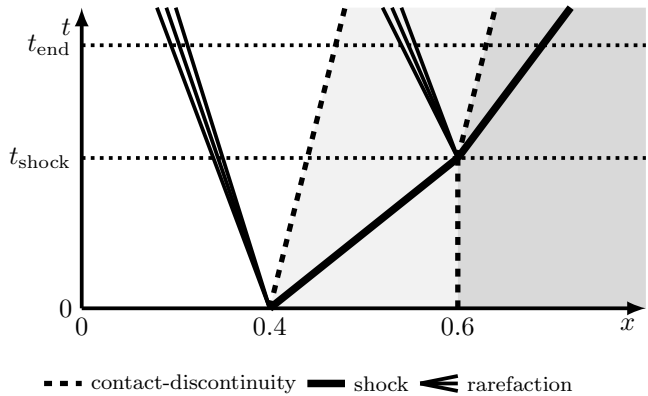
The computational domain is discretized over a {500-cell} mesh and we perform simulations with both the anti-diffusive and the Lagrange-Remap upwind scheme using $C_\mathrm{CFL}=0.8$. Transparent boundary conditions are set at each end of the domain.
\begin{figure}
 \centering
\begin{tabular}{cc}
 \includegraphics[width=\picwidth, clip, trim =
 12mm 8mm 0mm 0mm]{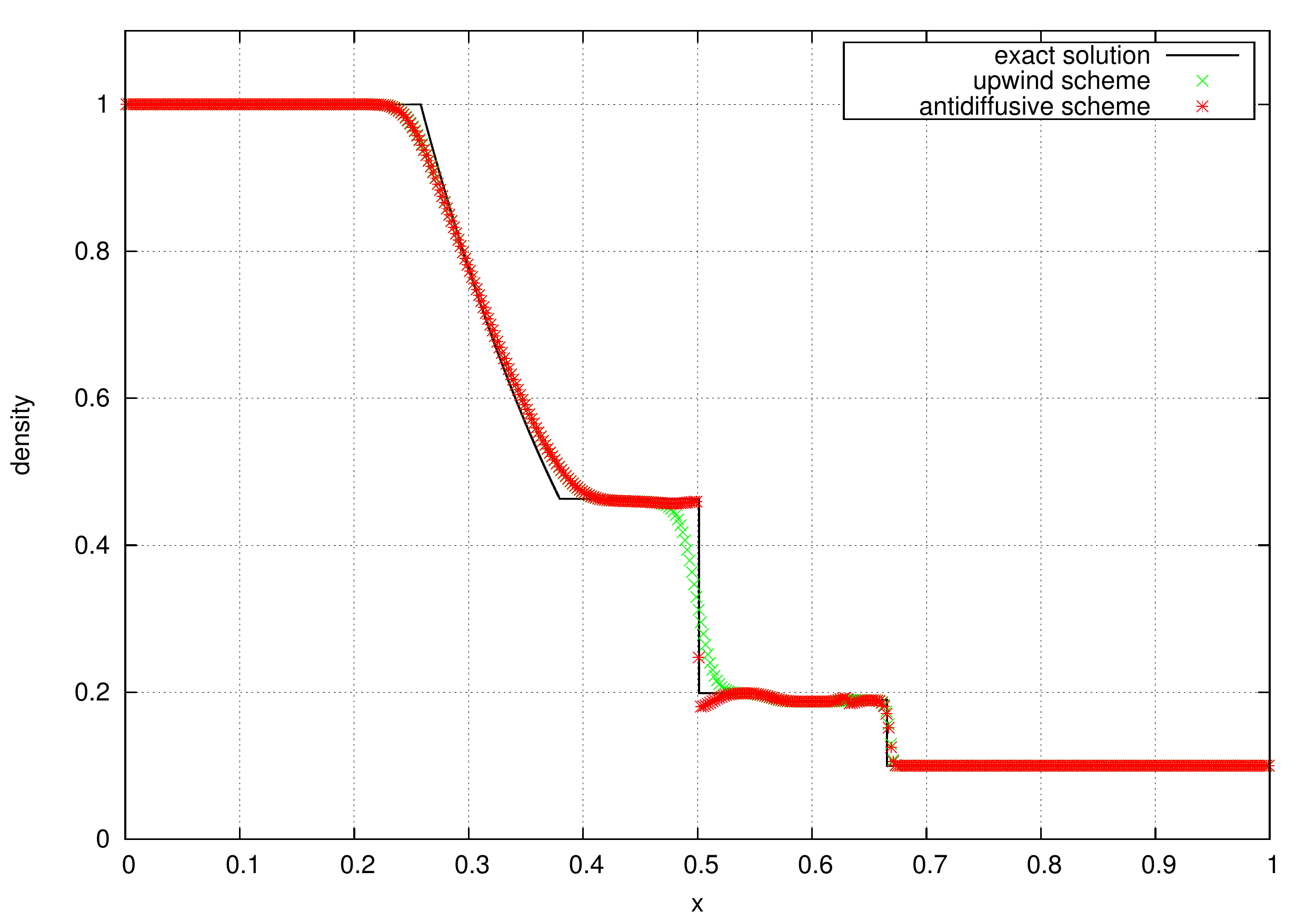} 
&
 \includegraphics[width=\picwidth, clip, trim =
 12mm 8mm 0mm 0mm]{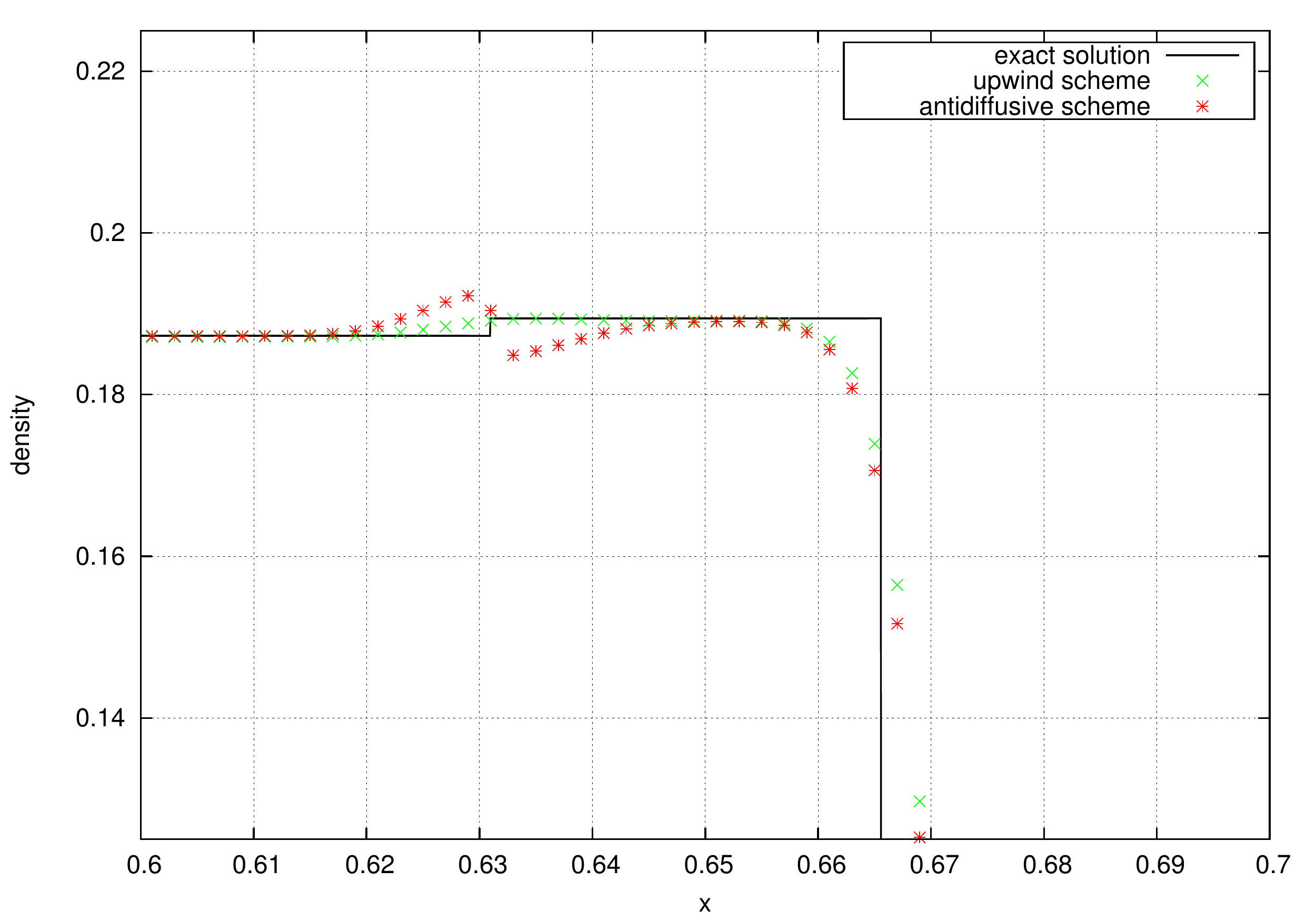}  
\\
Density & Zoom on the density\\
 \includegraphics[width=\picwidth, clip, trim =
 12mm 8mm 0mm 0mm]{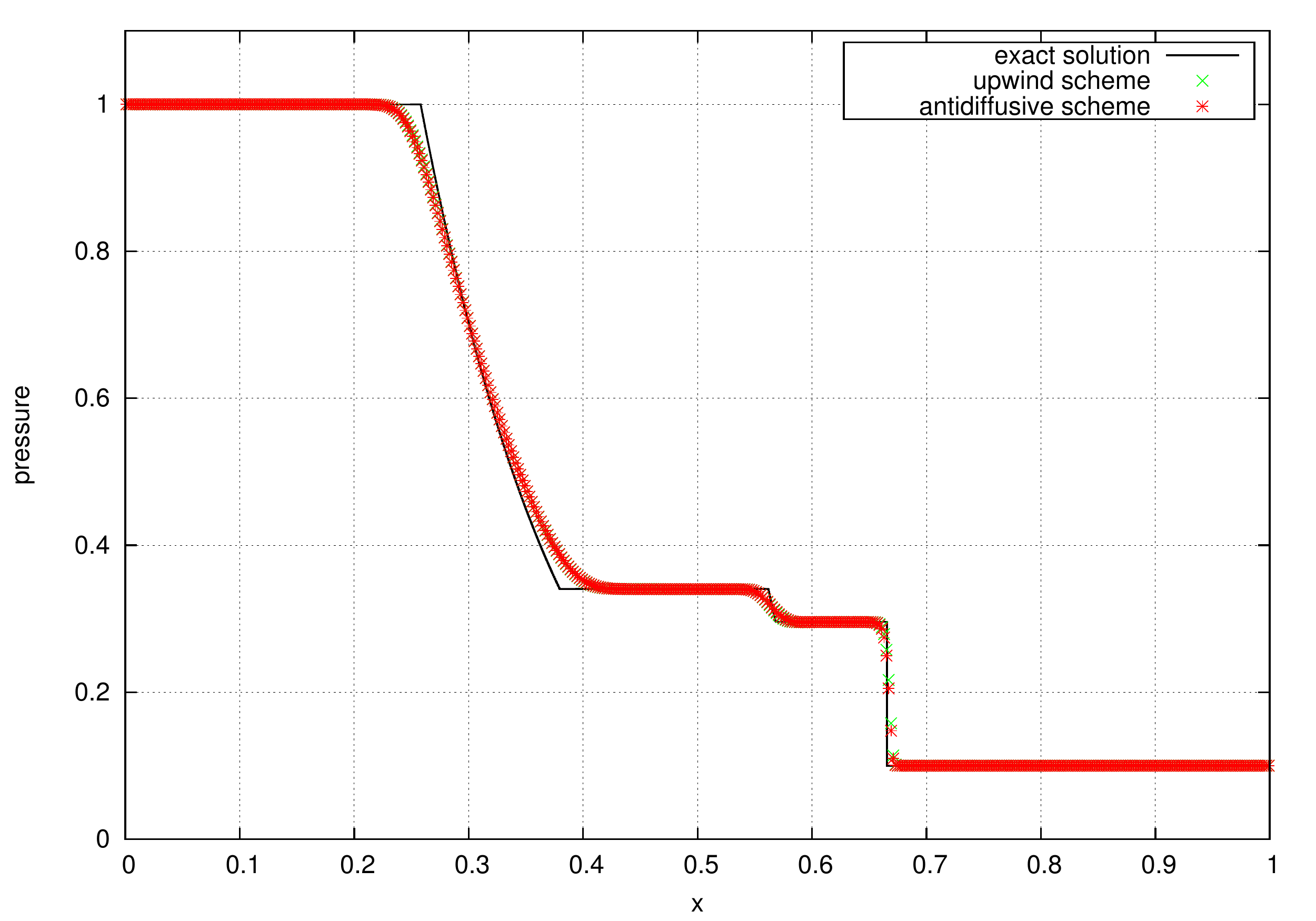}
 &
 \includegraphics[width=\picwidth, clip, trim =
 12mm 8mm 0mm 0mm]{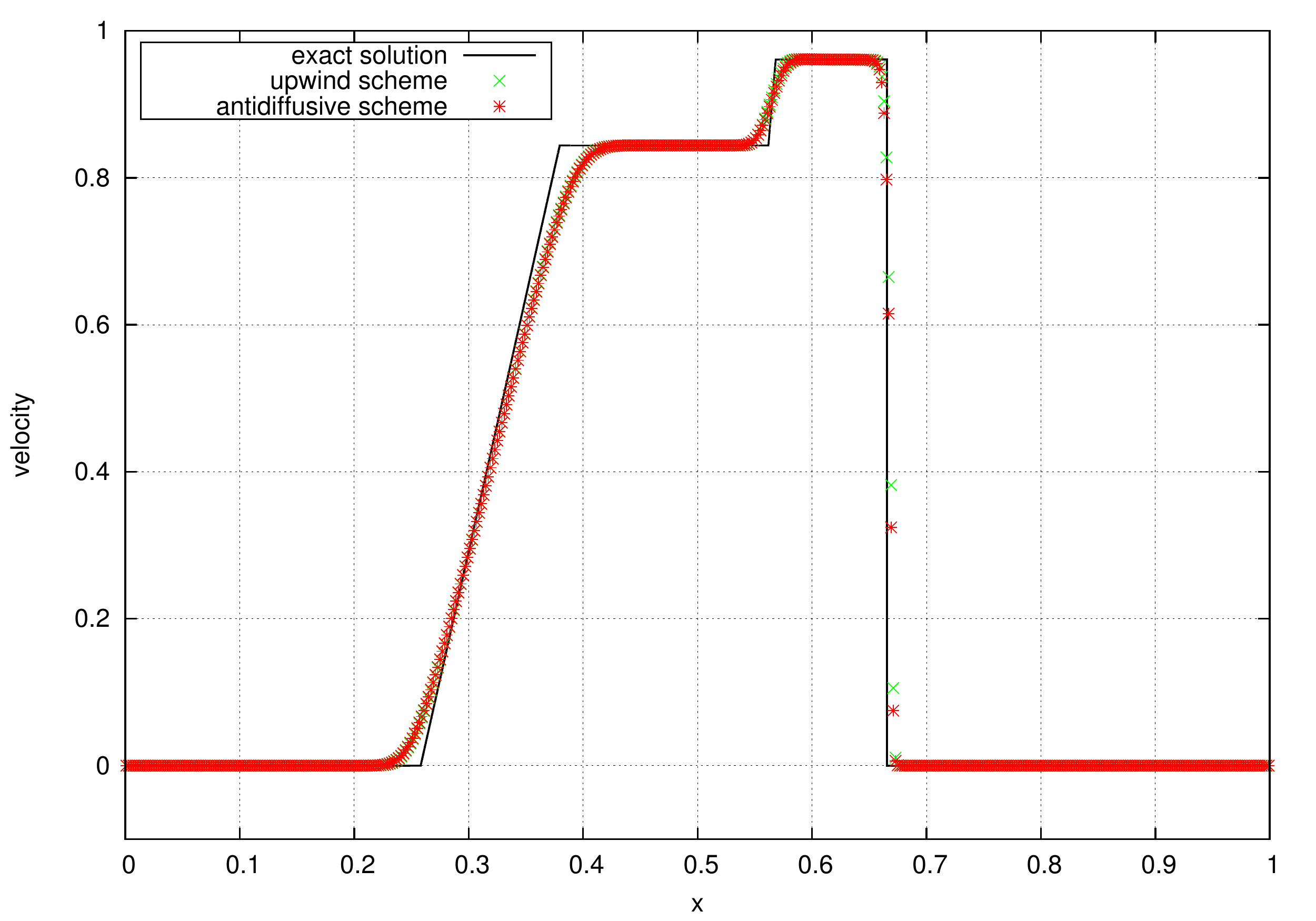} \\
 Pressure & Velocity
\end{tabular}
\caption{test 2, one-dimensional three-material Riemann problems juxtaposition. Profile of the density (top left) and zoom on the second interface (top right), pressure and velocity in the domain at instant $t_\mathrm{end}=0.12\,\mathrm{s}$  for the $500$-cell mesh.}
\label{fig: shocktube pressure density velocity}
\end{figure}
\begin{figure}
 \centering
\begin{tabular}{cc}
 \includegraphics[width=\picwidth, clip, trim =
 12mm 8mm 0mm 0mm]{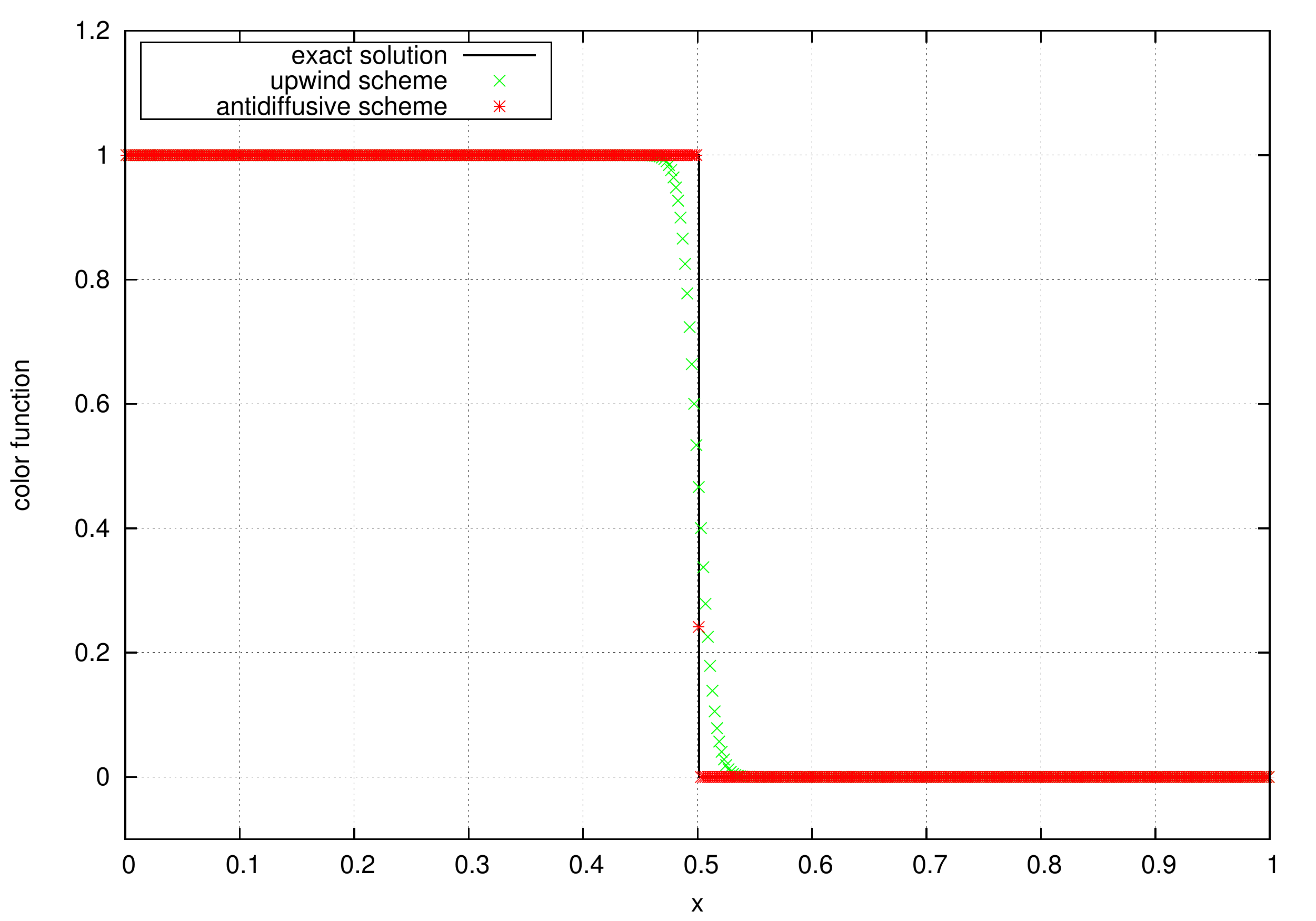} 
&
 \includegraphics[width=\picwidth, clip, trim =
 12mm 8mm 0mm 0mm]{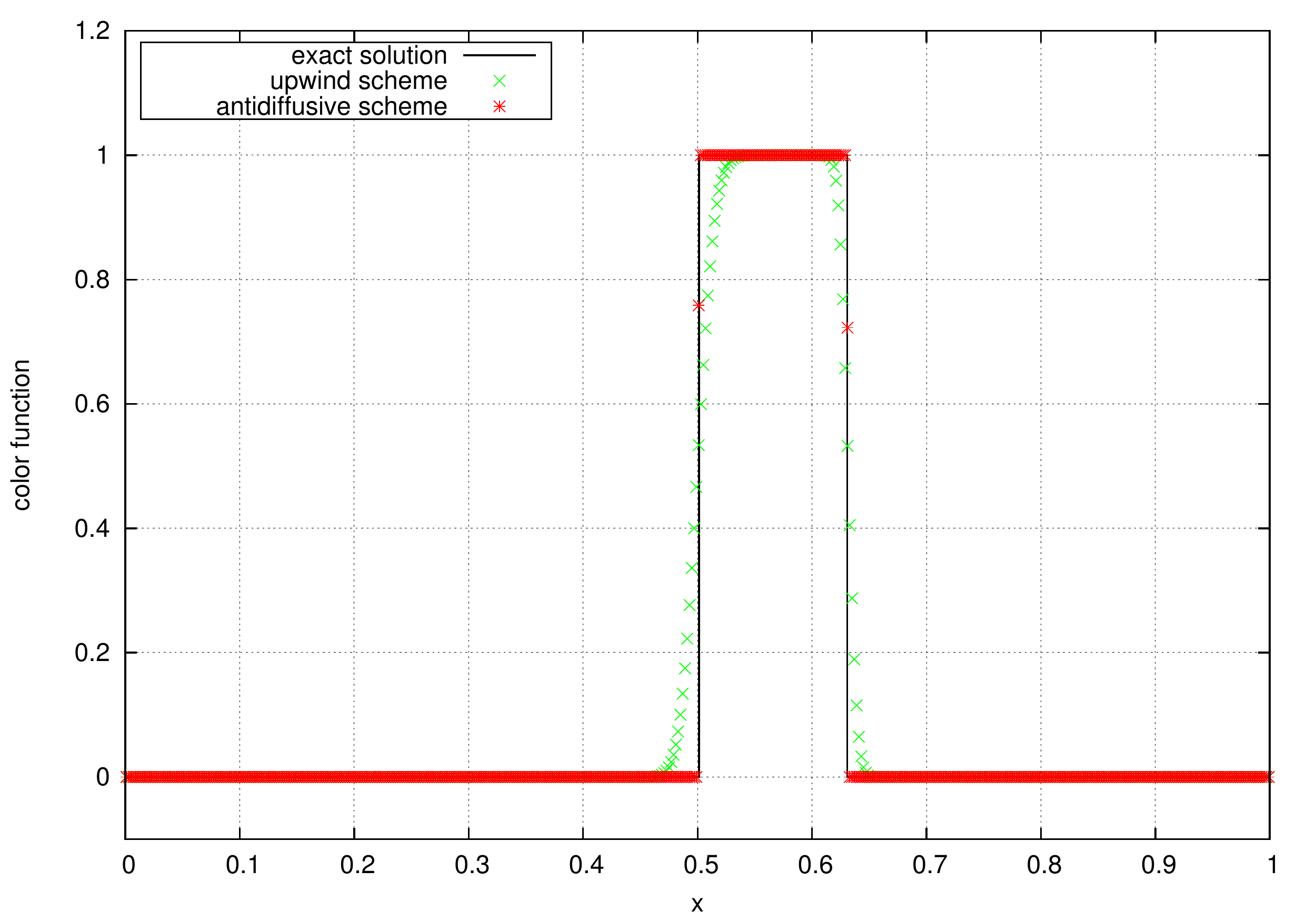} 
\\
${\cal Z}_1$ & ${\cal Z}_2$\\
 \includegraphics[width=\picwidth, clip, trim =
 12mm 8mm 0mm 0mm]{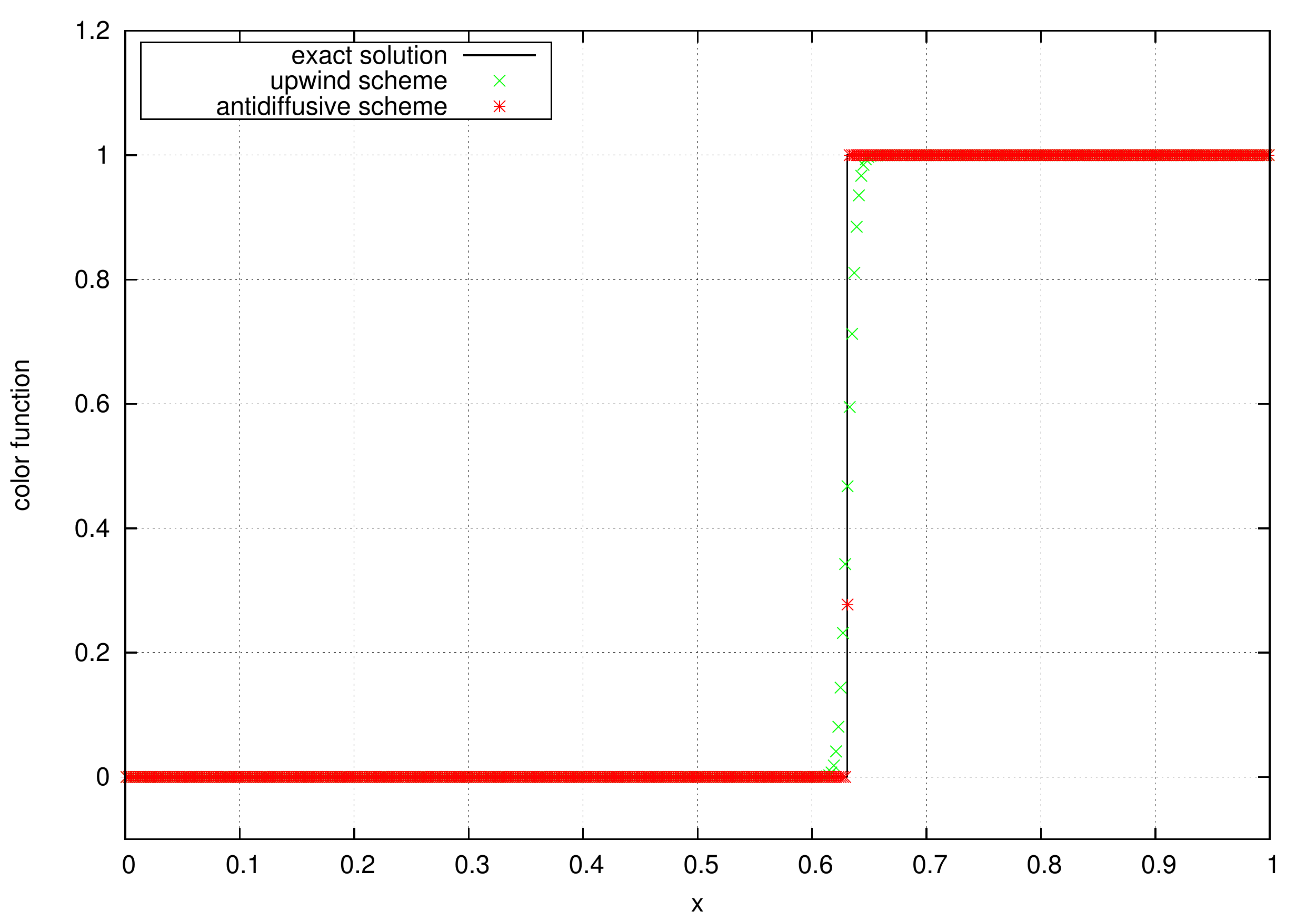} 
&
 \includegraphics[width=\picwidth, clip, trim =
 12mm 8mm 0mm 0mm]{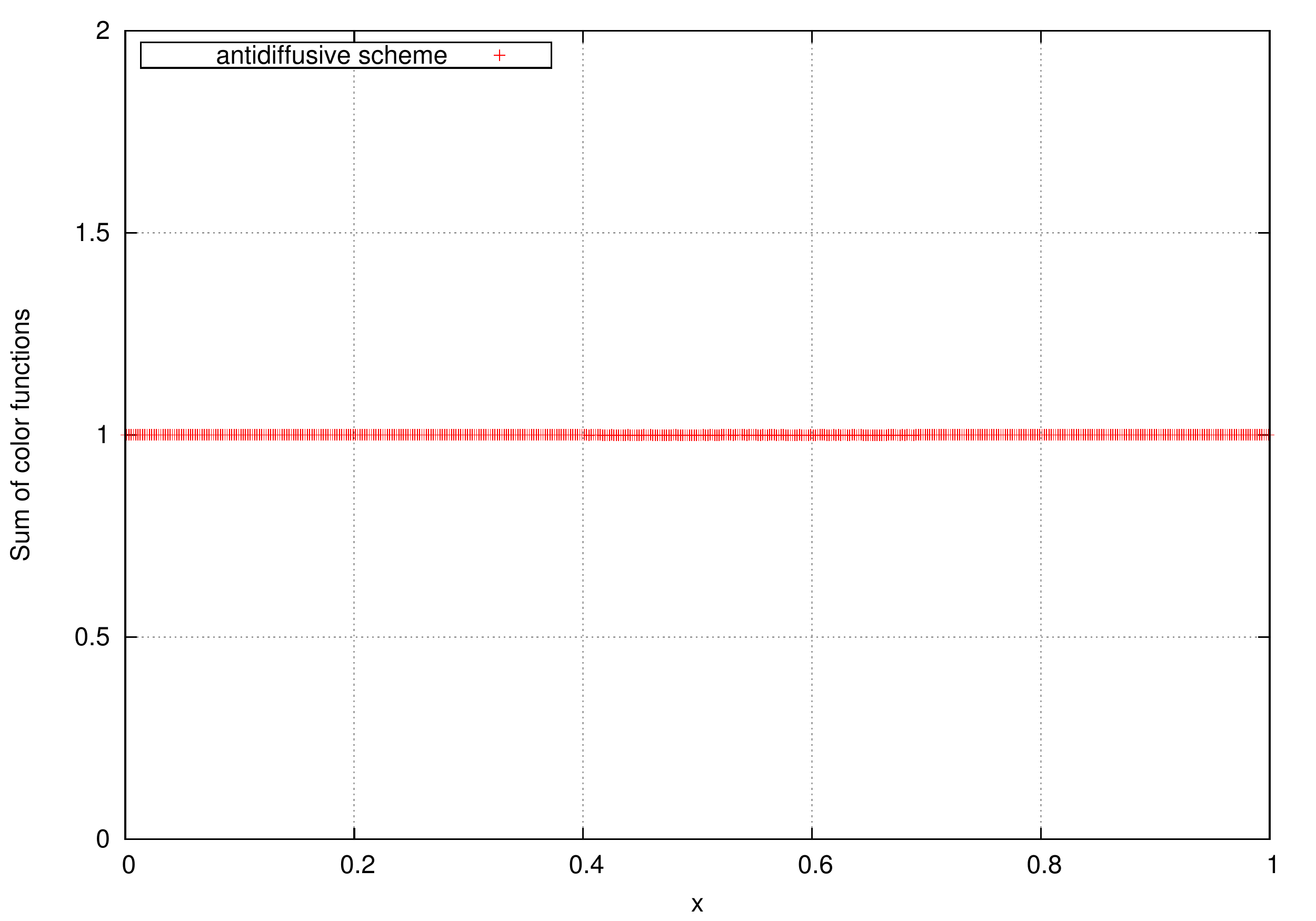} \\
 ${\cal Z}_3$ & $\z_1+\z_2+\z_3$\\
\end{tabular}
\caption{test 2, one-dimensional three-material Riemann problems juxtaposition. Profile of the color functions $\z_k$, $k=1,2,3$ and of $\z_1+\z_2+\z_3$ at instant $t_\mathrm{end}=0.12\,\mathrm{s}$ for the $500$-cell mesh.}
\label{fig: shocktube z}
\end{figure}
\begin{figure}
 \centering
\begin{tabular}{cc}
 \includegraphics[width=\picwidth, clip, trim =
 12mm 8mm 0mm 0mm]{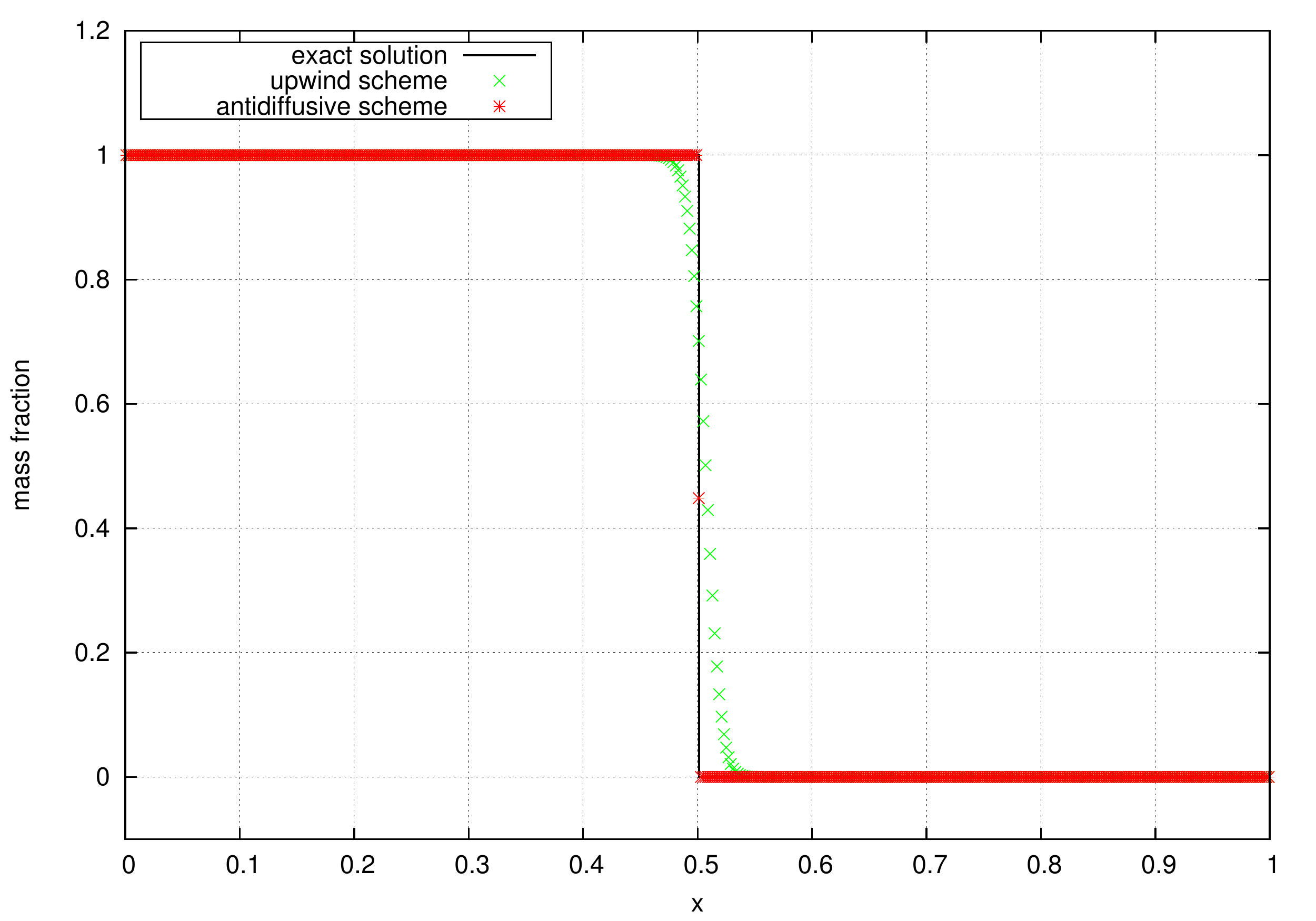} 
&
 \includegraphics[width=\picwidth, clip, trim =
 12mm 8mm 0mm 0mm]{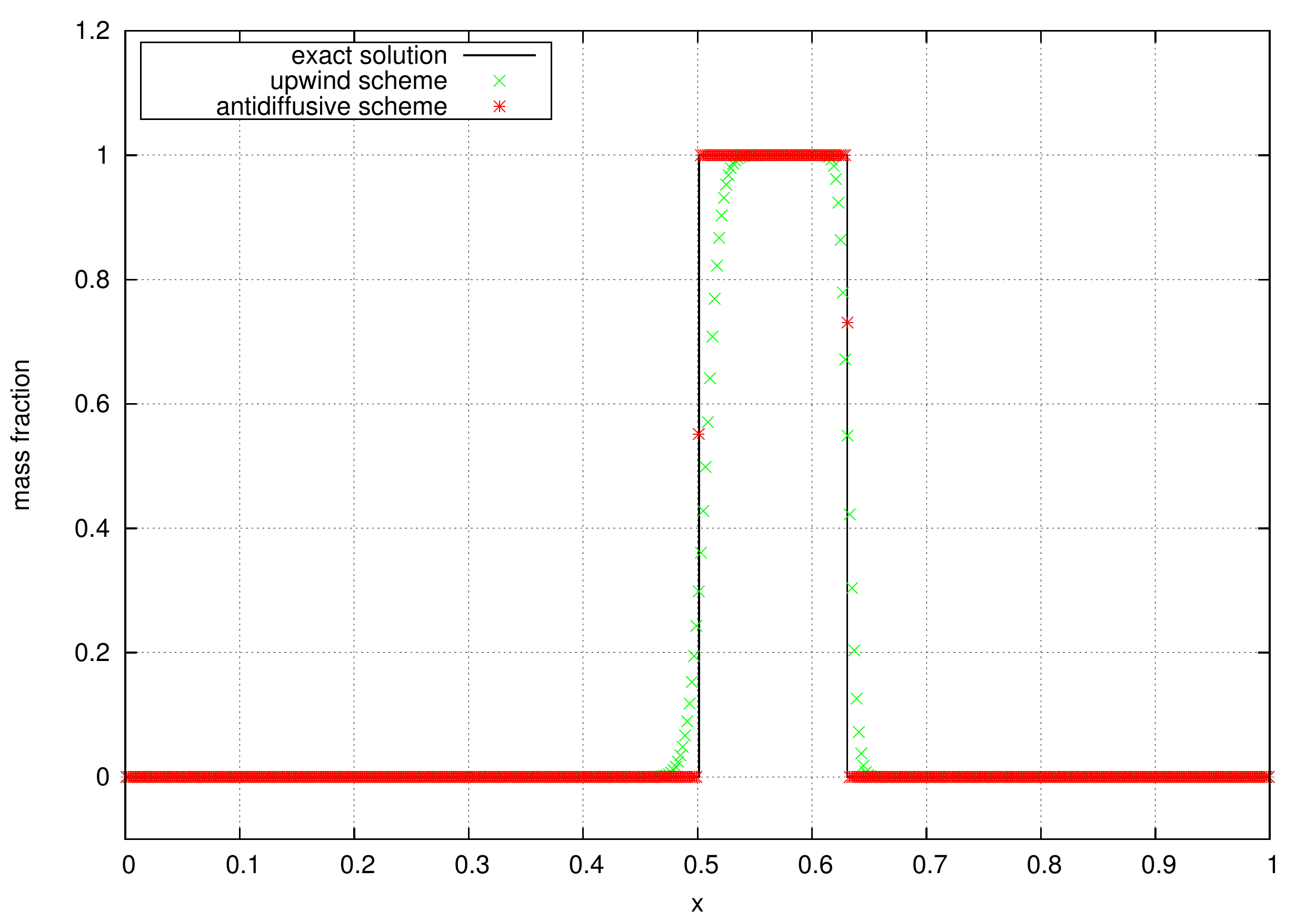} 
\\
${\cal Y}_1$ & ${\cal Y}_2$\\
 \includegraphics[width=\picwidth, clip, trim =
 12mm 8mm 0mm 0mm]{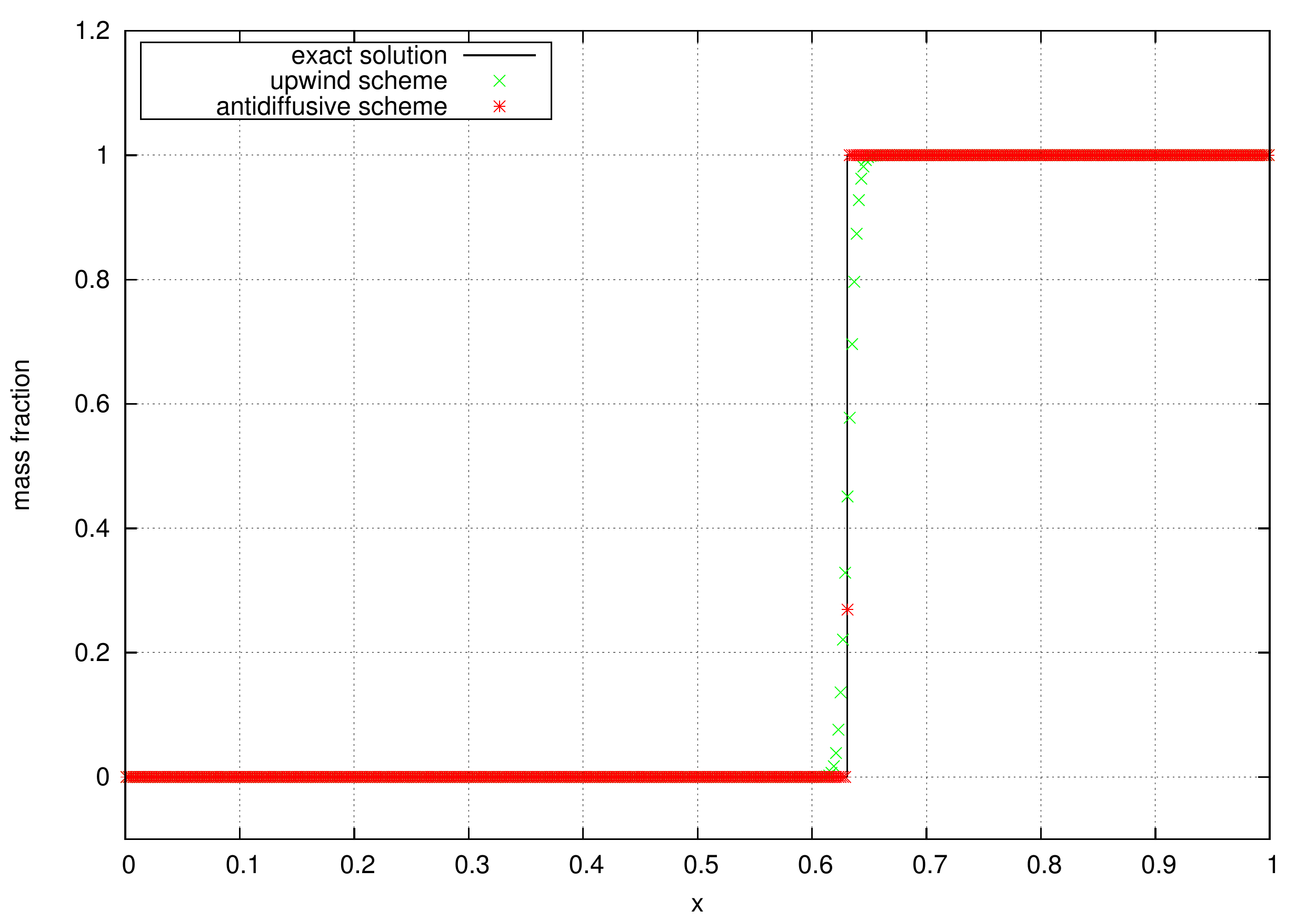} 
&
 \includegraphics[width=\picwidth, clip, trim =
 12mm 8mm 0mm 0mm]{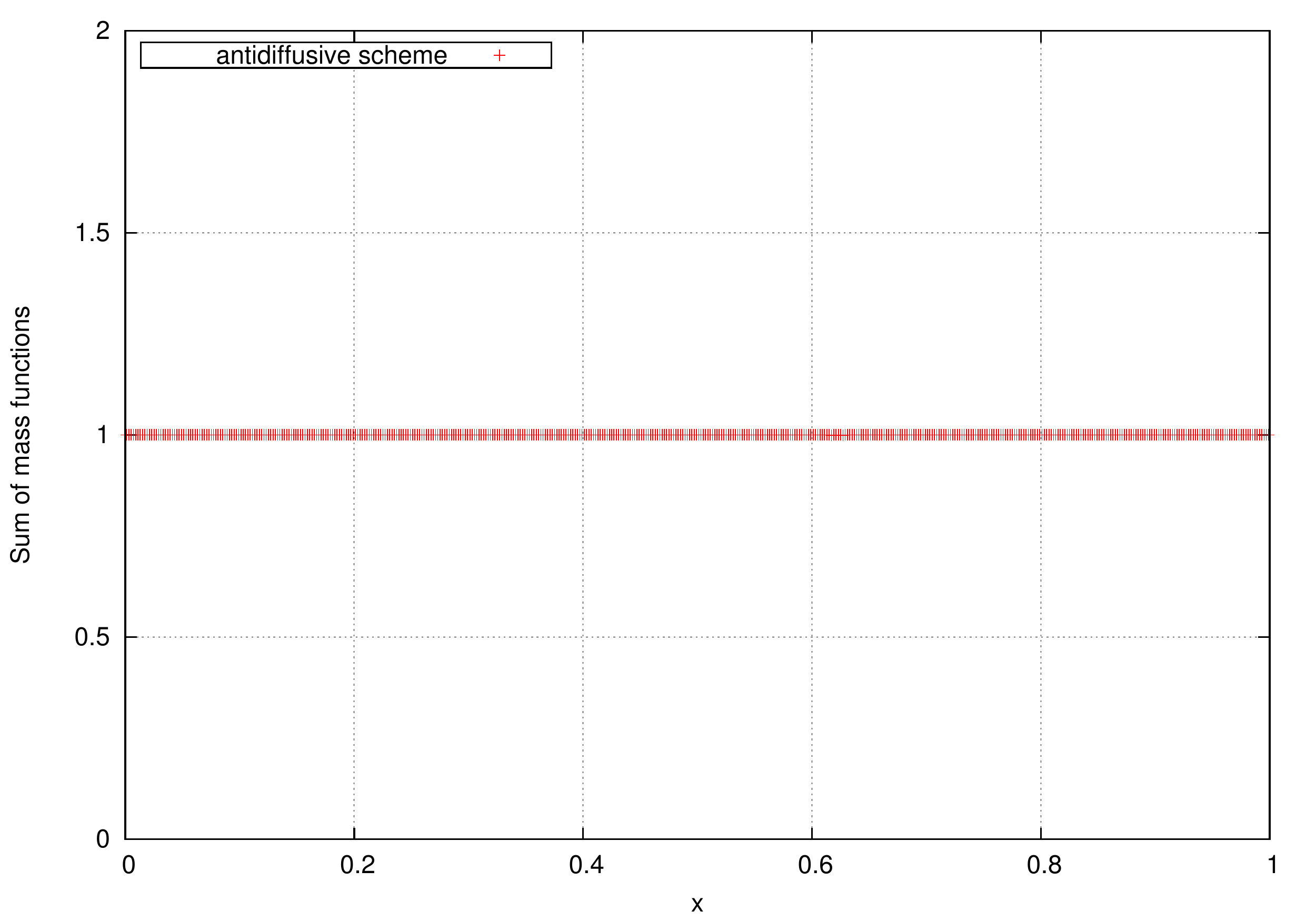} \\
  ${\cal Y}_3$ & $\y_1+\y_2+\y_3$\\
\end{tabular}
\caption{test 2, one-dimensional three-material Riemann problems juxtaposition. Profile of the mass fractions $\y_k$, $k=1,2,3$ and of $\y_1+\y_2+\y_3$ at instant $t_\mathrm{end}=0.12\,\mathrm{s}$ for the  $500$-cell mesh.}
\label{fig: shocktube y}
\end{figure}
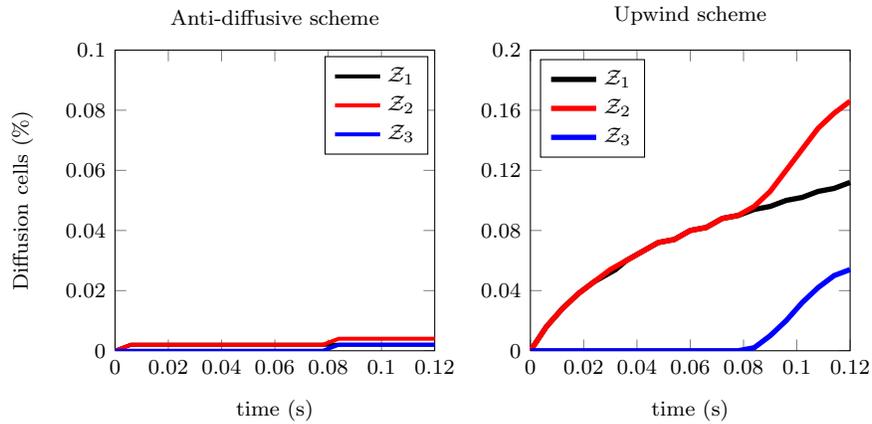
\begin{figure}
\centering 
\input{RES/1D/TAC/diffusion/antidiff_SOD.tikz}
\input{RES/1D/TAC/diffusion/upwind_SOD.tikz}
\caption{test 2, one-dimensional three-material Riemann problems juxtaposition. Percent of cells in the domain where the color functions $\z_k$ are numerically diffused for $t \in \{0,0.12\}\,\mathrm{s}$ for both anti-diffusive (left) and upwind (right) schemes.\label{fig: SOD diff cells }}
\end{figure}

Figure~\ref{fig: shocktube pressure density velocity} displays the profile of the density, pressure and velocity and figures \ref{fig: shocktube z} and \ref{fig: shocktube y} show respectively the color functions and mass fractions at instant   $t_\mathrm{end}=0.12\,\mathrm{s}$. {As we can see, the agreement between the solution obtained with both solvers and the exact solution is good. As for the two-component solver of \cite{Kokh1}, the behavior of the upwind solver and the anti-diffusive solvers are quite similar for the approximation of the shock waves and the rarefaction waves. On the contrary the anti-diffusive mechanism is very efficient for the resolution of the contact waves that coincide with the material interfaces. Indeed, the density jump at the interface between fluid $1$ and $2$ (resp. $2$ and $3$) is captured with very few cells of numerical diffusion: as depicted in figure \ref{fig: SOD diff cells } 
the anti-diffusive scheme generates at most two cells of numerical diffusion for each variable $\zk$ (i.e. $0.4\%$ of cells).} Furthermore, we again satisfy the maximum principle \eqref{eq:Z max} and discrete unit constraint \eqref{eq:Z unity}. \\ 
{A slight undershoot across both interfaces between the different materials is present on the density profiles for the anti-diffusive scheme. This is typical of the proposed method, indeed similar observations have been made two-component case \cite{Billaud1,Kokh1}. Numerical experiment shows that this defect disappears with mesh refinement demonstrating that the numerical solution converges to the exact one. In order to illustrate this point, we run the same three-material shock tube with a finer mesh of $50000$ cells. Results for the density, the pressure and the velocity are displayed in figure~\ref{fig: shocktube pressure density velocity 50000} and as expected the numerical solution converge to the exact one, in particular the overshoots at the interfaces for the density profile have significantly shrunk. The results for the mass fractions and color functions agree with the exact solutions without any surprising phenomenons.} \\
\begin{figure}
 \centering
\begin{tabular}{cc}
 \includegraphics[width=\picwidth, clip, trim =
 12mm 8mm 0mm 0mm]{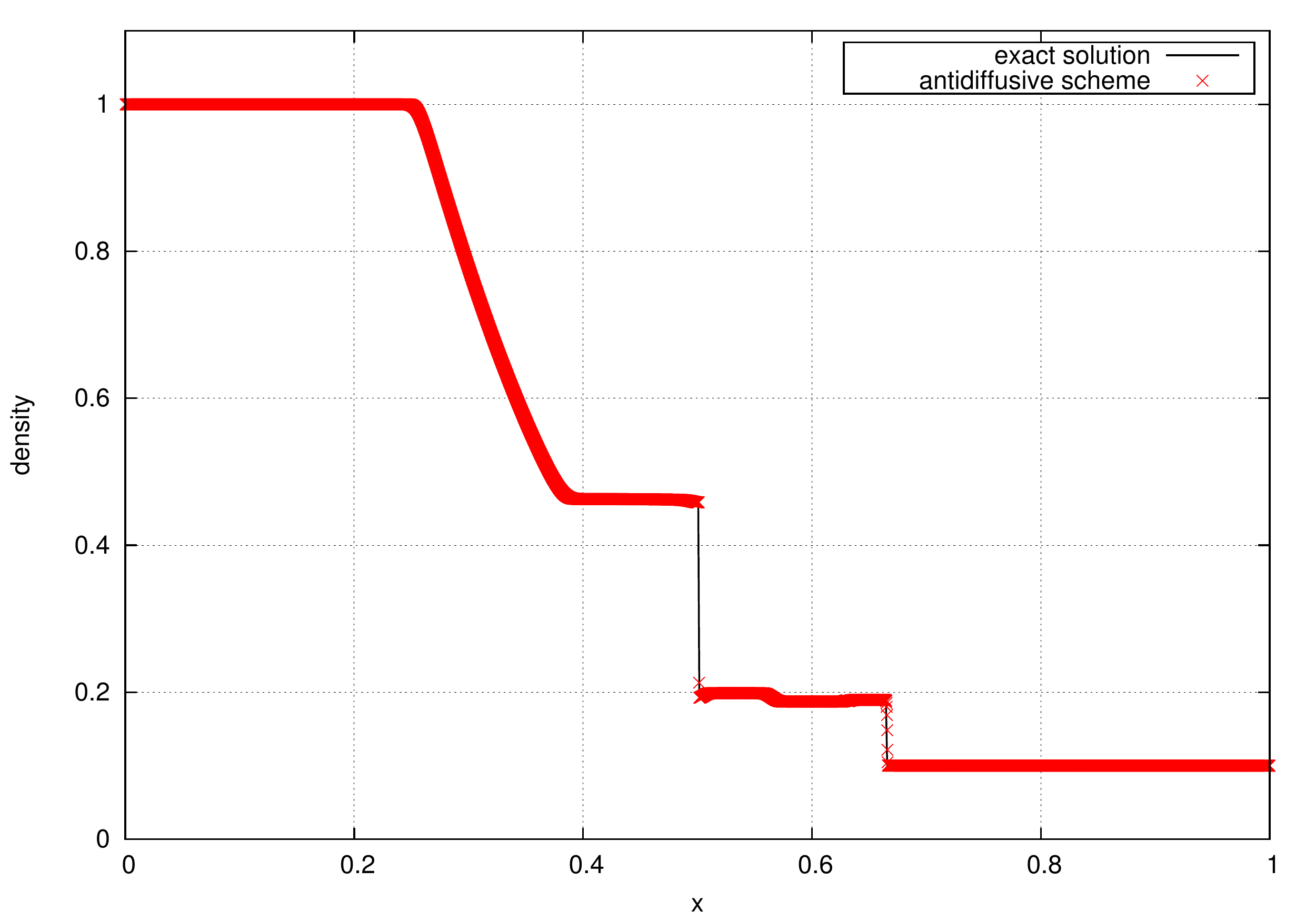} 
&
 \includegraphics[width=\picwidth, clip, trim =
 12mm 8mm 0mm 0mm]{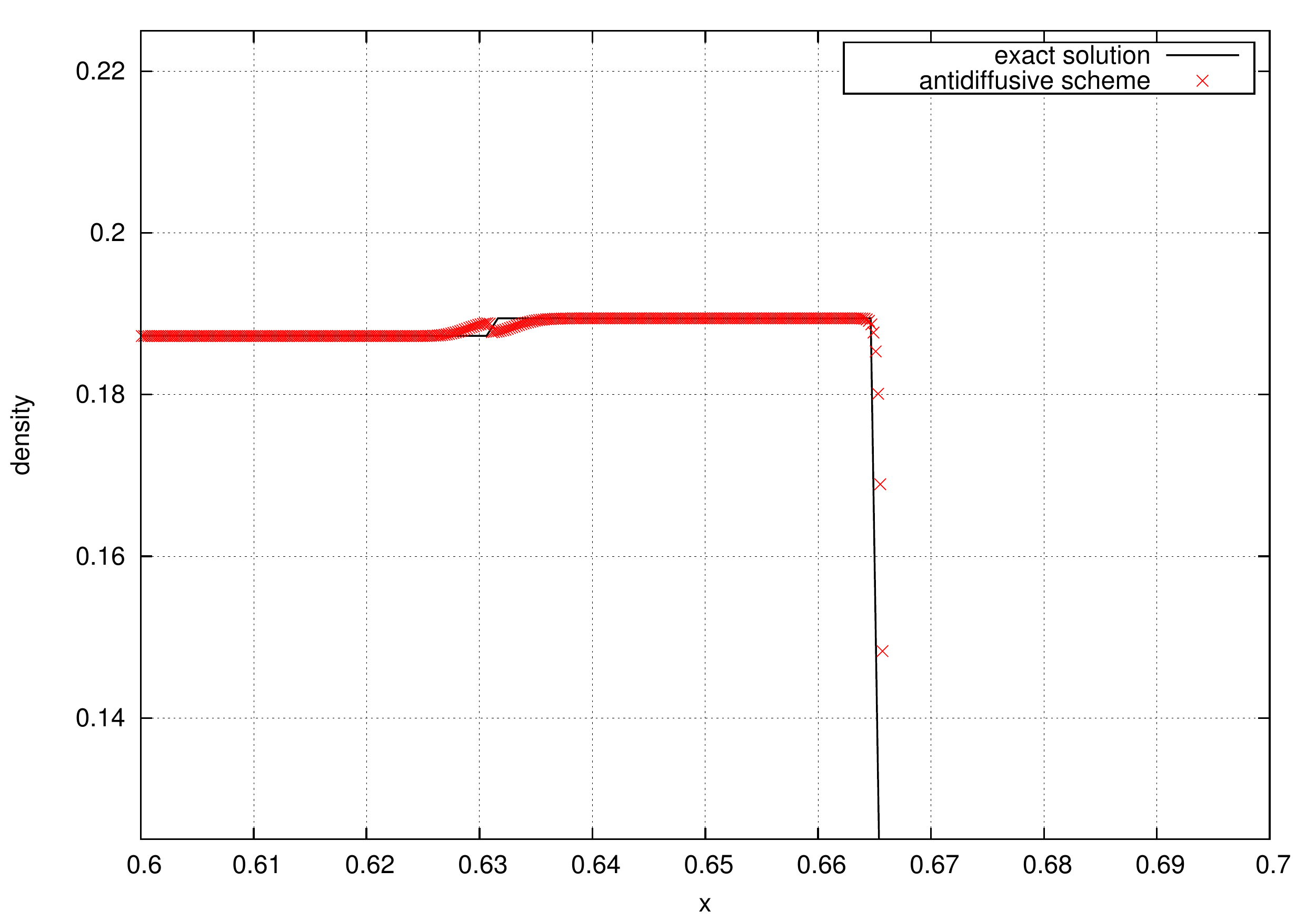}  
\\
Density & Zoom on the density\\
 \includegraphics[width=\picwidth, clip, trim =
 12mm 8mm 0mm 0mm]{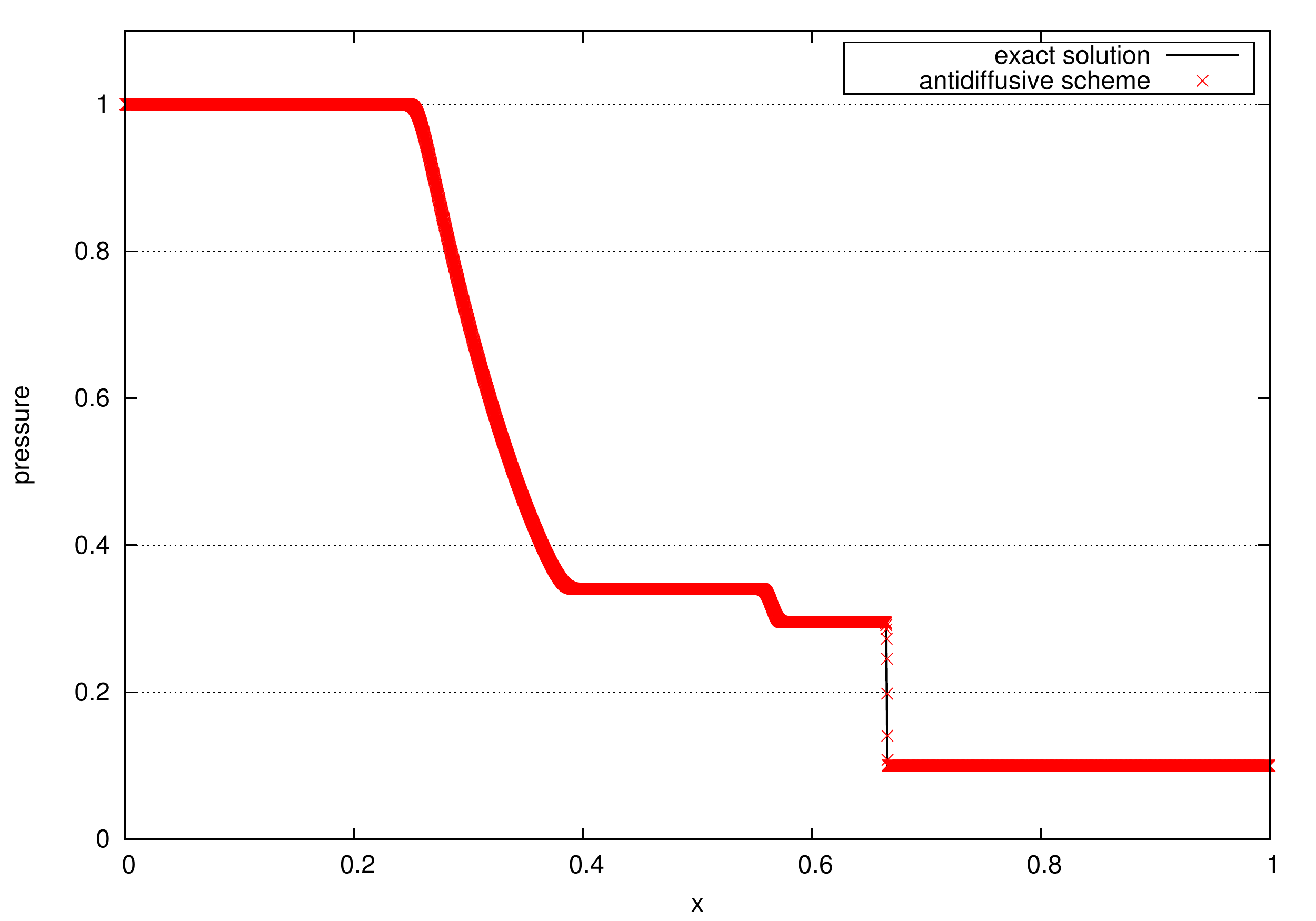}
 &
 \includegraphics[width=\picwidth, clip, trim =
 12mm 8mm 0mm 0mm]{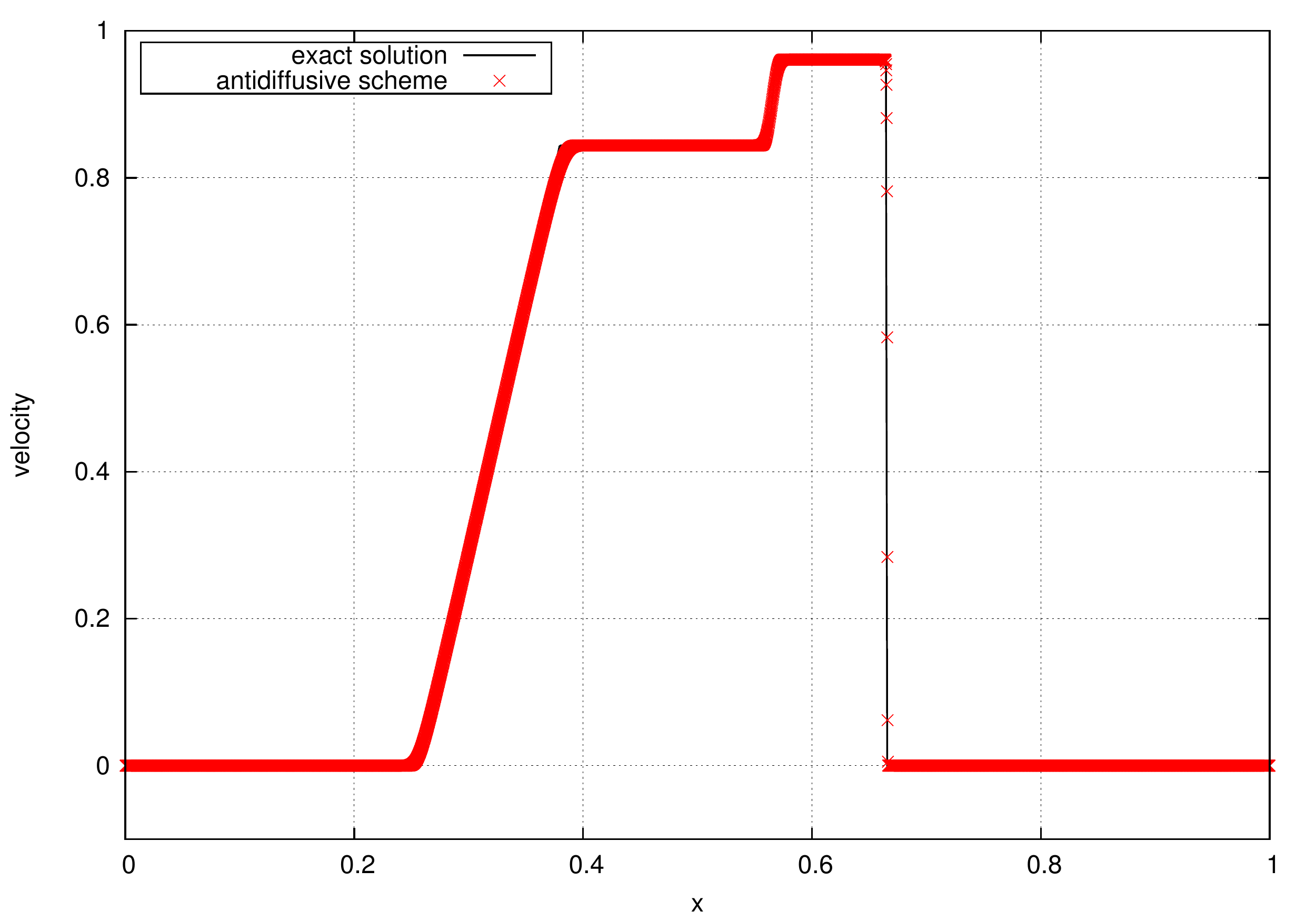} \\
 Pressure & Velocity
\end{tabular}
\caption{test 2, one-dimensional three-material Riemann problems juxtaposition. Profile of the density (top left) and zoom on the second interface (top right), pressure and velocity in the domain at instant $t_\mathrm{end}=0.12\,\mathrm{s}$ for the $50000$-cell mesh.}
\label{fig: shocktube pressure density velocity 50000}
\end{figure}

{ As for the one dimensional five-material passive transport problem, we examine now the influence of the material ordering on the computation of the numerical solution with the anti-diffusive scheme. This time, the following permutation has been considered $ \sigma(1)=2, \sigma(2)=3, \sigma(3)=1$.
To quantify the impact of the fluids ordering change, the absolute and relative difference $e_1,e_2$ for $t_{end}=0.12s$ are computed and their numerical values are given in the tables \ref{tab: TAC e1}-\ref{tab: TAC e2}. Once again, both $e_1$, $e_2$ remain very small (at most $10^{-12}$), demonstrating that the fluid order (and then the computation order of the fluxes) does not impact the numerical solution in a significant manner.\\
}

\begin{table}[]
\centering
\caption{test 2, one-dimensional three-material Riemann problems juxtaposition. Differences $e_1$ for $t_{end}=0.12s$} 
{\small
\begin{tabular}{cccc}
\hline\hline
$a$ & $\rho$ & $p$ & $u$  \\
\hline
$e_1(a)$ &  $2.19~10^{-14}$ &  $9.70^{-15}$ & $3.96~10^{-12}$ \\
\hline\hline
\end{tabular}
}
 \label{tab: TAC e1}
\end{table}

\begin{table}[]
\centering
\caption{test 2, one-dimensional three-material Riemann problems juxtaposition. Differences $e_2$ for $t_{end}=0.12s$} 
{\small
\begin{tabular}{ccccccc}
\hline\hline
$a_k$         & ${\cal Z}_1$ & ${\cal Z}_2$ & ${\cal Z}_3$ & ${\cal Y}_1$ & ${\cal Y}_2$ & ${\cal Y}_3$  \\
\hline
$e_2(a_k)$& $2.79~10^{-14}$&  $2.98~10^{-14}$& $2.99~10^{-14}$ & $6.47~10^{-14}$ & $6.47~10^{-14}$ & $3.09~10^{-14}$ \\
\hline\hline
\end{tabular}
}
 \label{tab: TAC e2}
\end{table}

At last, in order to examine the accuracy and convergence rate of the proposed numerical schemes, we perform several simulations for different meshes with a {number of cells} varying between 100 and 10000. We compute the $L^1$ relative error for each variable $\rho,u,p,\zk,\yk$ to deduce the convergence rates of both upwind and anti-diffusive scheme see table \ref{tab: TAC convergence rates}. Like for the two-material case \cite{Kokh1}, we observe that the accuracy is improved by the anti-diffusive scheme (up to first order) for the variables that are sensitive to contact discontinuities that carry material interfaces.\\

\begin{table}[]
\centering
\caption{test 2, one-dimensional three-material Riemann problems juxtaposition. Convergence rate estimates for all the variables computed  with both upwind scheme and anti-diffusive scheme.} 
{\small
\begin{tabular}{cccccccccc}
\hline\hline
Variable & $\rho$ & $p$ & $u$ & ${\cal Z}_1$ & ${\cal Z}_2$ & ${\cal Z}_3$
& ${\cal Y}_1$ & ${\cal Y}_2$ & ${\cal Y}_3$ \\
\hline
Upwind &  0.646 &  0.776 & 0.796 & 0.545 & 0.535 & 0.519
& 0.483 & 0.493 & 0.517\\
Anti-Diff. & 0.786 &  0.783 & 0.807 & 1.002 & 1.031 & 1.112
&1.003 &1.032 & 1.101\\
\hline\hline
\end{tabular}
}
 \label{tab: TAC convergence rates}
\end{table}

{
\subsection{Test 3: One-dimensional three-material Riemann problems juxtaposition involving high pressure ratios}
We now test the proposed scheme against a three-material Riemann problem with high pressure ratios. We consider again a one-dimensional 1 m long domain occupied by three materials separated by two interfaces localized at $x=0.75~\mathrm{m}$ and $x=0.95~\mathrm{m}$. At the initial time, the flow is at rest and the first region $0<x<0.75$ is occupied by a stiffened gas while the two remaining regions respectively $0.75 \le x \le 0.95$ and $0.95<x<1$ are occupied by two different perfect gases. The EOSs parameters as well as initial states for this test are provided in table~\ref{tab: shock tube 2 init data}. The computational domain is discretized over a {2000-cell} mesh. As the acoustic solver is only first order accurate the mesh is chosen sufficiently fine to capture the waves dynamic especially in the smaller regions of the domain occupied by the perfect gases.  We perform simulations with both anti-diffusive and upwind schemes using $C_\mathrm{CFL}=0.8$ until $t_{end}=270~\mu\mathrm{s}$. Transparent boundary conditions are used at each boundary of the domain.\\
The wave pattern for this test is the same as the one depicted in figure \ref{fig:struct_sol}. Indeed, the initial pressure ratio between the two first material is $10^4$ which generates a strong shock that propagates towards the right at velocity $D_1=819.92~\mathrm{m.s}^{-1}$, and a rarefaction moving to the left. At time $t_{shock}=26.07~\mu\mathrm{s}$ this shock reaches the second interface (as depicted in figure \ref{fig:struct_sol}) and generates a new shock moving two the right at the speed $D_2= 1271~\mathrm{m.s}^{-1}$ and a rarefaction wave traveling leftwards between both interfaces. 
The profiles of the pressure, velocity and density computed with both schemes are depicted in figure \ref{fig: shocktube 2 pressure density velocity} along with the exact solution. Both numerical solutions are in good agreement with the exact solution and illustrate the ability  of the proposed scheme to deal with high pressure ratio. For this test anti-diffusive scheme provide approximations for the color functions and mass fractions that remain very sharp (see figure \ref{fig: shocktube 2 z} and \ref{fig: shocktube 2 y} respectively) and that verifiy the maximum principle and the discrete unit constraint. 
 }
\begin{table}
\centering
\caption{test 3, one-dimensional three-material Riemann problems juxtaposition involving high pressure ratios. Initial values for pressure and density in the domain.}
\begin{tabular}{ccccccc}
\hline\hline
 location & $k$ & $\rho$ \scriptsize{$(\mathrm{kg}.\mathrm{m}^{-3})$}  & $p$ \scriptsize{$(\mathrm{Pa})$} & $\gamma_k$ & $\pi_k$  \\
\hline 
$0\le x<0.75$          & $1$ & $1000$ & $1~10^9$ & $4.4$ & $6~10^8$\\   
$0.75\le x \le 0.95$ & $2$ & $ 50$    & $1~10^5$ & $2.4$ & $\times$\\ 
$0.95<x\le1$          & $3$ & $ 1    $   & $1~10^5$ & $1.4$ & $\times$\\ 
\hline\hline
\end{tabular}
 \label{tab: shock tube 2 init data}
\end{table}

\begin{figure}
 \centering 
\begin{tabular}{cc}
  \includegraphics[width=\picwidth, clip, trim =
 12mm 8mm 0mm 0mm]{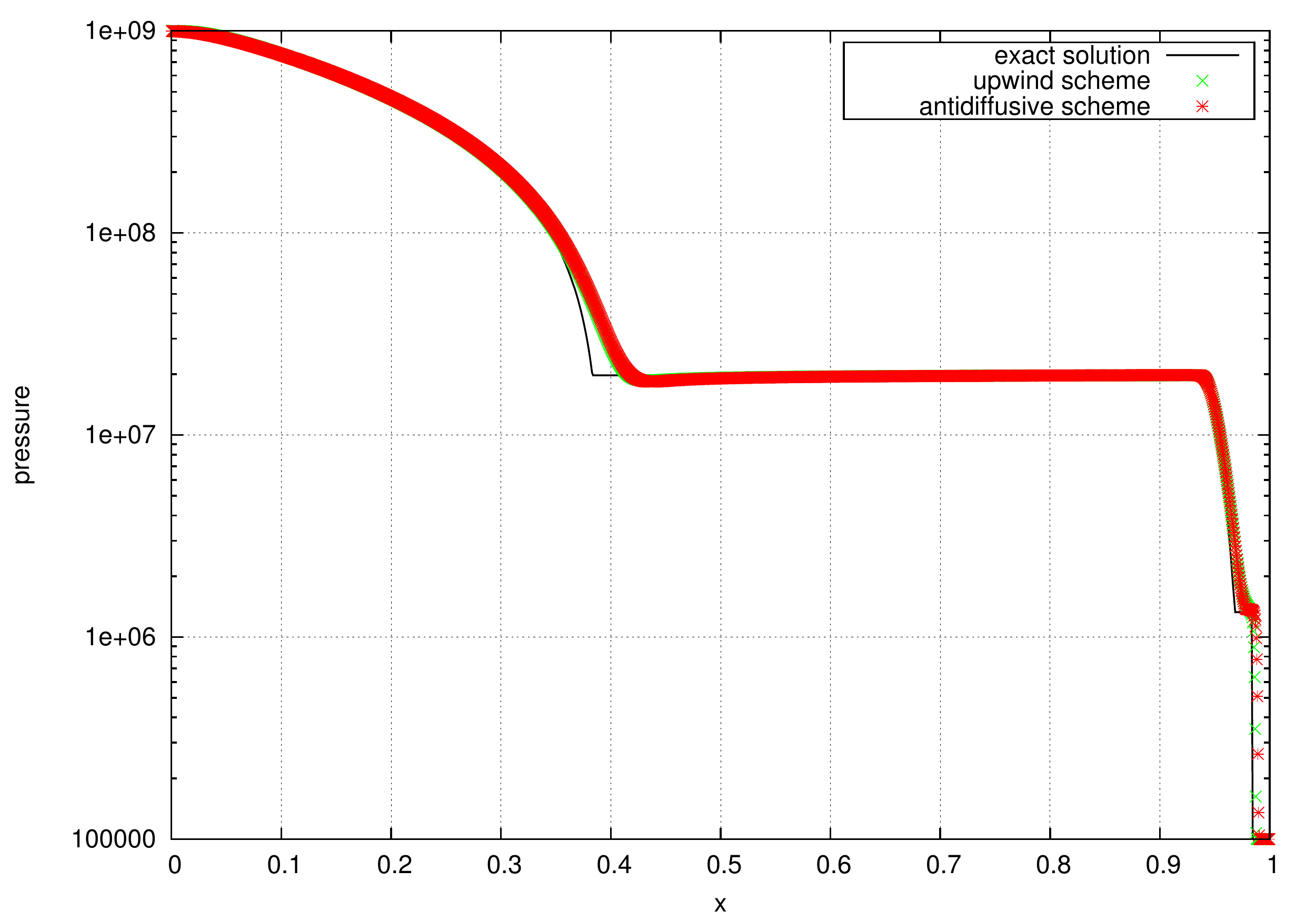}
 &
 \includegraphics[width=\picwidth, clip, trim =
 12mm 8mm 0mm 0mm]{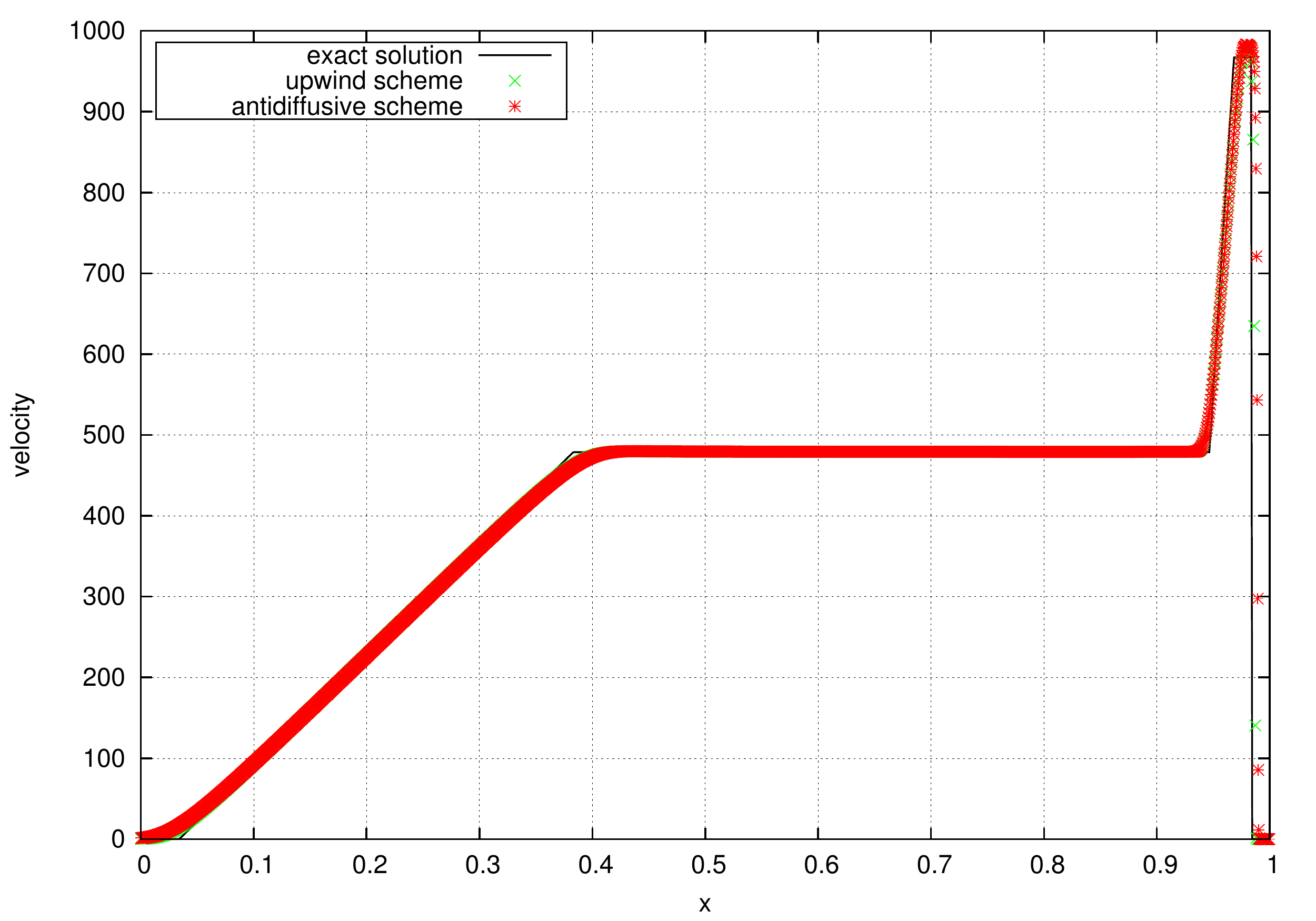} \\
 Pressure & Velocity
\end{tabular}
 \includegraphics[width=\picwidth, clip, trim =
 12mm 8mm 0mm 0mm]{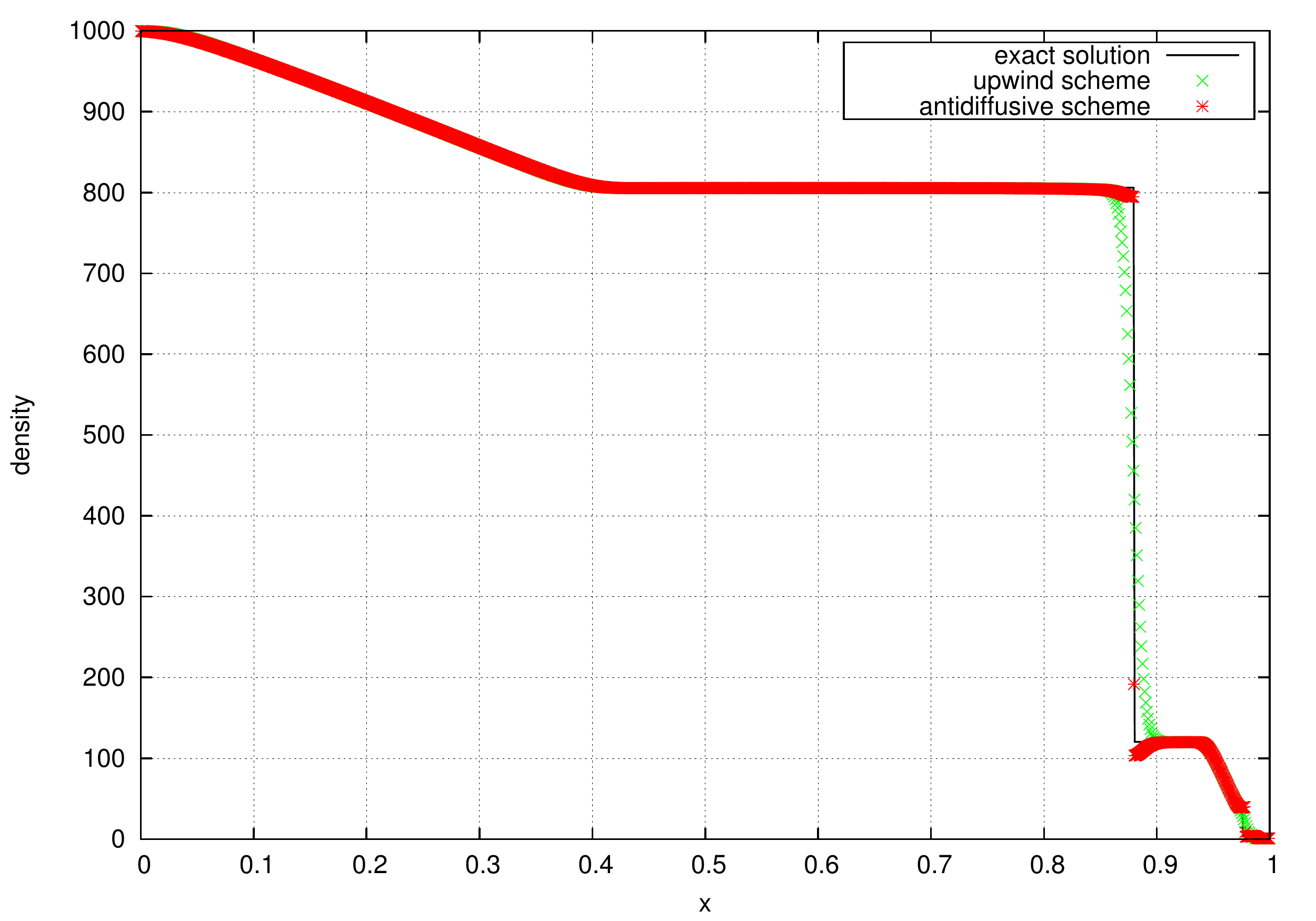}\\
 Density
\caption{test 3, one-dimensional three-material Riemann problems juxtaposition involving high pressure ratios. Profile of the density, pressure (semi-log scale) and velocity in the domain at instant $t_\mathrm{end}=1.2~10^{-4}\,\mathrm{s}.$}
\label{fig: shocktube 2 pressure density velocity}
\end{figure}
\begin{figure}
 \centering
\begin{tabular}{cc}
 \includegraphics[width=\picwidth, clip, trim =
 12mm 8mm 0mm 0mm]{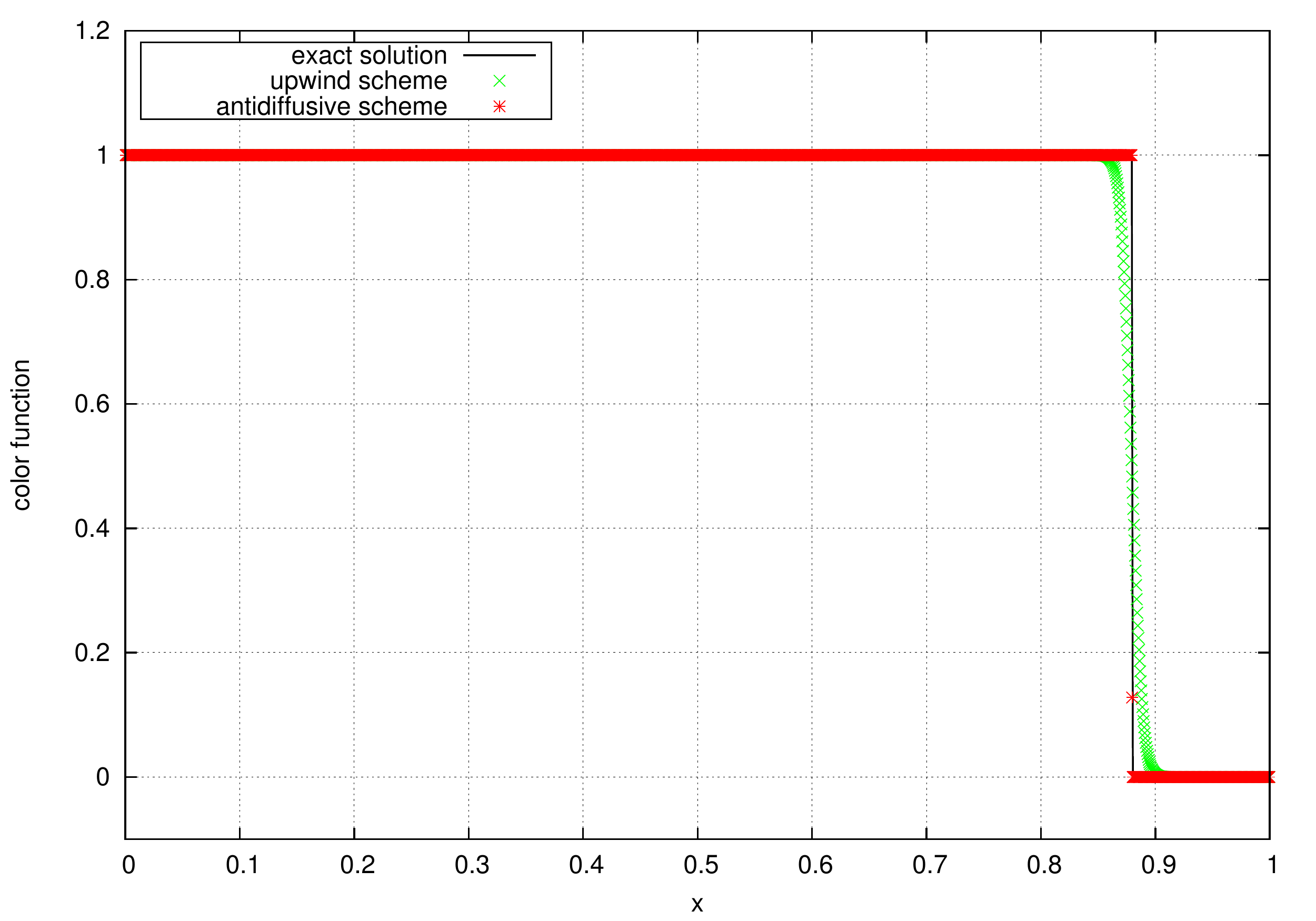} 
&
 \includegraphics[width=\picwidth, clip, trim =
 12mm 8mm 0mm 0mm]{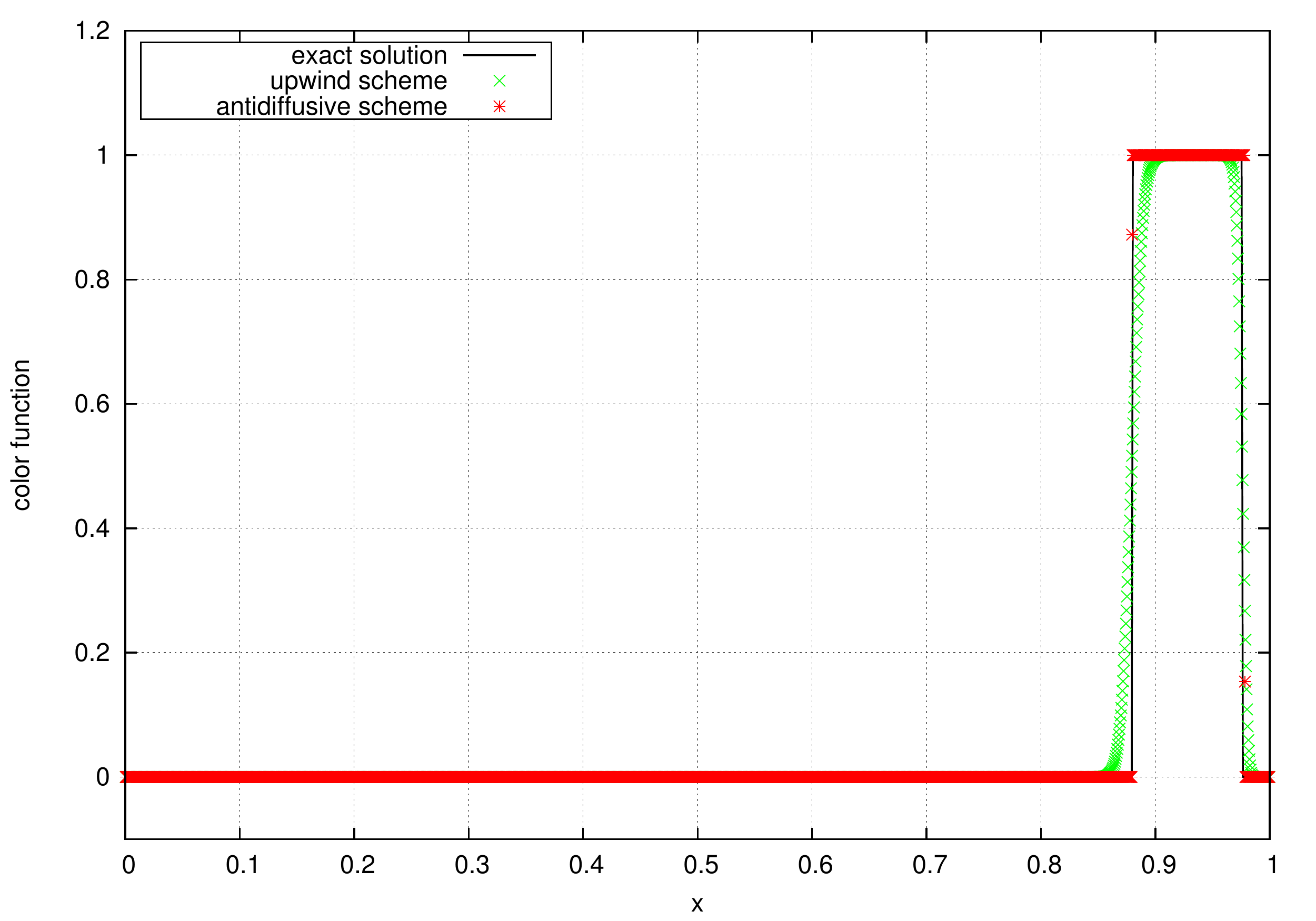} 
\\
${\cal Z}_1$ & ${\cal Z}_2$\\
 \includegraphics[width=\picwidth, clip, trim =
 12mm 8mm 0mm 0mm]{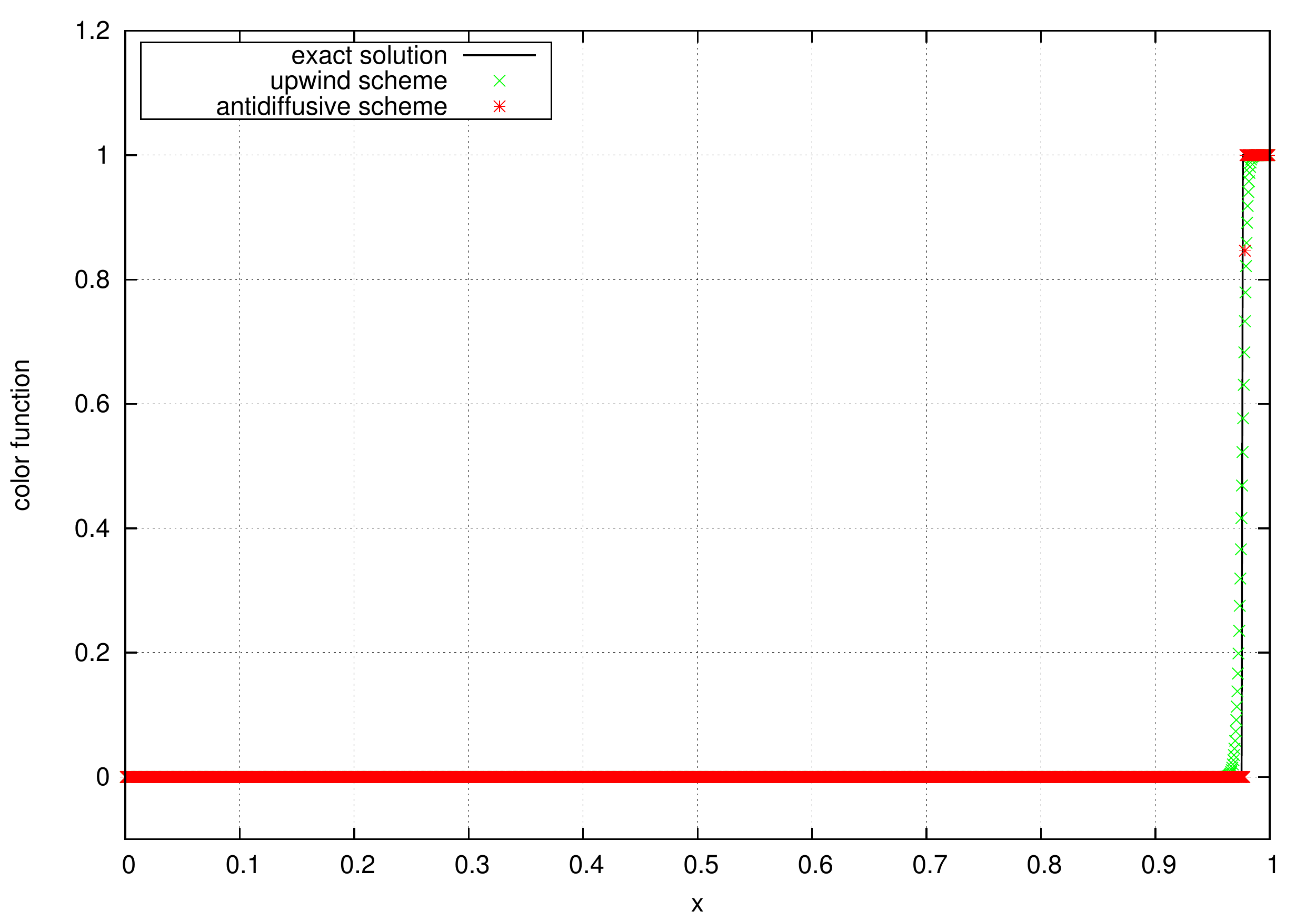} 
&
 \includegraphics[width=\picwidth, clip, trim =
 12mm 8mm 0mm 0mm]{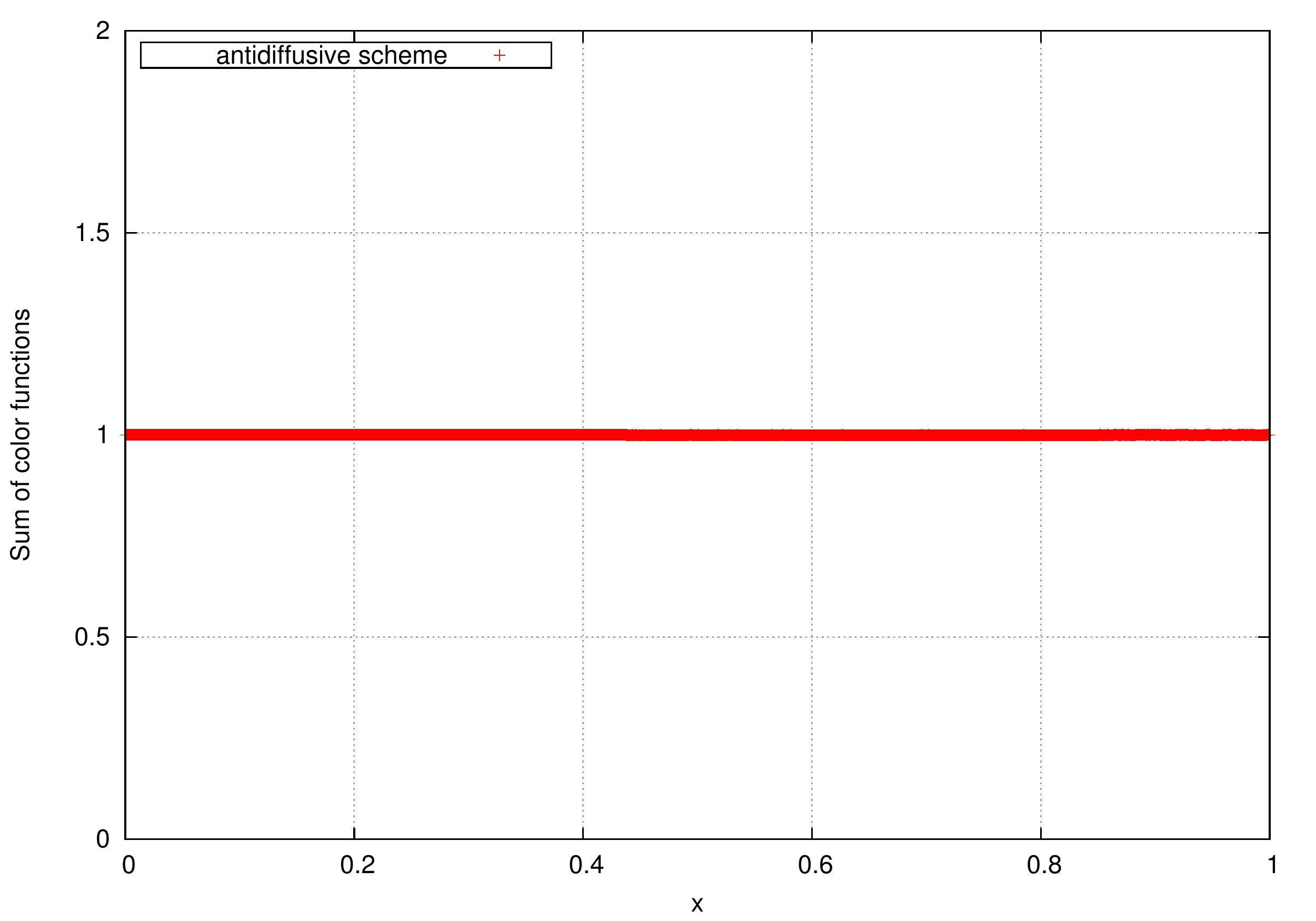} \\
 ${\cal Z}_3$ & $\z_1+\z_2+\z_3$\\
\end{tabular}
\caption{test 3, one-dimensional three-material Riemann problems juxtaposition involving high pressure ratios. Profile of the color functions $\z_k$, $k=1,2,3$ and of $\z_1+\z_2+\z_3$ at instant $t_\mathrm{end}=1.2~10^{-4}\,\mathrm{s}$.}
\label{fig: shocktube 2 z}
\end{figure}
\begin{figure}
 \centering
\begin{tabular}{cc}
 \includegraphics[width=\picwidth, clip, trim =
 12mm 8mm 0mm 0mm]{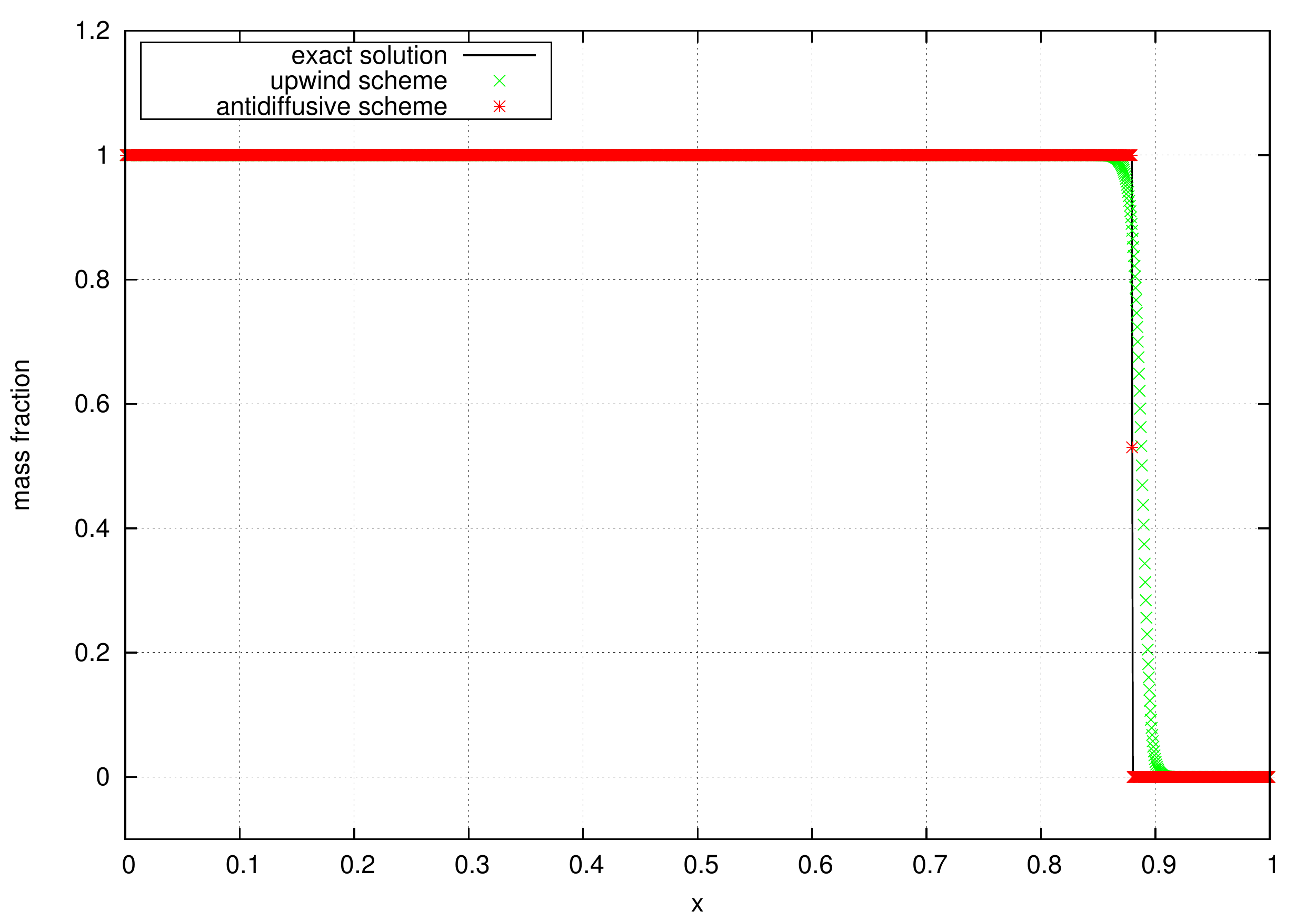} 
&
 \includegraphics[width=\picwidth, clip, trim =
 12mm 8mm 0mm 0mm]{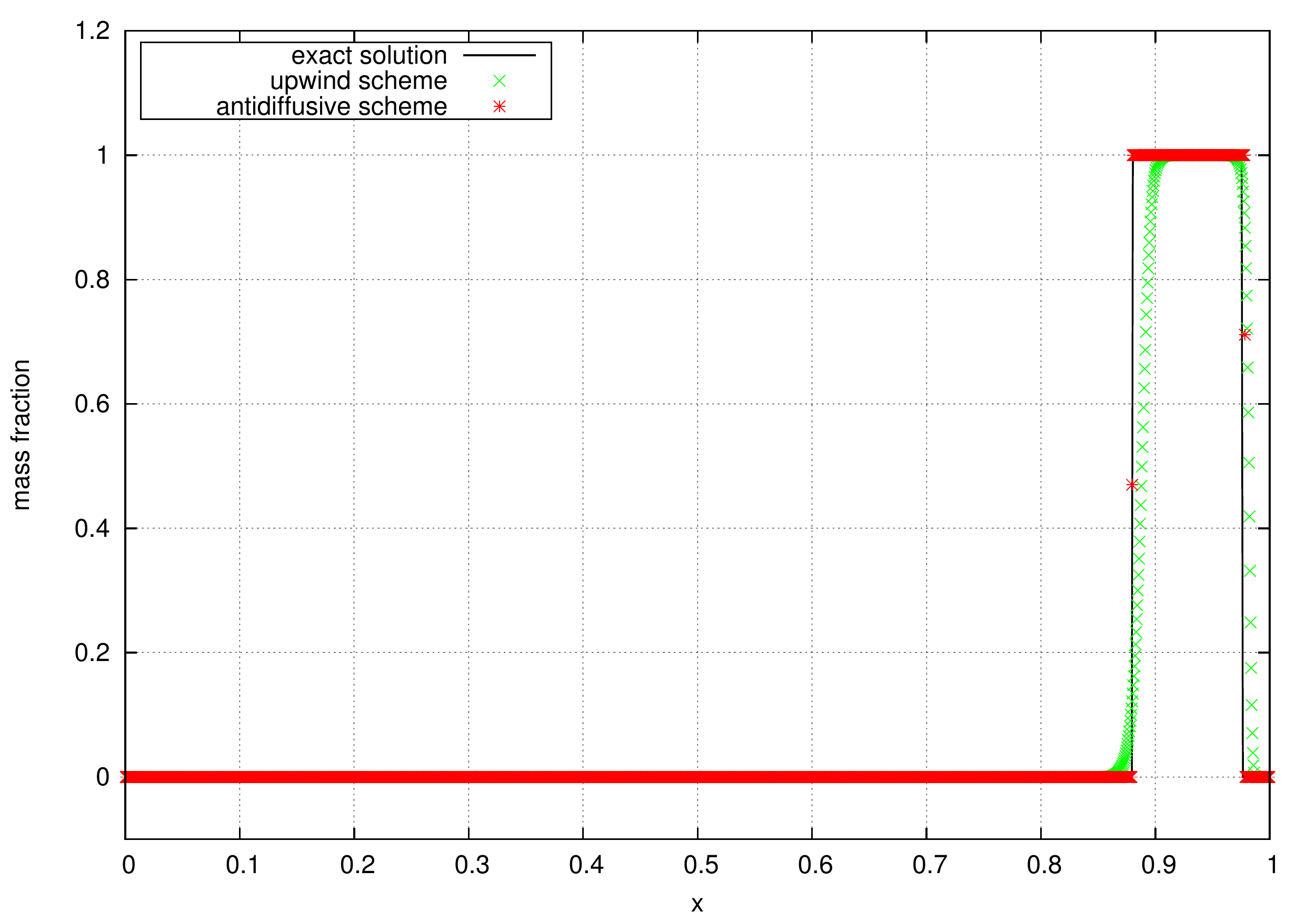} 
\\
${\cal Y}_1$ & ${\cal Y}_2$\\
 \includegraphics[width=\picwidth, clip, trim =
 12mm 8mm 0mm 0mm]{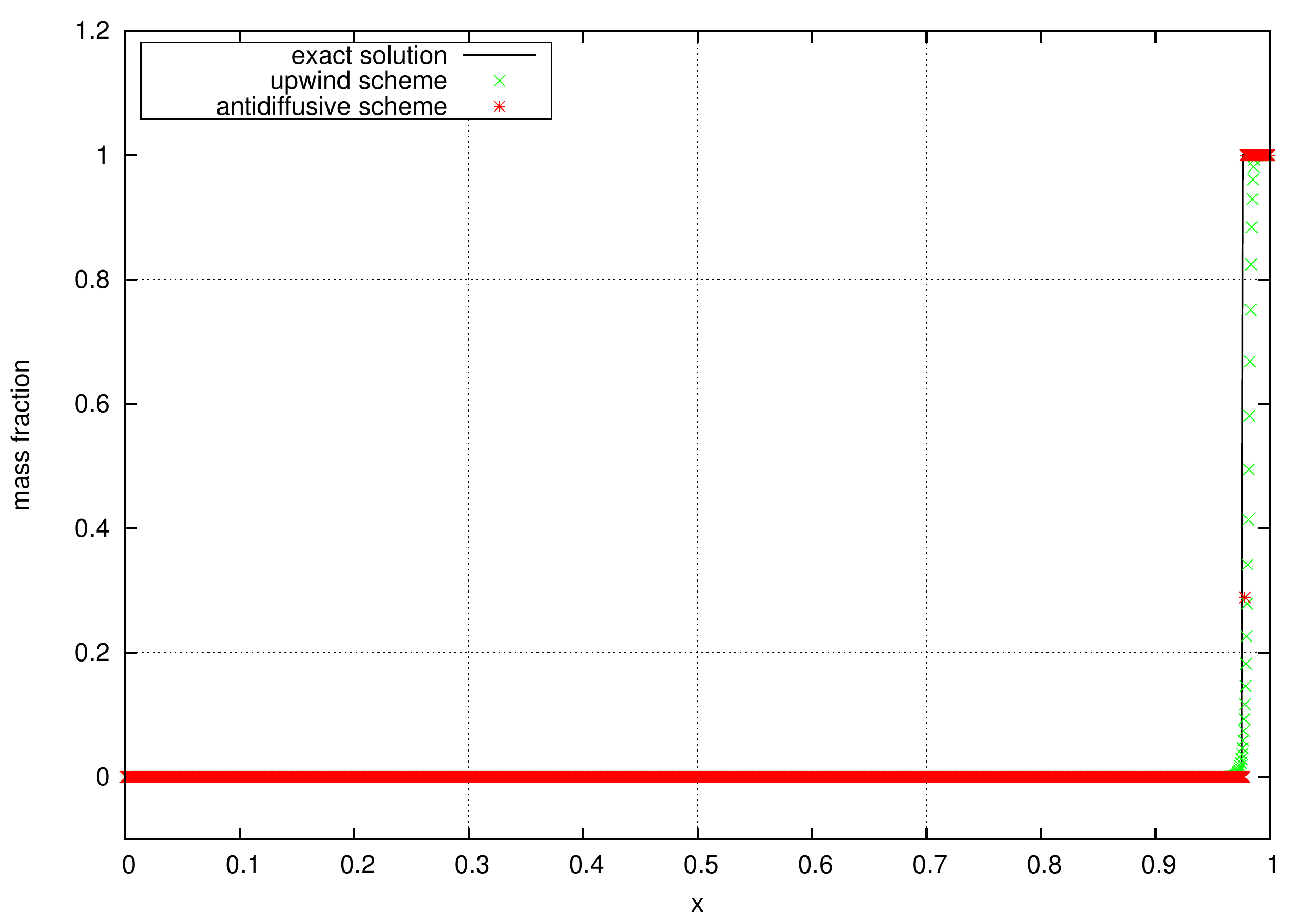} 
&
 \includegraphics[width=\picwidth, clip, trim =
 12mm 8mm 0mm 0mm]{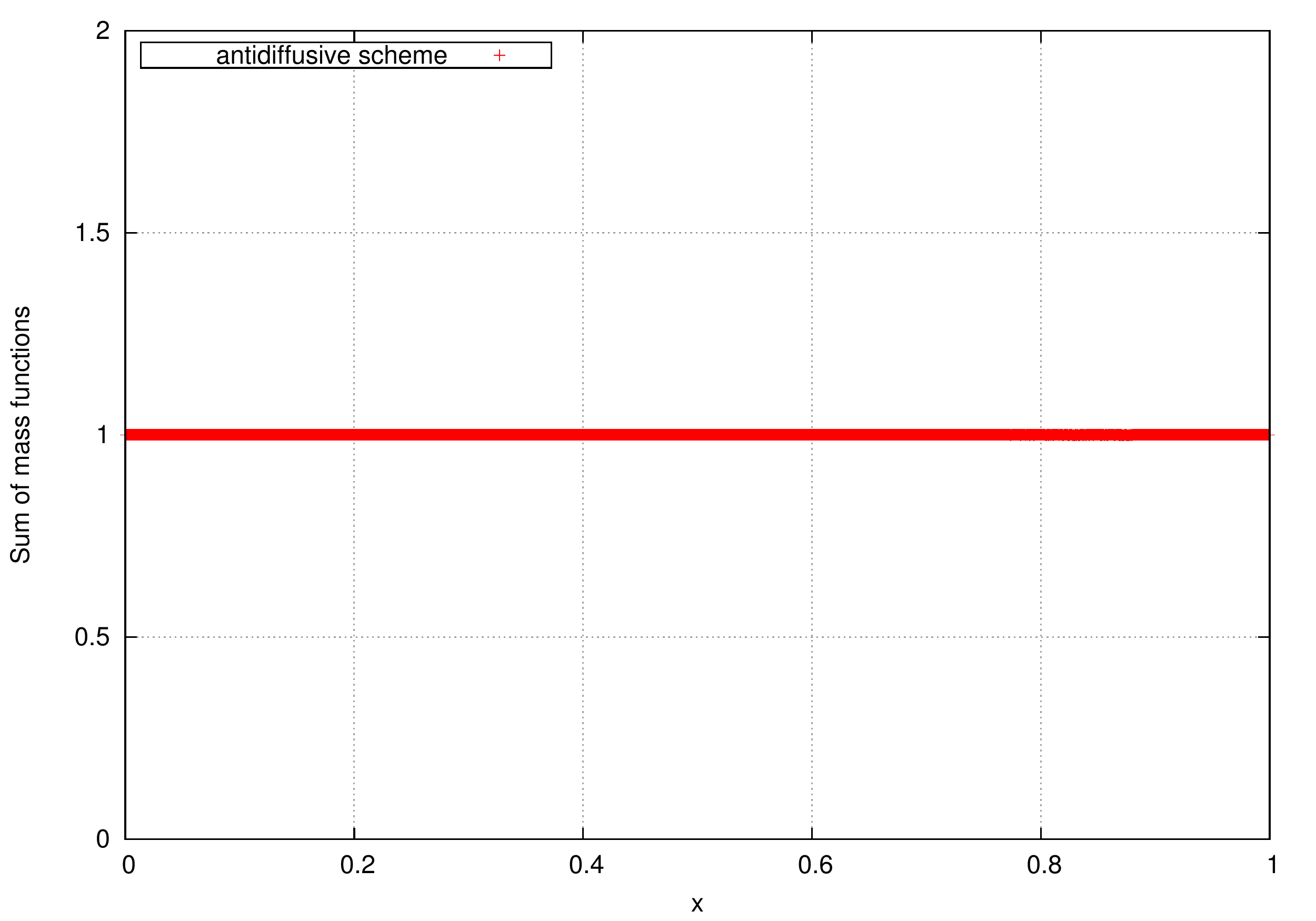} \\
  ${\cal Y}_3$ & $\y_1+\y_2+\y_3$\\
\end{tabular}
\caption{test 3, one-dimensional three-material Riemann problems juxtaposition involving high pressure ratios. Profile of the mass fractions $\y_k$, $k=1,2,3$ and of $\y_1+\y_2+\y_3$ at instant $t_\mathrm{end}=1.2~10^{-4}\,\mathrm{s}$.}
\label{fig: shocktube 2 y}
\end{figure}

\subsection{Test 4: Two-dimensional passive transport with four components}
We aim here at testing the ability of our numerical scheme to accurately transport 
material interfaces with a complex topology. The computational domain $D$ is the square 
$D=[0,60]\times[0,60]$. We suppose that this domain is occupied by four perfect gases numbered $k=1,\dots,4$. 
At $t=0$, the velocity and the pressure are uniform in the domain, their value are respectively $(u_1,u_2)=(\sqrt{2}\,\mathrm{m}/\mathrm{s},\sqrt{3}\,\mathrm{m}/\mathrm{s})$, $p=1\,\mathrm{Pa}$. We consider the subsets 
$C$, $S$, $P$ of $D$ defined by
$$
\begin{aligned}
P &= \left\{ (x_1,x_2) \in D; (x_1,x_2) \in [27.5,32.5]^2 \right \}
,\\
C &= \left\{ (x_1,x_2) \in D; (x_1-30)^2+(x_2-30)^2 \leq 15^2 \right \}
,\\
S &= \left\{ (x_1,x_2) \in D; 37.5>x_2>22.5, \sqrt 3 (x_1-30)+15 <x_2<\sqrt 3 (x_1-30)+45, 
\right. 
\\
&\left. \qquad-\sqrt 3 (x_1-30)+15<x_2<-\sqrt 3 (x_1-30)+45 \right\}.
\end{aligned}
$$
The regions $C$, $S$, $P$ are respectively a disk, a star shaped and a square portion of $D$.
The EOS parameters of these fluids are given in table~\ref{tab: 2D transport} along with their initial location within the domain and density value. The domain $D$ is discretized over a  $200\times200$-cell regular mesh and we impose periodic boundary conditions. The resulting initial condition for the material interface and the density is depicted in figure~\ref{fig:4mat initial instant}.
\begin{table}[]
\centering
\caption{test 4, two-dimensional passive transport with four components. Initial location and EOS parameters of the components.} 
\begin{tabular}{cccc}
\hline\hline
material index &$\gamma_k$ & initial location & initial density \scriptsize{$(\mathrm{kg}/\mathrm{m}^3)$}\\
\hline
$k=1$ & $2.2$ & $D\setminus C$ & 0.01\\ 
$k=2$ & $1.6$  & $C \setminus S$ & 0.1\\ 
$k=3$ & $1.4$  & $S\setminus P$ & 1.0\\ 
$k=4$ & $1.2$  & $P$ & 10.0\\
\hline\hline
\end{tabular}
 \label{tab: 2D transport}
\end{table}
Figure~\ref{fig:4mat z} displays a mapping of $\sum_{k=1}^4 k \z_k$ at $t=42.5\,\mathrm{s}$ obtained with both the upwind scheme and the anti-diffusive scheme. The final instant is reached after 2626 time steps for the upwind scheme and 2627 time steps for the anti-diffusive scheme. We plot in figures \ref{fig:4mat z} and \ref{fig:4mat density}  the density profiles and the interfaces by representing $\sum_{k=1}^4 k{\cal Z}_k$ at the final time $t=42.5\,\mathrm{s}$. As in the one dimensional case, the interfaces remain quite sharp and their original topology is well preserved by the numerical approximation computed by the anti-diffusive scheme. With the upwind scheme, the interface shapes are completely spread due to numerical diffusion. Nevertheless, both schemes preserve velocity and pressure profiles as in the one-dimensional test. Indeed, at the final time the relative errors $||\overline{\boldsymbol{u}}-\boldsymbol{u}||_{L^1}$ and $||p-\overline{p}||_{L^1}$ for the references $\overline{\boldsymbol{u}}=\sqrt{5}$ and $\overline{p}=1$ remain bounded between $3.9\times 10^{-16}$ and $4.9\times 10^{-14}$ (see table \ref{tab: error 2D transport} for more details). Finally, figure \ref{fig:4mat unity constraint} displays the graphs $t \mapsto \int_{D} \sum_{k=1}^4 {\cal Z}_k(\bx,t) dx$ and  $t \mapsto \int_{D} \sum_{k=1}^4 {\cal Y}_k(\bx,t) dx$ for the anti-diffusive scheme. Again, we observe that the scheme preserve the unity constraints \eqref{eq:Z unity} for both color functions and mass fractions at each time step.

%
%
%
%
%
% \begin{figure}[]
% \centering
% \begin{tabular}{cc}
% \includegraphics[scale=1, trim = 34mm 15mm 34mm 0mm, clip]{./RES/2D/INT4/mat_up_fin.png}
% &
% \includegraphics[scale=1, trim = 34mm 15mm 34mm 0mm, clip]{./RES/2D/INT4/mat_anti_fin.png}
% \end{tabular}
% \caption{Two-dimensional passive transport with four components. Mapping of the color functions obtained with both upwind (left) and anti-diffusive (right) schemes at $t=6\,\mathrm{s}$. \label{fig:4mat}}
% \end{figure}
%
%

\begin{figure}[]
\centering
\begin{tabular}{cc}
\includegraphics[scale=0.4 , trim = 55mm 30mm 55mm 30mm, clip]{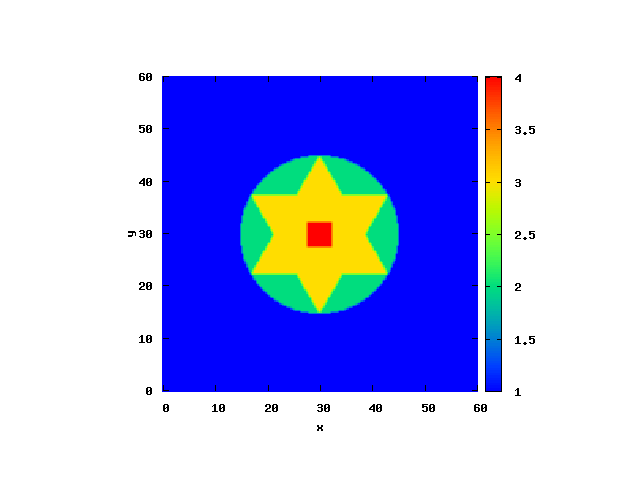}
&
\includegraphics[scale=0.4 , trim = 55mm 30mm 55mm 30mm, clip]{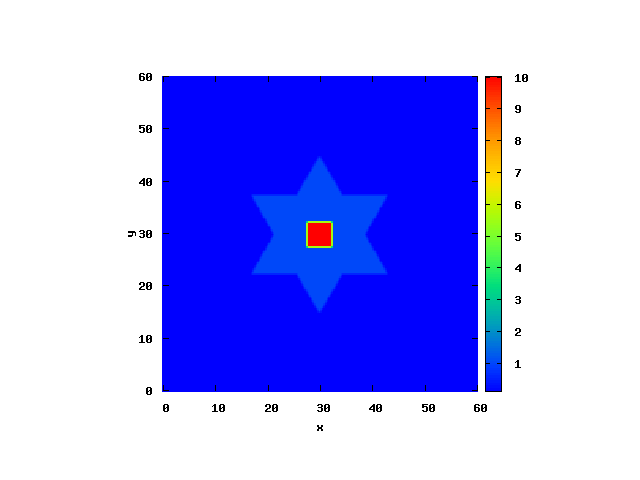}
\end{tabular}
\caption{test 4, two-dimensional passive transport with four components. Mapping of $\sum_{k=1}^4 k\z_k$ (left) and the density (right) at $t=0$. 
\label{fig:4mat initial instant}
}
\end{figure}
\begin{figure}[]
\centering
\begin{tabular}{cc}
% \includegraphics[scale=1, trim = 34mm 15mm 34mm 0mm, clip]{figures/2D-advec/mat_up_middle.pdf}
% &
% \includegraphics[scale=1, trim = 34mm 15mm 34mm 0mm, clip]{./figures/2D-advec/mat_anti_middle.pdf}
% \\
\includegraphics[scale=0.4 , trim = 55mm 30mm 55mm 30mm, clip]{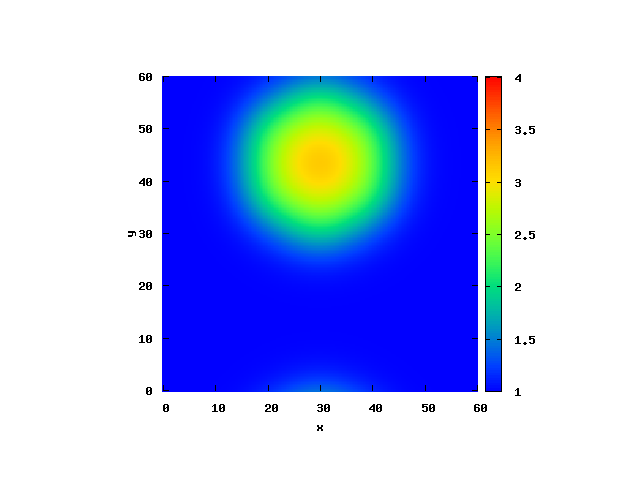}
&
\includegraphics[scale=0.4 , trim = 55mm 30mm 55mm 30mm, clip]{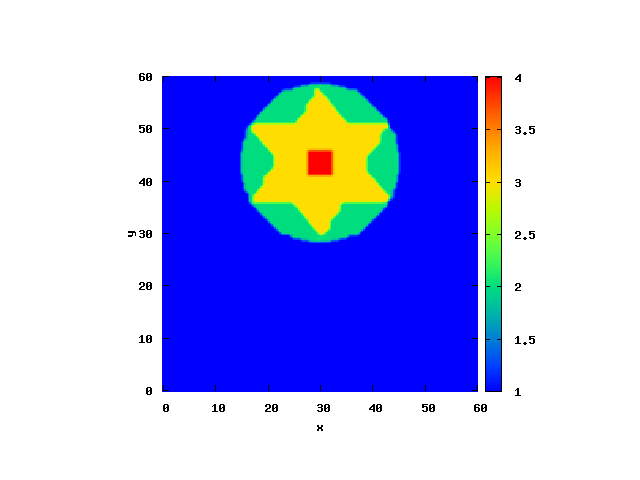}
\\
upwind
& anti-diffusive
\end{tabular}
\caption{test 4, two-dimensional passive transport with four components. Mapping of $\sum_{k=1}^4 k\z_k$ at $t=42.5\,\mathrm{s}$ obtained with both upwind scheme and the anti-diffusive scheme.
\label{fig:4mat z}
}
\end{figure}
\begin{figure}[]
\centering
\begin{tabular}{cc}
\includegraphics[scale=0.4 , trim = 55mm 30mm 55mm 30mm, clip]{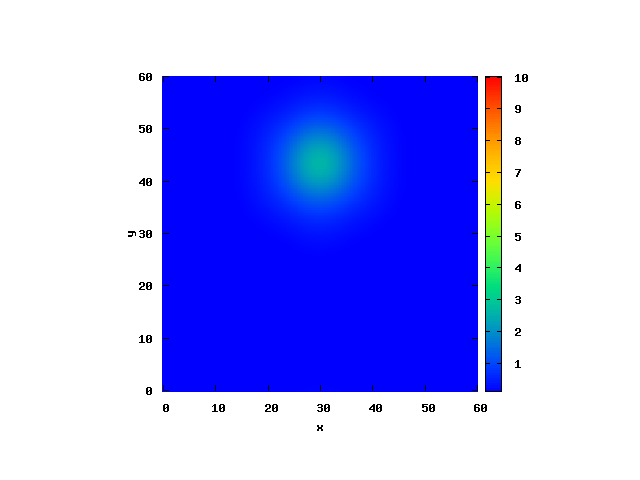}
&
\includegraphics[scale=0.4 , trim = 55mm 30mm 55mm 30mm, clip]{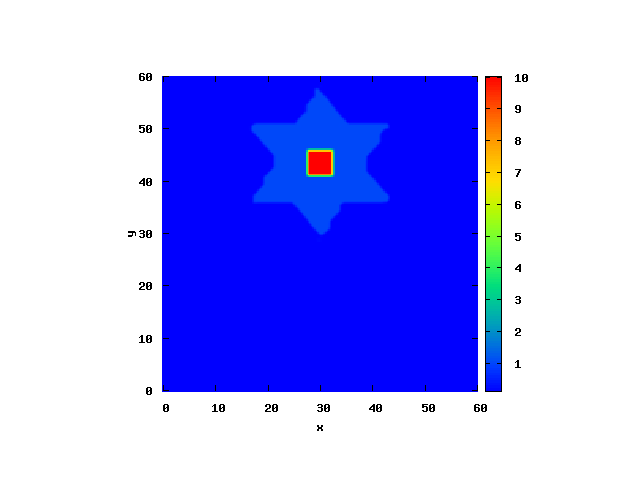}
\\
upwind
& anti-diffusive
\end{tabular}
\caption{test 4, two-dimensional passive transport with four components. Mapping of the density at $t=42.5\,\mathrm{s}$ obtained with both upwind scheme and the anti-diffusive scheme.
\label{fig:4mat density}
}
\end{figure}

\begin{table}[]
\centering
\caption{test 4, two-dimensional passive transport with four components. Relative error of $p$ and $\bu$ in $L^1$ norm to the constant profile $\overline{p}=1$ and $\overline{\bu} = \sqrt{5}$ at final time $t=42.5\,\mathrm{s}$.} 
\begin{tabular}{cccc}
\hline\hline
Scheme    &$||p-\overline{p}||_{L^1}$ & $||\bu-\overline{\bu}||_{L^1}$\\
\hline
Anti-diffusive & $4.88\times 10^{-14}$ & $ 3.97 \times 10^{-16}$\\ 
Upwind     & $3.99\times 10^{-14}$  & $1.19\times 10^{-15} $\\
\hline\hline
\end{tabular}
 \label{tab: error 2D transport}
\end{table}

\begin{figure}[]
\centering
\begin{tabular}{cc}
\includegraphics[scale=0.225]{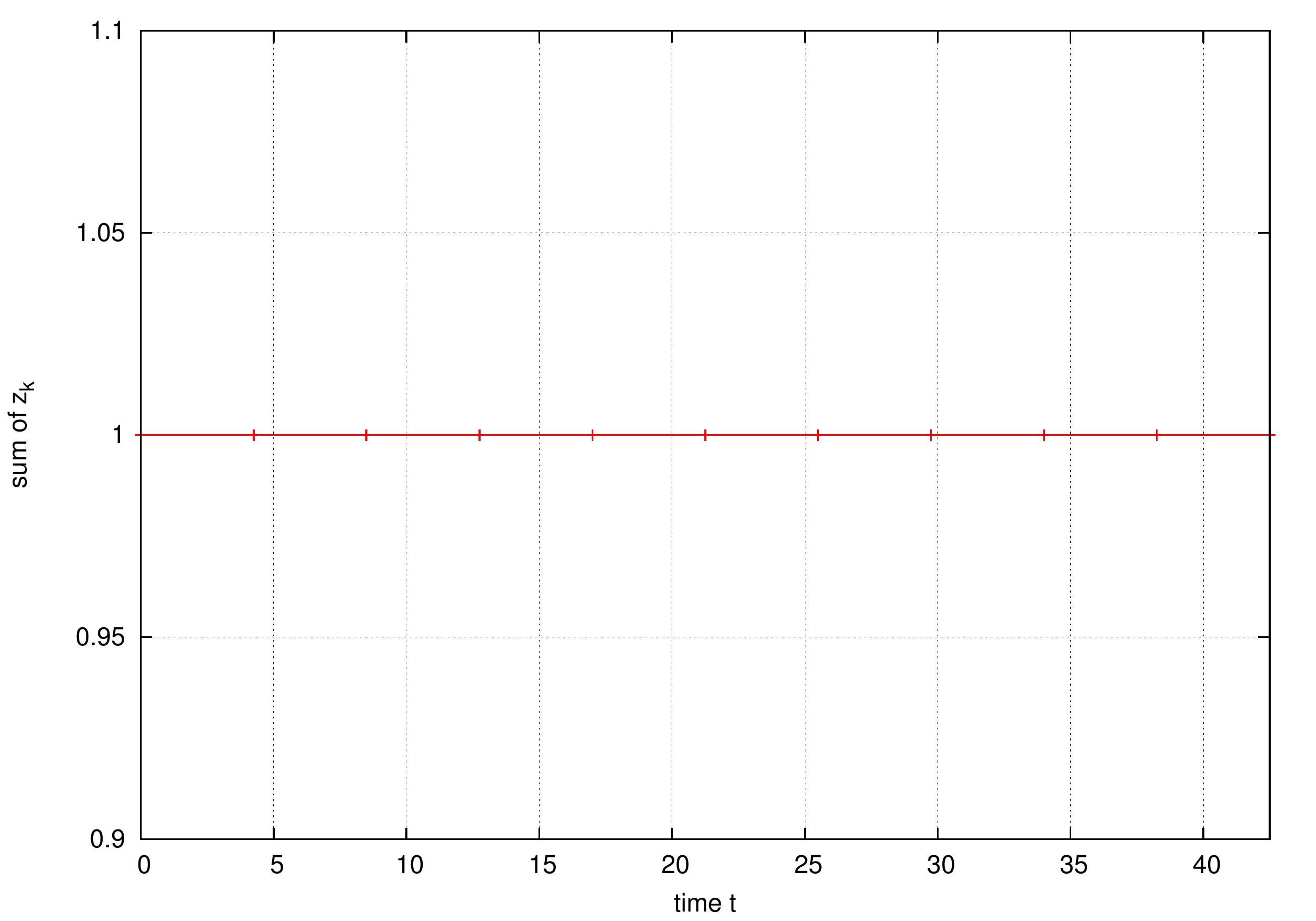}
&
\includegraphics[scale=0.225]{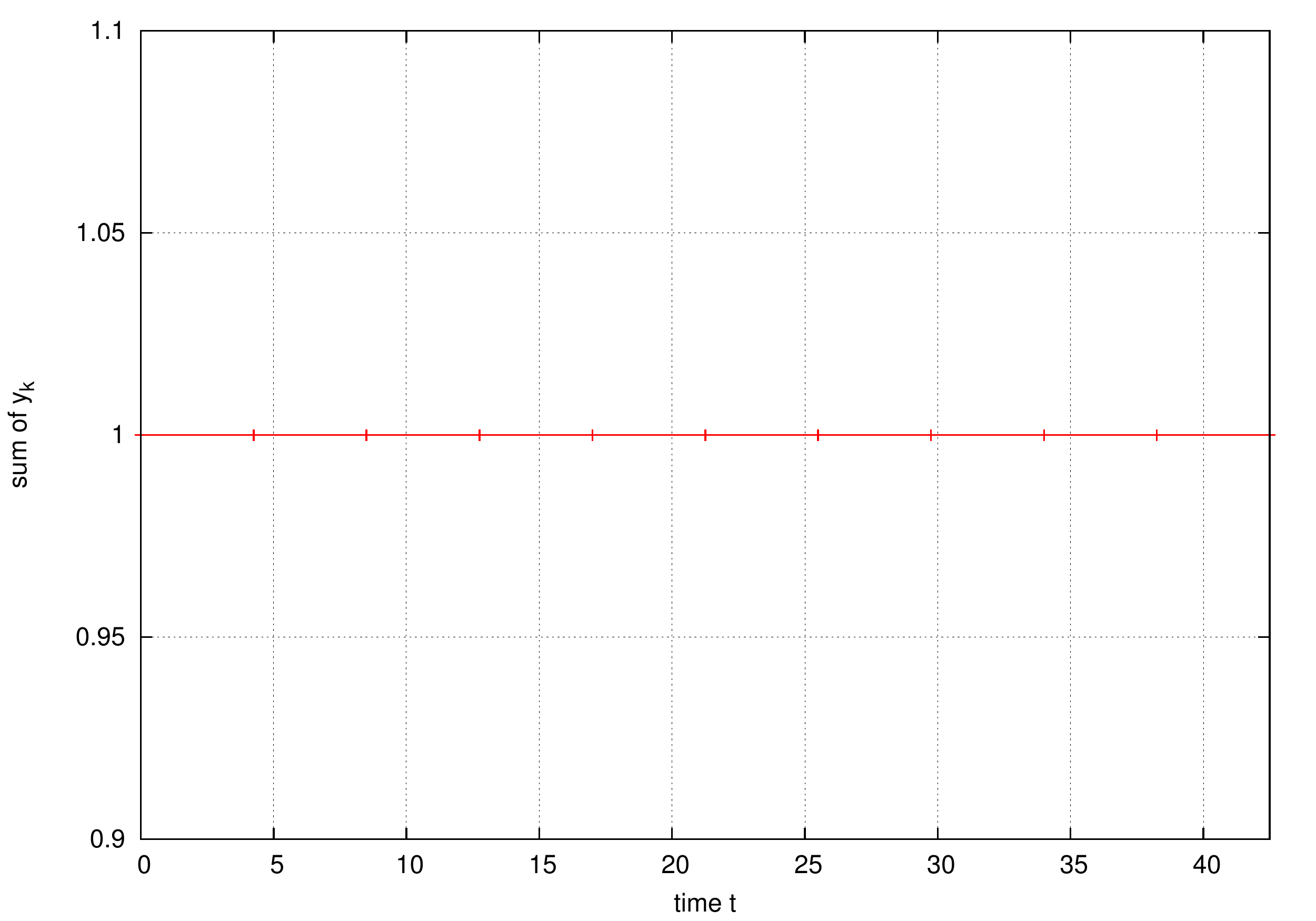}
\\
Color functions
& Mass fractions
\end{tabular}
\caption{test 4, two-dimensional passive transport with four components. Graphd of $t\mapsto\int_D \sum_{k=1}^4 {\cal Z}_k(x,t)\,\mathrm{d}x$ (left) and $t\mapsto\int_D \sum_{k=1}^4 {\cal Y}_k(x,t)\,\mathrm{d}x$ (right) for the approximation computed with the anti-diffusive scheme.
\label{fig:4mat unity constraint}
}
\end{figure}
%
%
%
%
%%%%%%%%%%%%%%%%%%%%%%%%%%%%%%%%%%%%%%%%%%%%%%%%%%%%%%%%%%%%%%%
%%%%%%%%%%%%%%%%%%%%%%%%%%%%%%%%%%%%%%%%%%%%%%%%%%%%%%%%%%%%%%%
%%%%%%%%%%%%%%%%%%%%%%%%%%%%%%%%%%%%%%%%%%%%%%%%%%%%%%%%%%%%%%%
%%%%%%%%%%%%%%%%%%%%%%%%%%%%%%%%%%%%%%%%%%%%%%%%%%%%%%%%%%%%%%%
%
%
%
\subsubsection{Test 5: Triple point problem}
We propose here a two-dimensional problem that involves the interaction of three Riemann problems across three initial material discontinuities and a triple point. 
Similar tests have been studied in the literature, see \textit{e.g.}{} \cite{Delpino1,Galera1,Kucharik1,Sijoy1}.
The computation domain is $[0,7\,\mathrm{m}] \times [0,3\,\mathrm{m}]$ and it is occupied by three perfect gases located as depicted in figure~\ref{fig:pt}. The medium, is initially at rest and all initial states and parameters of each fluid are given in table~\ref{tab:tab4}. %
\begin{figure}[]
\centering
 \begin{tikzpicture}[xscale=0.4,yscale=0.4]
  \filldraw[fill=green!10,draw=green!10] (0,0) rectangle (20,5);
  \filldraw[fill=blue!10,draw=blue!10] (0,0) rectangle (5,5); 
  \filldraw[fill=red!10,draw=red!10]   (5,0) rectangle (20,2.5); 
  \draw[dashed,violet] (5,0) -- (5,5);
  \draw[dashed,orange] (5,2.5) -- (20,2.5);
  \draw (2.5,3.5) node[below]{Fluid 1}; 
  \draw (12,4.5)node[below]{Fluid 2};
  \draw (12,2)node[below]{Fluid 3};
  \draw (0,5) node [left]{\small $x_2=3$};
  \draw (20,0) node [below]{\small $x_1=7$};
  \draw (20,2.5) node [right]{ \small $x_2=1.5$};
  \draw (5,0) node [below]{\small $x_1=1$};	
\end{tikzpicture}
\caption{test 5: triple point problem. Initial position of the interfaces in the computational domain.\label{fig:pt}}
\end{figure}
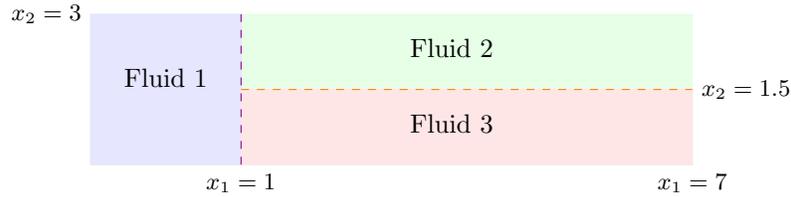
\begin{table}[]
\centering
\caption{test 5: triple point problem. Initial data and EOS parameters.} 
\begin{tabular}{ccccc}
\hline\hline
$k$ &$\rho$ \scriptsize{$(\mathrm{kg}.\mathrm{m}^{-3})$} & $p$ \scriptsize{$(\mathrm{Pa})$} & $u$ \scriptsize{$(\mathrm{m}/\mathrm{s},\mathrm{m}/\mathrm{s})$} & $\gamma_k$ \\
\hline
$k=1$ & $1.0$ &$1.0$    &  $(0,0)$ & $1.6$\\ 
$k=2$ & $0.125$     & $0.1$ & $(0,0)$  & $1.5$\\ 
$k=3$ & $1.0$  &$0.1$ & $(0,0)$  & $1.4$\\ 
\hline\hline
\end{tabular}
 \label{tab:tab4}
\end{table}

 At the initial instant, the pressure in the fluid $k=1$, depicted in blue in figure~\ref{fig:pt}, is greater than in the rest of the domain. This generates a set of waves, including two shocks travelling towards the right end of the domain: one of these waves travels within the fluid $2$ (green in figure~\ref{fig:pt}), the other within the fluid $3$ (red in figure~\ref{fig:pt}). The jump between the densities and the material properties
across the material interface that separates fluid $2$ and $3$ creates an instability.
For this simulation we use a $700\times300$-cell regular mesh and impose wall boundary conditions.
{Figures \ref{fig:pt_mat}, \ref{fig:pt_rho} and \ref{fig:pt_pre} show respectively the color function, the density and pressure profiles} at $t = 5\,\mathrm{s}$ obtained with both the upwind and the anti-diffusive scheme. The numerical solution obtained is in good agreement with those found in the literature (see for example \cite{Galera1,Galera2}). Here again the interfaces computed by the anti-diffusive scheme remain very sharp. Finally, we observe as expected no  significant difference between the pressure profiles obtained with both schemes as in \cite{Galera1}.
 
\begin{figure}[h!]
\centering
\includegraphics[scale=0.6, trim = 35mm 52mm 35mm 50mm, clip]{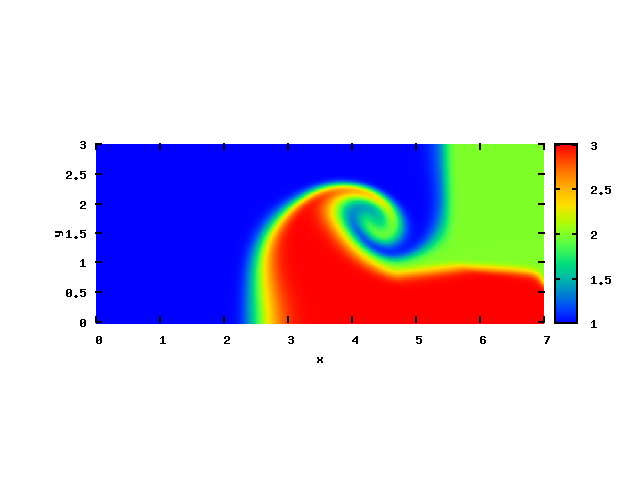}
\includegraphics[scale=0.6, trim = 35mm 52mm 35mm 50mm, clip]{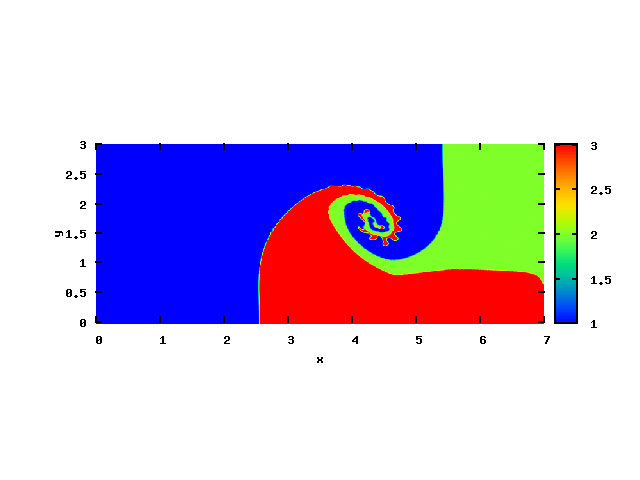}
\caption{test 5: triple point problem. Profiles of color functions for both upwind (top) and anti-diffusive (bottom) schemes at $t = 5\,\mathrm{s}$. \label{fig:pt_mat}}
\end{figure}

\begin{figure}[h!]
\centering
\includegraphics[scale=0.6, trim = 35mm  52mm 5mm 50mm, clip]{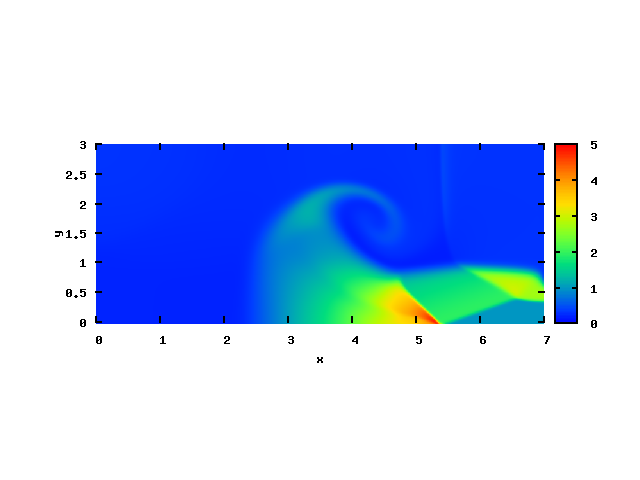}
\includegraphics[scale=0.6, trim = 35mm  52mm 5mm 50mm, clip]{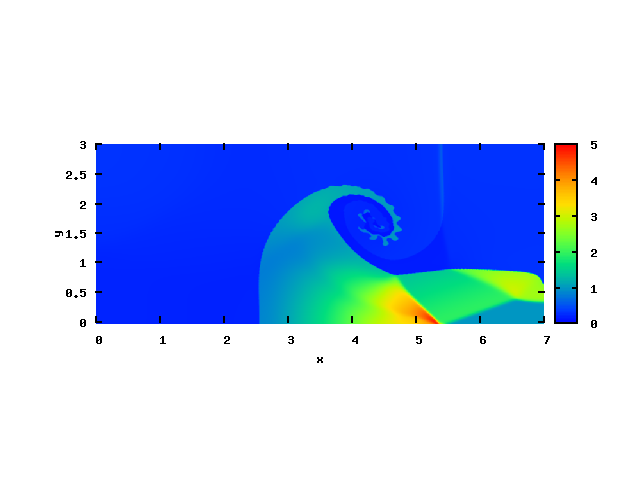}
\caption{test 5: triple point problem. Profiles of density for both upwind (top) and anti-diffusive (bottom) schemes at $t = 5\,\mathrm{s}$. \label{fig:pt_rho}}
\end{figure}

\begin{figure}[h!]
\centering
\includegraphics[scale=0.6, trim = 35mm  52mm 5mm 50mm, clip]{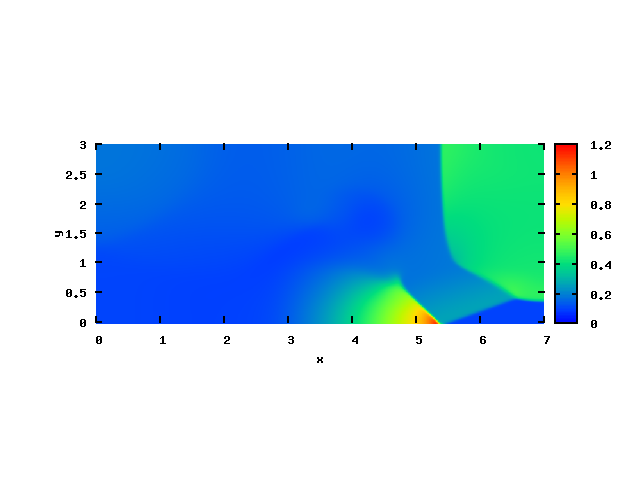}
\includegraphics[scale=0.6, trim = 35mm  52mm 5mm 50mm, clip]{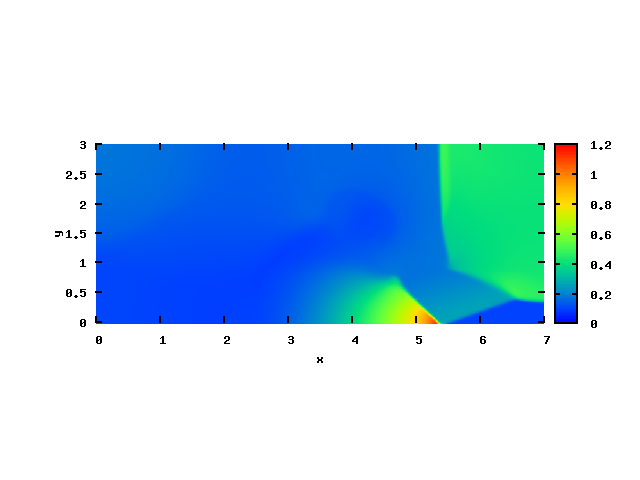}
\caption{test 5: triple point problem. Profiles of pressure for both upwind (top) and anti-diffusive (bottom) schemes at $t = 5\,\mathrm{s}$. \label{fig:pt_pre}}
\end{figure}
%
%
%
%
%
%
%
%%%%%%%%%%%%%%%%%%%%%%%%%%%%%%%%%%%%%%%%%%%%%%%%%%%%%%%%%%%%%%%
%%%%%%%%%%%%%%%%%%%%%%%%%%%%%%%%%%%%%%%%%%%%%%%%%%%%%%%%%%%%%%%
%%%%%%%%%%%%%%%%%%%%%%%%%%%%%%%%%%%%%%%%%%%%%%%%%%%%%%%%%%%%%%%
%%%%%%%%%%%%%%%%%%%%%%%%%%%%%%%%%%%%%%%%%%%%%%%%%%%%%%%%%%%%%%%
%
%
%
%
%
\subsection{Test 6: Two-dimensional shock-bubble interaction with three materials}
We now consider a test inspired by the experimental results of \cite{Haas1}. This problem has been studied in several publications with various approaches (see \textit{e.g.}\cite{Braconnier1,Galera1,Kokh1}). 
We consider here a slighty modified set up that involves three {gases: air, R22 (Chlorodifluoromethane) and Helium}, that are represented here by perfect gases. The computational domain is a rectangular region $P$ whose dimensions are $L_1\times L_2$. Let $(X_1,X_2)$ be a point of the domain, the regions
$$
\begin{aligned}
\mathcal{E} &= \{(x_1,x_2)\in P~|~ (x_1-X_1)^2 + (x_2-X_2)^2 < r_2 \},
\\
\mathcal{F} &= \{(x_1,x_2)\in P~|~ r_2 < (x_1-X_1)^2 + (x_2-X_2)^2 < r_1 \}
\end{aligned}
$$
are respectively a disc and a ring of center $(X_1,X_2)$. The region $E$ is filled with
Helium and $F$ contains R22. The gas in the rest of the domain is air. At the beginning of the computation, a Mach $1.22$ shock with initial position $x_s$ is travels through air towards the left end of the domain. For the present test, we use the following values
$$
\begin{aligned}
L_1 &= 0.445\,\mathrm{m}
,&
L_2 &= 0.089\,\mathrm{m}
,&
r_1 &= 0.025\,\mathrm{m}
  ,&
r_2 &= 0.015\,\mathrm{m}
,
\\
X_\mathrm{shock} &= 0.275\,\mathrm{m}
,&
X_1
&=0.225\,\mathrm{m}
,&
X_2 &=0.0445\,\mathrm{m}
.
\end{aligned}
$$
The initial state of the fluids and the EOS parameters we chose for the fluids are given in table~\ref{tab:tab5}.
\begin{table}[h!]
\centering
 \caption{test 6, two-dimensional shock-bubble interaction with three materials. EOS parameters and initial data.} \begin{tabular}{ccrcl}
 \hline\hline
material  & $\rho$ \scriptsize{$(\mathrm{kg}.\mathrm{m}^{-3})$} & $p$ \scriptsize{$(\mathrm{Pa})$} & $\boldsymbol{u}$ \scriptsize{$(\mathrm{m}.\mathrm{s}^{-1},\mathrm{m}.\mathrm{s}^{-1})$} & $\gamma_k$  \\
\hline\hline
Helium              & $0.138$    &$1.01325~10^5$  &$(0,0)$          & $1.6$\\ 
R22                   & $ 3.863$   &$1.01325~10^5$  &${(0,0)}$  & $1.249$\\ 
Air (pre-shock)  & $1.686$    &$1.59~10^5      $  &$(0,0)$          & $1.4$\\ 
Air (post-shock) & $1.225$   &$1.01325~10^5$  &${(-113.5,0)}$          & $1.4$\\ 
\hline\hline
\end{tabular}
\label{tab:tab5}
\end{table}
We use a regular mesh composed by $1250\times250$ cells for discretizing $P$, and the domain is closed using wall boundaries. 
The position of the material interfaces obtained with both the anti-diffusive scheme and the upwind scheme are displayed in figure~\ref{fig:CB} at several instants. The shock hits the air/R22 interface at  $t\simeq 60\,\mu\mathrm{s}$. After the impact the bubble and the gas ring are set in motion towards the left end of the domain. 
Both R22 and Helium bulk are compressed by the shock. Their shape is deformed by the motion so that two vortices appear on the right side of the gas bulk. The upwind scheme and the anti-diffusive scheme provide similar position of the material fronts, although the material interfaces obtained thanks to the anti-diffusive scheme are much sharper.

\begin{figure}[h!]
\centering
%\begin{tabular}{m{0.5cm}cc}
 \begin{tabular}{>{\centering\arraybackslash} m{2cm} >{\centering\arraybackslash} m{4.5cm}  >{\centering\arraybackslash} m{4.5cm} }
 $t= 0~\text{s}$&
\includegraphics[scale=0.4 , trim = 50mm 70mm 60mm 60mm, clip]{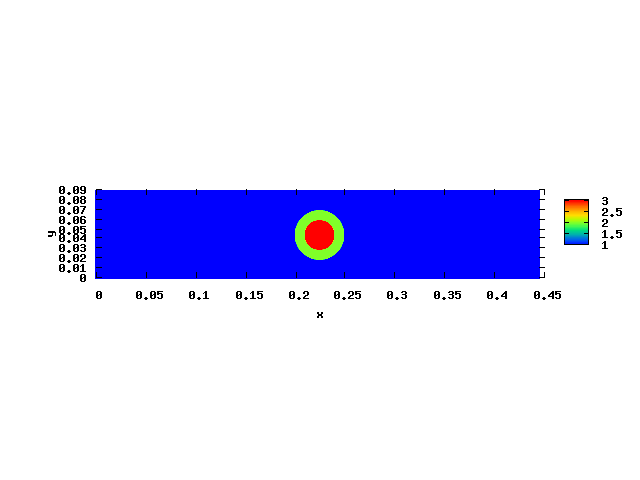}~~&
~\includegraphics[scale=0.4 , trim = 50mm 70mm 60mm 60mm, clip]{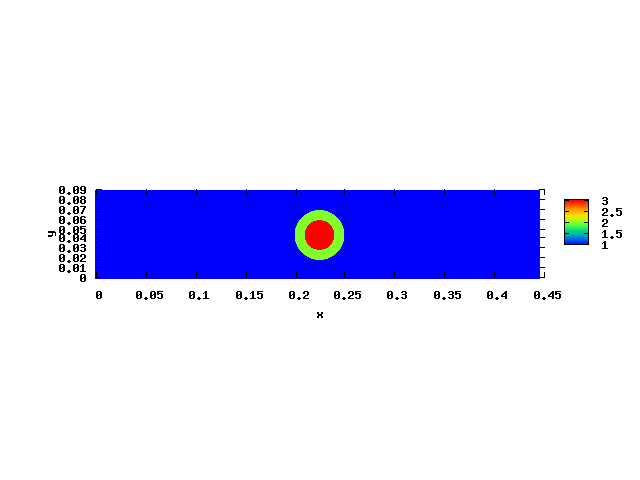} \\
  $t= 120~\mu\text{s}$& 
\includegraphics[scale=0.4 , trim = 50mm 70mm 60mm 60mm, clip]{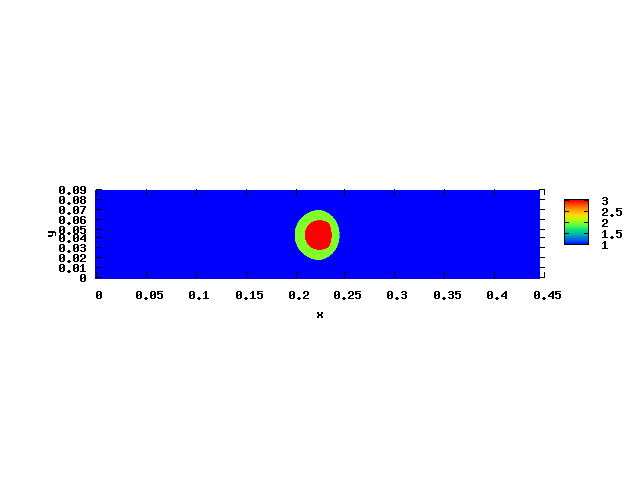}&
\includegraphics[scale=0.4 , trim = 50mm 70mm 60mm 60mm, clip]{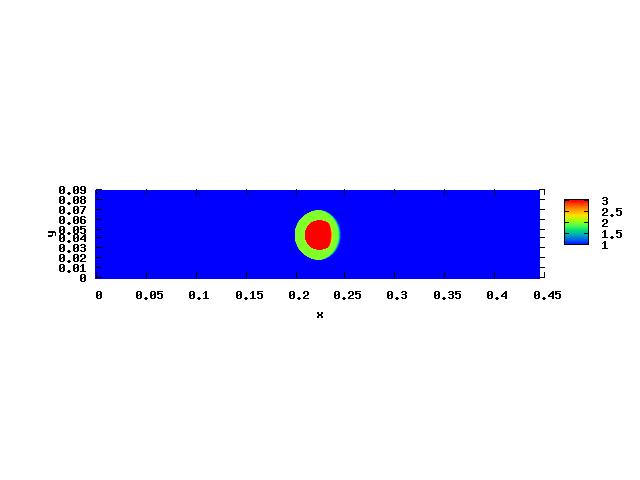}\\
 $t= 300~\mu\text{s}$&
%\includegraphics[scale=0.4 , trim = 50mm 70mm 35mm 60mm, clip]{./RES/2D/CB/ANTIDIFF/mat_anti_05.png}&
%\includegraphics[scale=0.4, trim = 50mm 70mm 35mm 60mm, clip]{./RES/2D/CB/UPWIND/mat_anti_05.png}\\
%  $t= 480~\mu\text{s}$&
 \includegraphics[scale=0.4 , trim = 50mm 70mm 60mm 60mm, clip]{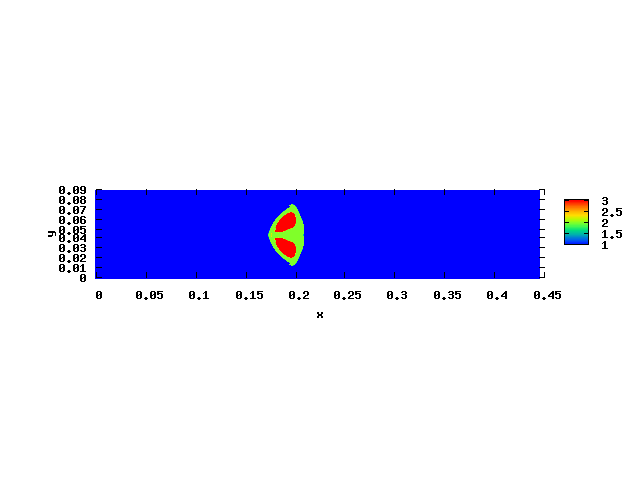}&
\includegraphics[scale=0.4, trim = 50mm 70mm 60mm 60mm, clip]{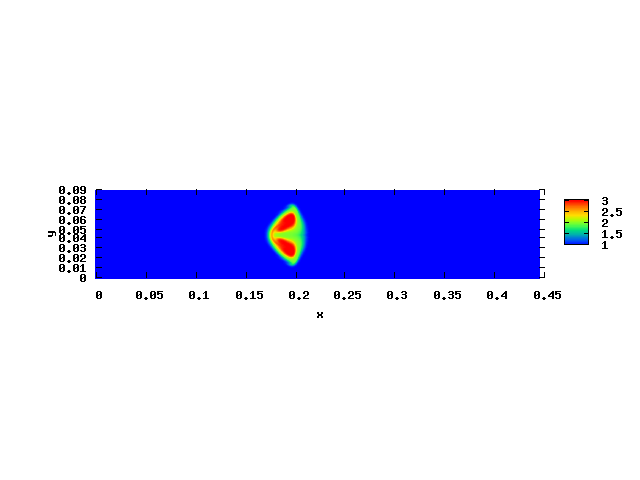}\\
 $t= 510~\mu\text{s}$&
\includegraphics[scale=0.4 , trim = 50mm 70mm 60mm 60mm, clip]{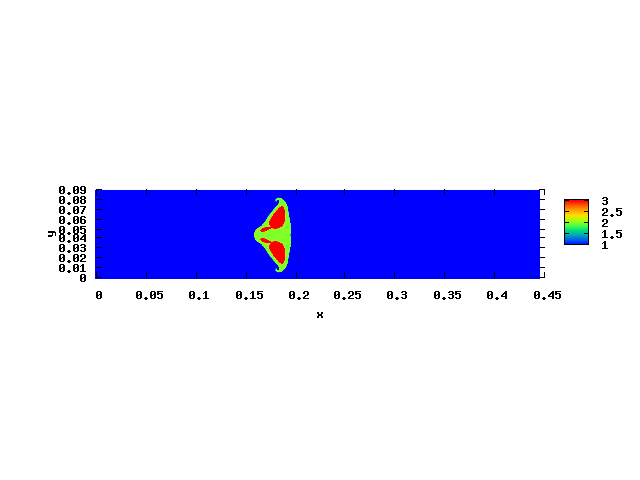} &
\includegraphics[scale=0.4 , trim = 50mm 70mm 60mm 60mm, clip]{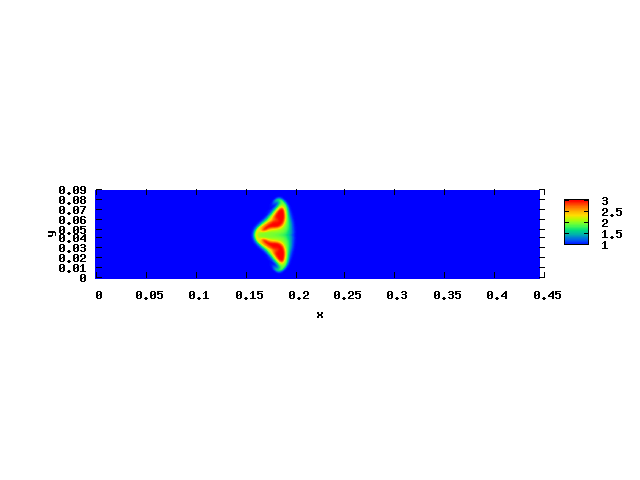}\\
 $t= 780~\mu\text{s}$ &
 \includegraphics[scale=0.4 , trim = 50mm 70mm 60mm 60mm, clip]{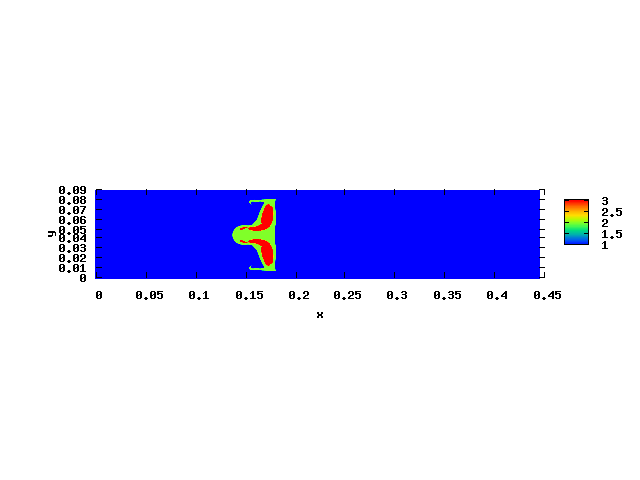} &
\includegraphics[scale=0.4 , trim = 50mm 70mm 60mm 60mm, clip]{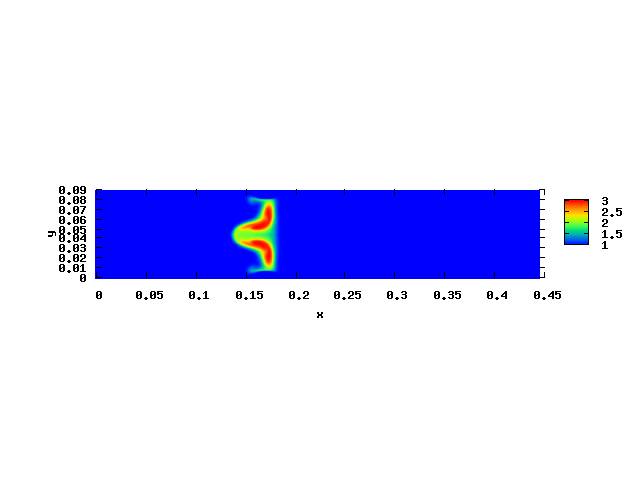}\\
 $t= 1200~\mu\text{s}$ & 
\includegraphics[scale=0.4 , trim = 50mm 70mm 60mm 60mm, clip]{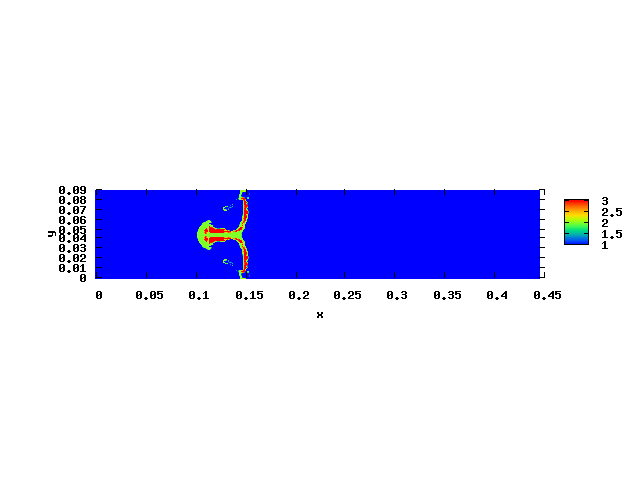}&
\includegraphics[scale=0.4 , trim = 50mm 70mm 60mm 60mm, clip]{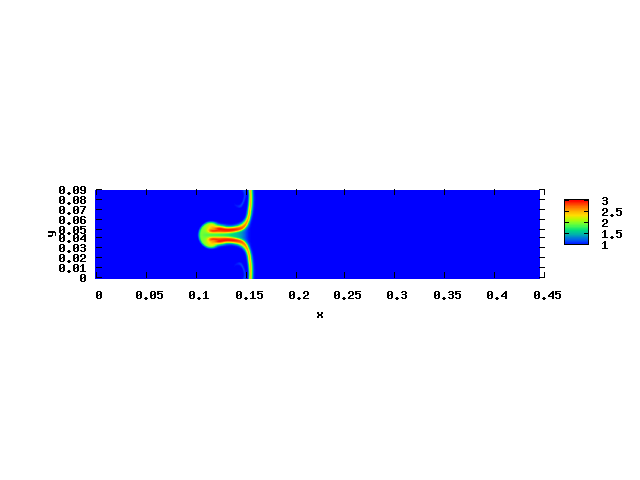}
\end{tabular}
\caption{test 6, two-dimensional shock-bubble interaction with three materials. Mapping of $\sum_{k=1}^3 k \z_k$ for both upwind (right) and anti-diffusive (left) schemes at different times. \label{fig:CB}}
\end{figure}

\begin{figure}[h!]
\centering
\begin{tabular}{cc}
\includegraphics[scale=0.4 , trim = 40mm 70mm 35mm 60mm, clip]{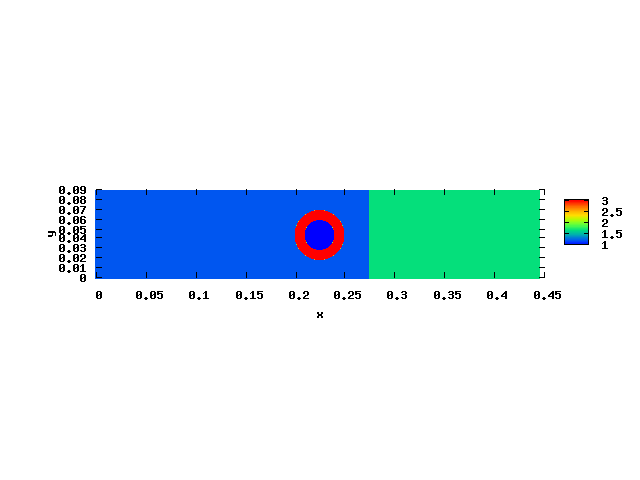}
\includegraphics[scale=0.4 , trim = 40mm 70mm 35mm 60mm, clip]{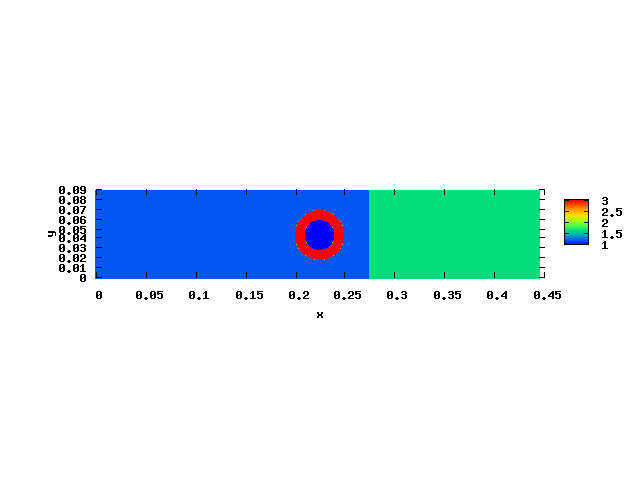} \\
 $t= 0~\text{s}$\\
\includegraphics[scale=0.4 , trim = 40mm 70mm 35mm 60mm, clip]{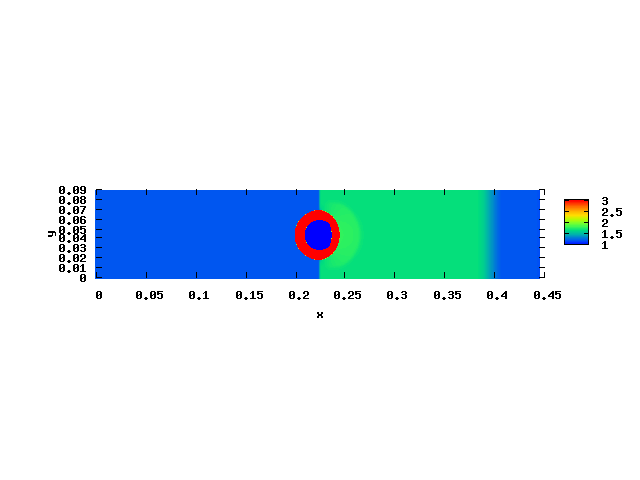}
\includegraphics[scale=0.4 , trim = 40mm 70mm 35mm 60mm, clip]{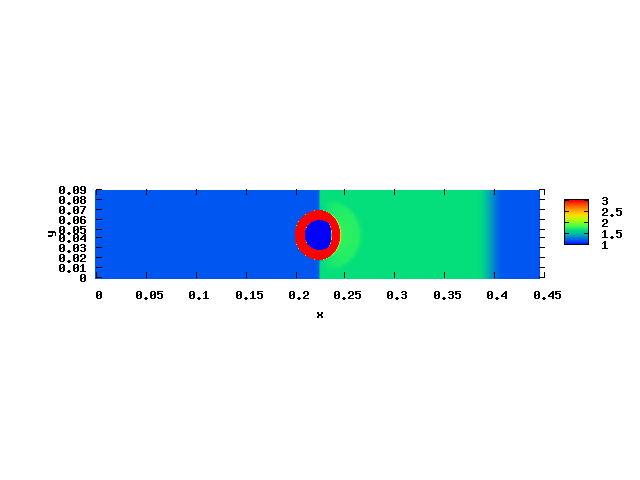}\\
 $t= 120~\mu\text{s}$\\
\includegraphics[scale=0.4 , trim = 40mm 70mm 35mm 60mm, clip]{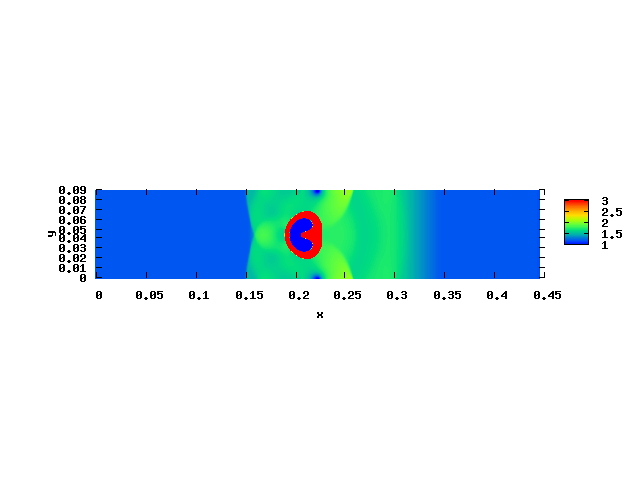}
\includegraphics[scale=0.4 , trim = 40mm 70mm 35mm 60mm, clip]{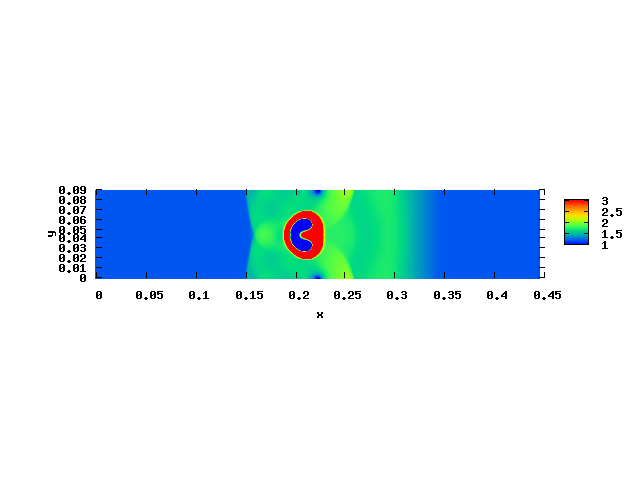}\\
 $t= 300~\mu\text{s}$\\
 \includegraphics[scale=0.4 , trim = 40mm 70mm 35mm 60mm, clip]{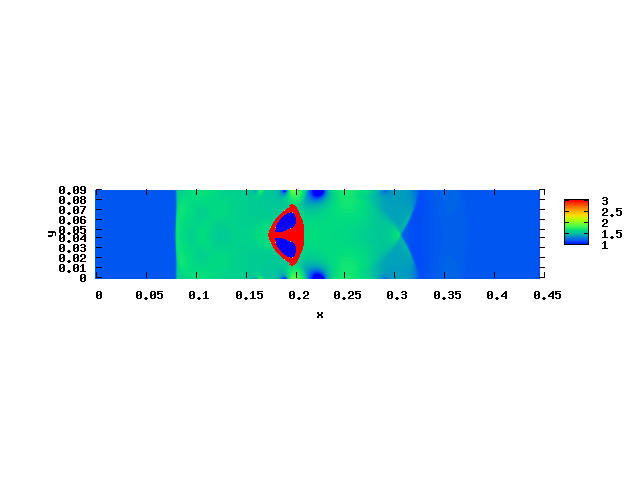}
\includegraphics[scale=0.4 , trim = 40mm 70mm 35mm 60mm, clip]{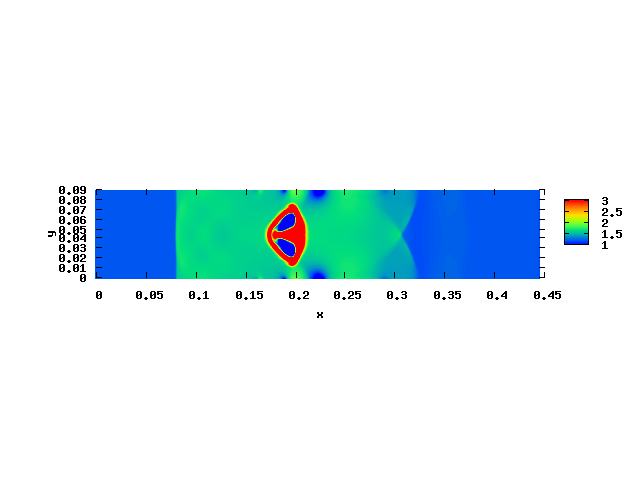}\\
 $t= 480~\mu\text{s}$\\
%\includegraphics[scale=0.55 , trim = 40mm 70mm 35mm 60mm, clip]{./RES/2D/CB/ANTIDIFF/mat_anti_10.png}
%\includegraphics[scale=0.55 , trim = 40mm 70mm 35mm 60mm, clip]{./RES/2D/CB/UPWIND/mat_anti_10.png}\\
% $t= 510~\mu\text{s}$\\
 \includegraphics[scale=0.4 , trim = 40mm 70mm 35mm 60mm, clip]{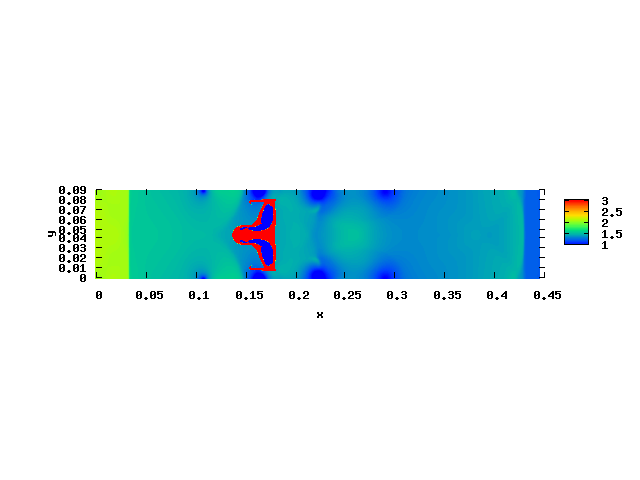}
\includegraphics[scale=0.4 , trim = 40mm 70mm 35mm 60mm, clip]{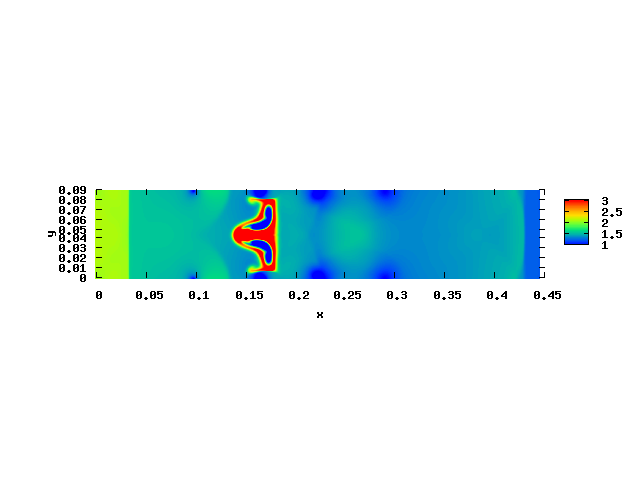}\\
 $t= 780~\mu\text{s}$\\
\includegraphics[scale=0.4 , trim = 40mm 70mm 35mm 60mm, clip]{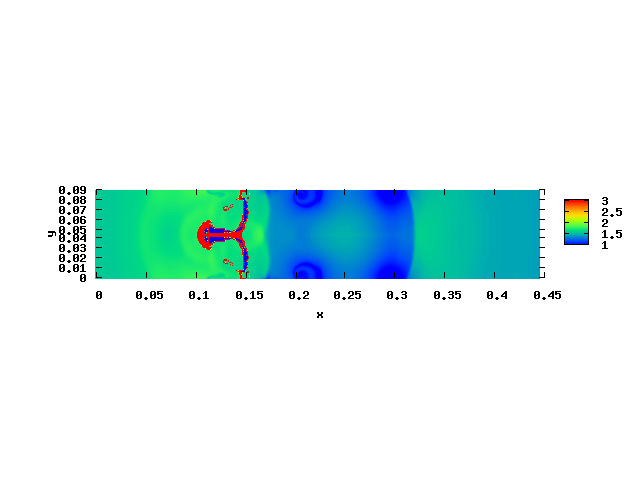}
\includegraphics[scale=0.4 , trim = 40mm 70mm 35mm 60mm, clip]{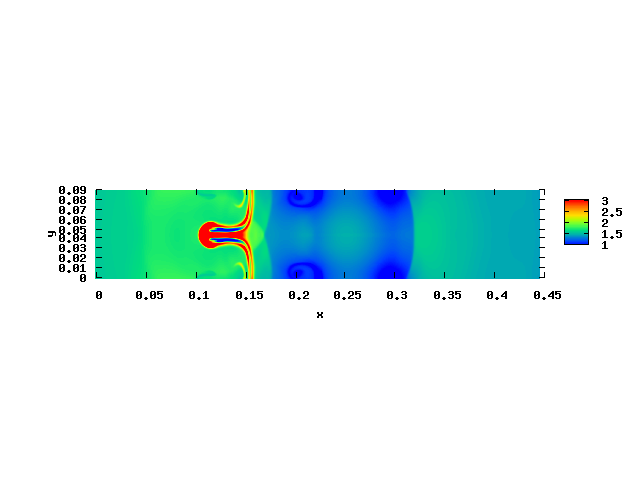}\\
 $t= 1020~\mu\text{s}$
\end{tabular}
\caption{test 6, two-dimensional shock-bubble interaction with three materials. Density profiles for both upwind (right) and anti-diffusive (left) schemes at different times. \label{fig:CB}}
\end{figure}

\subsection{Test 7: Two-dimensional  Kelvin-Helmoltz instability}

In this last part, we want to test the capability of the proposed method to capture complex modification of the interface geometry between an arbitrary number of materials. To this end, we are interested in computing a multi-material Kelvin-Helmoltz instability inspired by \cite{Galera1,Springel1}.\\ 

We consider a simple two-dimensional shear flow in a square periodic domain of dimension $1\mathrm{m} \times 1\mathrm{m}$. Here, we suppose that the domain is initially occupied by three perfect gases numbered $k=1,\dots,3$ as depicted in figure \ref{fig:kh}. At $t=0$, we assume the pressure constant and the lighter gas (in blue) is initially moving from the right to the left while the heavier ones (green and red) are moving in the opposite direction. The initial state of the fluid is given in the table \ref{tab:kh}. The interface is then perturbed by an imposed vertical velocity $\phi$ field that reads
$$
\phi(x_1,x_2) = \omega_0 sin(4 \pi x) \left [e^{\frac{-(x_2-0.25)^2}{2\sigma^2}}
+ e^{\frac{-(x_2-0.75)^2}{2\sigma^2}} \right]
$$
where $\omega_0 = 0.1$ and $\sigma = 0.05/\sqrt{2}$. Such a choice guarantees that only a single mode will dominate the linear grow of the instability that will develop at the interfaces. The domain is discretized by means of a $1000 \times 1000$-cell regular mesh.\\
 Figure~\ref{fig:kh mat} represents a mapping of the computed color functions at different times  between $t=0\mathrm{s}$ and $t=2\mathrm{s}$ for both anti-diffusive and upwind schemes. As expected in the references, the Kelvin-Helmoltz instability is illustrated by the formation for both schemes of four symmetric vertices. Once again, all the fine structures are  sharply captured with the anti-diffusive scheme while it becomes very difficult to distinguish the interfaces with the upwind scheme due to numerical diffusion.

\begin{figure}[!ht]
\centering
 \begin{tikzpicture}[scale=0.4]
     \tikzstyle{every node}=[font=\footnotesize]
  \filldraw[fill=green!10,draw=green!10] (0,0) rectangle (10,10);
  \filldraw[fill=red!10,draw=red!10] (2.5,2.5) rectangle (7.5,7.5); 
  \filldraw[fill=blue!10,draw=blue!10]   (0,0) rectangle (10,2.5); 
  \filldraw[fill=blue!10,draw=blue!10]   (0,7.5) rectangle (10,10); 
\draw[dashed,violet] (0,2.5) -- (10,2.5);
\draw[dashed,violet] (0,7.5) -- (10,7.5);
\draw[dashed,orange] (2.5,2.5) -- (2.5,7.5);
\draw[dashed,orange] (7.5,2.5) -- (7.5,7.5);
\draw (5,8.5) node[]{Fluid 1}; 
\draw (5,1.5) node[]{Fluid 1}; 
\draw (5,5)node[]{Fluid 2};
\draw (1.25,5)node[]{\footnotesize Fluid 3};
\draw (8.75,5)node[]{\footnotesize Fluid 3};
\draw (0,0) node [below]{\small $x_1=0$};
\draw (10,0) node [below]{\small $x_1=1$};
\draw (0,0) node [left]{ \small $x_2=0$};
\draw (0,10) node [left]{ \small $x_2=1$};
\draw (0,2.5) node [left]{ \small $x_2=0.25$};
\draw (0,7.5) node [left]{ \small $x_2=0.75$};
%  \draw (5,0) node [below]{\small $x_1=1$};	
\end{tikzpicture}
\caption{test 7, two-dimensional  Kelvin-Helmoltz instability. Initial interface positions. \label{fig:kh}}
\end{figure}
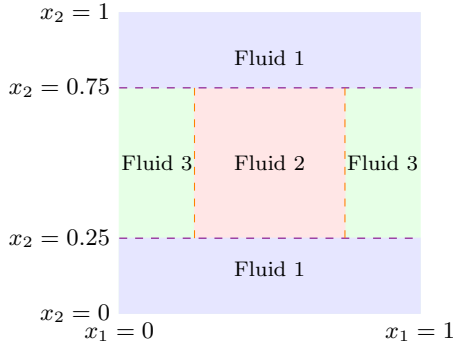

\begin{table}[ht!]
\centering
\caption{test 7, two-dimensional  Kelvin-Helmoltz instability. Initial data and EOS parameters.} 
\begin{tabular}{ccccc}
\hline\hline
material index & density \scriptsize{$(\mathrm{kg}.\mathrm{m}^{-3})$} & pressure \scriptsize{$(\mathrm{Pa})$} & velocity  \scriptsize{$(\mathrm{m}.\mathrm{s}^{-1},\mathrm{m}.\mathrm{s}^{-1})$} & $\gamma_k$  \\
\hline \hline
$k=1$ & $1$ & $2.5$    & $(   -0.5,0)$ & $5/3$ \\ 
$k=2$ & $2$ & $2.5$    & $(~~0.5,0)$ & $1.4$ \\ 
$k=3$ & $2$ & $2.5$    & $(~~0.5,0)$ & $ 2.4$ \\ 
\hline\hline
\end{tabular}
 \label{tab:kh}
\end{table}

\begin{figure}[h!]
\centering
\begin{tabular}{cc}
 \includegraphics[scale=0.5 , trim = 55mm 30mm 55mm 30mm, clip]{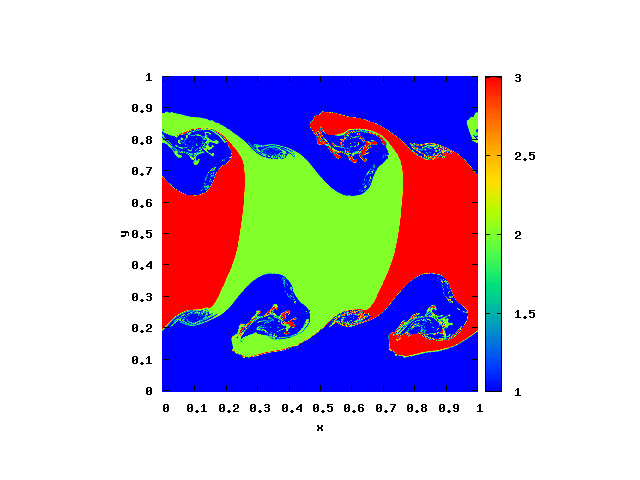}
\includegraphics[scale=0.5 , trim = 55mm 30mm 55mm 30mm, clip]{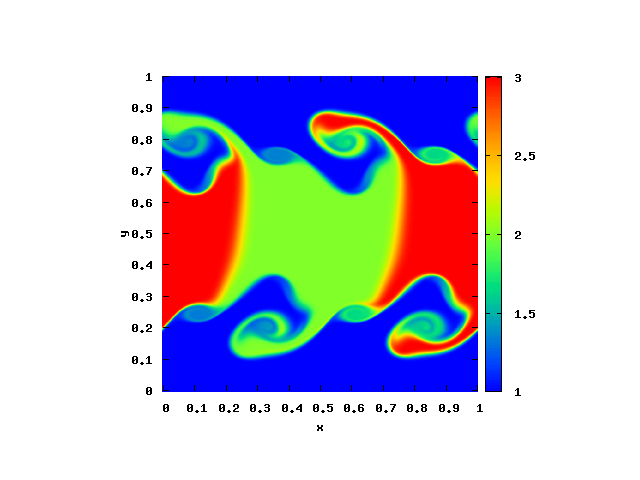}\\
 $t= 1~\text{s}$\\
\includegraphics[scale=0.5 , trim = 55mm 30mm 55mm 30mm, clip]{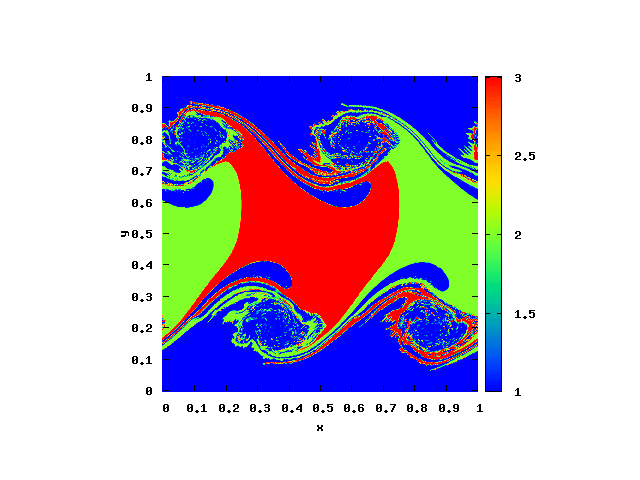}
\includegraphics[scale=0.5 , trim = 55mm 30mm 55mm 30mm, clip]{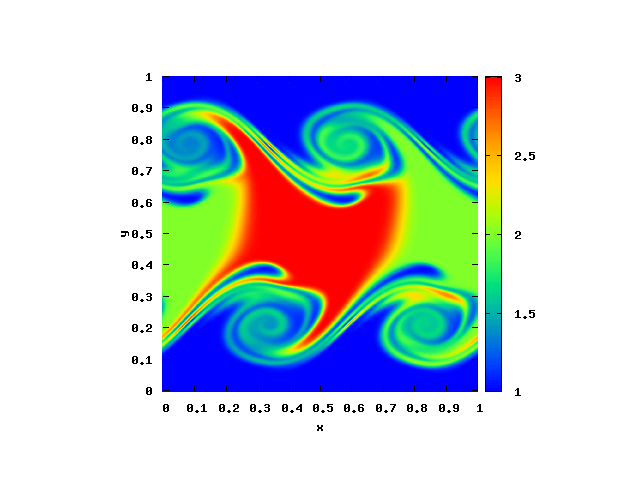}\\
 $t= 2\text{s}$
\end{tabular}
\caption{test 7, two-dimensional  Kelvin-Helmoltz instability. Mapping of $\sum_{k=1}^3 k \z_k$ for both upwind (right) and anti-diffusive (left) schemes at different times. \label{fig:kh mat}}
\end{figure}

\section{Conclusion}

We propose in the present paper a method for simulating compressible flows involving an arbitrary number $m$ of materials separated by sharp interfaces.

Our approach relies on two elements. First, we introduced a $m$-component flow model that extends the two-component five-equation model with isobaric closure of~\cite{Allaire1, Allaire2} to $m$  components. Under simple thermodynamical assumptions this model is hyperbolic and may be expressed equivalently with a quasi-conservative form or a fully conservative form. This model is adapted to a wide range of EOSs, including tabulated laws. Second, we proposed an anti-diffusive Lagrange-Remap Finite Volume scheme combining the lines of \cite{Kokh1} and \cite{Despres4}. This numerical scheme succeeds in controlling the numerical diffusion of the material fronts without any interface reconstruction process. The overall numerical scheme is conservative with respect to the mass of each component, the total mass, the momentum and the total energy. It is also endowed with stability properties under a classic CFL condition. Moreover we also proved and tested the ability of our discretization to preserve a category of solutions with constant velocity and pressure.

Our simulation strategy was tested against several one-dimensional and two-dimensional test cases involving up to five distinct materials and complex EOSs. The results show that the anti-diffusive mechanism succeeds in evolving the interfaces while preserving their sharpness. The accuracy of the overall method is enhanced at the vicinity of the contact discontinuities that drive the material fronts, while far from the interfaces the behavior of the scheme if similar to a classic upwind discretization.

The simulation approach we proposed here has been implemented in the code HERACLES~\cite{heracles-web-site}. Further developments will deal with enhancing the accuracy of the solver by using high order methods as in \cite{Billaud1} and propose an extension to unstructured meshes following the lines of \cite{Despres4, Faucher2013}.

\appendix

{
\section{Hyperbolicity and eigenstructure analysis}\label{section: eigenstructure analysis}
This section details the hyperbolicity study of the $m$-multicomponent model as stated in proposition \ref{prop:hyperbolicity}. To this goal, we start by considering some companion evolution equations of the system~\eqref{eq:syst1}. First, combining the partial mass conservation  equations ~\eqref{eq:syst1 partial mass} we get the conservation equation for the total mass
\begin{equation}
\partial_t \rho +  \nabla \cdot (\rho \bu ) =0. 
 \label{eq: mass conservation}
\end{equation}
Now, introducing the material derivative $D_t$ defined for any fluid parameter $(\bx,t)\mapsto a$ by
$$
D_t a = \partial_t a + \bu \cdot \nabla a.
$$
%The equation~\eqref{eq:syst1 color function} for the color function, simply states that $\zk$ is passively advected with the velocity $\bu$.
the system~\eqref{eq:syst1} can be classically recast into the following form
\begin{subequations}
\label{eq:syst1-md}
\begin{align}
\rho D_t \bu + \nabla p &=0,
\label{eq:syst1-md momentum}
\\
\rho D_t E + \nabla \cdot (p\bu) &=0,
\label{eq:syst1-md  energy }
\\
D_t (\rhoz) +  \rhoz (\nabla \cdot \bu  )&=0,& \quad  k=1,\dots,m,
\label{eq:syst1-md partial mass}
\\
D_t \zk  &=0,&\quad  k=1,\dots,m.
\label{eq:syst1-md color function}
\end{align}
\end{subequations}
either by assuming that $(\bx,t)\mapsto \bu$ is a regular function or by considering that the formulation~\eqref{eq:syst1-md} is a formal expression. Next, the medium specific internal energy $e$ verifies
\begin{equation}
 \rho D_t e + p(\nabla\cdot \bu) =0,
\label{eq: internal energy}
\end{equation}
where  the pressure $p$ is supposed to be uniquely defined by \eqref{eq:pressure}. 
Finally, we introduce the mass fraction $\yk$ of fluid $k=1,\dots,m$ defined by 
\begin{equation}
 \rho \yk = \rk \zk, 
\quad k=1,\dots,m,
 \label{eq: mass fraction definition}
\end{equation}
and using \eqref{eq:syst1-md partial mass} we obtain that 
\begin{equation}
 D_t \yk =0,\quad k=1,\dots,m.
\label{eq: mass fraction transport}
\end{equation}
%Now, let us examine the evolution equation for $p$. 
\begin{prop}
 The pressure $p$ verifies
\begin{equation}
 D_t p  = -\rho c^2 (\nabla\cdot \bu),
\label{eq: pressure evolution}
\end{equation}
where $c>0$ is the sound velocity of the $m$-component medium that is defined by
\begin{equation}
 \rho \xi c^2 = \sum_k \zk \rk \xi_k c_k^2.
 \label{eq: sound velocity}
\end{equation}
\begin{proof}
 The specific internal energy can be expressed thanks to the phasic internal energy and the mass fractions $\yk$. Indeed,
$$
e = \sum_k \yk e_k(\rho_k,p),
$$
thus
$$
\begin{aligned}
\rho D_t e 
&= 
\sum_k \rho \yk 
\left(\frac{\partial e_k}{\partial \rk}\right)_{p_k} D_t \rk
+
\left[
\sum_k \rho \yk
\left(\frac{\partial e_k}{\partial p_k}\right)_{\rk} 
\right]
D_t p
\\
&= 
\sum_k 
\rk 
\left(\frac{\partial e_k}{\partial \rk}\right)_{p_k} D_t (\zk\rk)
+
\left[
\sum_k \zk \rk
\left(\frac{\partial e_k}{\partial p_k}\right)_{\rk} 
\right]
D_t p
\\
&= 
-
\sum_k 
\zk \rk^2 
\left(\frac{\partial e_k}{\partial \rk}\right)_{p_k} (\nabla\cdot \bu)
+
\left[
\sum_k \zk 
\left(\frac{\partial \rk e_k}{\partial p_k}\right)_{\rk} 
\right]
D_t p,
\end{aligned}
$$
which leads to
$$
\rho D_t e 
=
- (\nabla\cdot \bu)
\sum_k 
\zk \rk^2 
\left(\frac{\partial e_k}{\partial \rk}\right)_{p_k} 
+
D_t p
\sum_k \zk \xi_k 
.
$$
Using \eqref{eq: def of xi} and \eqref{eq: internal energy}
we obtain
$$
\xi D_t p + 
(\nabla\cdot\bu)
\sum_{k=1}^{m}
\zk
\left[
p - \rk^2 \left(\frac{\partial e_k}{\partial \rk}\right)_{p_k} 
\right]
=0.
$$
One notes that \eqref{eq:sound_k} provides also the identity
$$
\rk\xi_kc_k^2 = - \rk^2\left(\frac{\partial e_k}{\partial \rk}\right)_{p_k} + p_k
,\quad k=1,\dots,m,
$$
then
$$
\xi D_t p + 
(\nabla\cdot\bu)
\sum_{k=1}^{m}
\zk \rk \xi_k c_k^2
=0,
$$
which completes the proof.
\end{proof}
\end{prop}
We now turn back to the system~\eqref{eq:prim}. Under the assumptions of proposition \ref{prop:hyperbolicity} the matrix $A$ satisfies the following properties
\begin{itemize}
 \item[a)] the matrix $\mathbf{A}(\mathbf{V})$ is diagonalizable,
 \item[b)] the eigenvalues $(\lambda_k)_{k=1,\ldots,2m+2}$ of $\mathbf{A}(\mathbf{V})$ are 
the real numbers
 $$\lambda_1 = u - c
, \quad 
\lambda_2 = u + c 
,\quad
\lambda_3 = \cdots = \lambda_{2m+2} = u
,
$$
\item[c)] the right eigenvectors $(\mathbf{r}_k)_{k=1,\ldots,2m+2}$ associated with the eigenvalues \\$(\lambda_k)_{k=1,\ldots,2m+2}$ are
$$
\begin{aligned}
\mathbf{r}_1 &= (1, -c \rho, 0, 0,\cdots,0 )^T
,\\
\mathbf{r}_2 &= (1,c\rho , 0 , 0,\cdots,0)^T
,\\
\mathbf{r}_k &= ( 0,\cdots,0,\underbrace{1}_{\text{$k^{\rm th}$ coordinate}},0,\cdots,0)^T
,\quad k=3,\ldots,2m+2,
\end{aligned}
$$
\item[d)]
the fields associated with the eigenvalues $\lambda_1$ and $\lambda_2$ are genuinely nonlinear, while those associated with
$\lambda_k$ are linearly degenerate for $k=3,\ldots,2m+2$.
\end{itemize}
As a consequence, the system \eqref{eq:syst1} is hyperbolic.
}\\

Finally we examine the definition of the non-conservative product ${\bf u}\cdot\nabla\mathcal{Z}_k$. System~\eqref{eq:syst1} can be expressed as a conservation law using the variables $\rho \bu$, $\rho E$, $\rho_k\mathcal{Z}_k$, $\rho \mathcal{Z}_k$, $k=0,\ldots,m$. Under this form, when one considers a pure one-dimensional discontinuous solution of the system, a classical set of Rankine-Hugoniot conditions is available. In particular, the jump condition associated with $\rho \mathcal{Z}_k$ for $k=0,\ldots,m$ ensures that the color function variable $\mathcal{Z}_k$  and the material velocity cannot simultaneously experience a jump across the discontinuity.

\section{Expression of the $m$-component system in Lagrangian Coordinates}
In this section we present present the expression of the $m$-component system~\eqref{eq:syst1} using Lagrangian coordinates for the case of one-dimensional problems. This will be useful for the derivation of the Lagrange-Projection algorithm we aim to elaborate in the next sections for solving \eqref{eq:syst1}. We consider the mapping $(X,t)\in\RR\times[0,+\infty)\mapsto \chi$ defined by the ordinary differential equation
$$
\partial \chi / \partial t = u(\chi(X,t),t)
,\quad
\chi(X,t=0) = X.
$$
The Lagrangian system of coordinates $(X,t)$ attached to the initial instant $t=0$ expresses that a fluid particle located at $x=\chi(X,t)$ at instant $t$ was originally at position $X$ at instant $t=0$. This means that
$
X = \chi^{-1}(x,t).
$
Now, let $(x,t)\mapsto q$ be a mapping that represents a medium parameter in Eulerian coordinates, we can define a Lagrangian representation $(X,t)\mapsto q^\mathrm{Lag}$ of this parameter by setting $q^{\mathrm{Lag}}(X,t) = q(\chi(X,t),t)$. 
{If} we note $\rho^{\mathrm{Lag}} (X,t=0) = \rho^{\mathrm{Lag}}_0 (X)$, the system~\eqref{eq:syst1} for one-dimensional problems is equivalent to
\begin{subequations}
\begin{align}
\rho^{\mathrm{Lag}}_0 &\partial_t \left({1}/{\rho^{\mathrm{Lag}}}\right) - \partial_X u^{\rm Lag} =0 
\\
\rho^{\rm Lag}_0 &\partial_t u^{\rm Lag} + \partial_X p^{\rm Lag} =0 
\\
\rho^{\rm Lag}_0 &\partial_t E^{\rm Lag} + \partial_X (p^{\rm Lag} u^{\rm Lag}) =0 
\\
\rho^{\rm Lag}_0 & \partial_t \yk^{\rm Lag} =0 ,\qquad k=1,\ldots,m-1,\\
\rho^{\rm Lag}_0 & \partial_t \zk^{\rm Lag} =0 ,\qquad k=1,\ldots,m-1.
\end{align}
\end{subequations}

\section{Preservation of iso-pressure and iso-velocity profiles}\label{section: iso-p iso-u profiles}

As in \cite{Kokh1}, we proove that the anti-diffusive scheme introduced in section 3 preserves a family of iso-pressure and iso-velocity profiles. Let us introduce the primitive variable $\V^n = (u^n, E^n, {\cal Y}_1^n, \dots , {\cal Y}_m^n )$ for one-dimensional problems.  
\begin{prop}
Let $(\rho^n \V^n,{\cal Z}_1^n, \dots , {\cal Z}_m^n) $ be the approximate state variable at instant $t=t^n$. Suppose that 
$(\rho^n \V^n,{\cal Z}_1^n, \dots , {\cal Z}_m^n)$ is a contact discontinuity with both uniform velocity and pressure such that 
$$
p_i^n = \overline{p},~~ u_i^n = \overline{u}, ~~(\rho_k)_i^n = \overline{\rho}_k, \quad\forall i \in \ZZ, \quad
1\leq k\leq m.
$$
Then the approximate solution $(\rho^{n+1} \V^{n+1},{\cal Z}_1^{n+1}, \dots , {\cal Z}_m^{n+1})$ computed with the anti-diffusive scheme satisfies
$$
p_i^{n+1}  = \overline{p}, ~~u_i^{n+1}  = \overline{u}, ~~(\rho_k)_i^{n+1}  = \overline{\rho}_k, \quad\forall i \in \ZZ
,\quad
1\leq k\leq m.
$$
\begin{proof}
Let us first consider the following variable $(\widetilde{\rho \V}^n,\widetilde{\cal Z}_1^n, \dots , \widetilde{\cal Z}_m^n)$ deduced from the Lagrangian step. Since the phasic densities and the pressure are constant at $t=t^n$, we have
\begin{equation}
(\rho_ke_k)_i^n =\overline{\rho_k e_k}, \quad \forall i \in \ZZ,\quad
1\leq k\leq 	m.
\end{equation}
In addition as the velocity is constant, the Lagrangian fluxes needed in the acoustic scheme are equal to zero i.e. $ \Flux{p^n}_i=0,  \Flux{p^nu^n}_i=0$ and $L_i = 1$. Thereby, we have $\forall i \in \ZZ$
$$
\left\{
\begin{array}{rcl}
(\Lrhou)_i&=& (\rho u)^n_i, \\[0.15cm]
(\Lrhoe)_i &=& (\rho E)^n_i, \\[0.15cm]
(\Lrhoz)_i &=& (\rhoz)^n_i, \quad 1\leq k\leq m,\\[0.15cm]
\widetilde{\cal Z}_{k,i}           &=&   {\cal Z}_{k,i}^n,\quad
1\leq k\leq m. 
\end{array}
\right.
$$
Thanks to the previous relations, we obtain  $\widetilde{\rho}_{k,i} = \rho_{k,i}^n= \overline{\rho}_k$, $1\leq k\leq m$. Then by summation over $k=1,\dots,m$ and using theorem \ref{theorem 1} we have $ \widetilde{\rho}_i = \rho^n_i$. This leads to
$$
\widetilde{u}_i = u_i^n = \overline{u} \text{ and }  \widetilde{E}_i = E_i^n, ~~ \forall i \in \ZZ.
$$
By definition of $E$, it yields $\widetilde{e}_i = e_i^n$ and thus
$$
 \widetilde{\rho e}_i = (\rho e)_i^n, ~~ \forall i \in \ZZ. 
$$
In the context of the isobaric closure we have for all $i$ in $\ZZ$
\begin{equation}
\left\{
\begin{array}{rcl}
\widetilde{p}_i        &=& p_1 ( \overline{\rho}_1, (\widetilde{\rho_1 e_1})_i)  = \dots= p_m ( \overline{\rho}_m, (\widetilde{\rho_m e_m})_i), \\
\widetilde{\rho e}_i &=& \displaystyle (\rho e)_i^n = \sum_{k=1}^m (\rhoz)^n_i \widetilde{\rho_k e_k}_i.
\end{array}
\right.
\label{eq:iso1}
\end{equation}
However at $t=t^n$ we have $p_i^n=\overline{p}$ and $(\rho_ke_k)_i^n =\overline{\rho_k e_k}$, $1\leq k \leq m$. Thus, we have
\begin{equation}
\left\{
\begin{array}{rcl}
\overline{p}  &=& p_1 ( \overline{\rho}_1, (\widetilde{\rho_1 e_1})_i)  = \dots= p_m ( \overline{\rho}_m, (\widetilde{\rho_m e_m})_i), \\	
(\rho e)^n_i  &=&  \displaystyle \sum_{k=1}^m (\rhoz)^n_i \overline{\rho_k e_k}.
\end{array}
\right.
\label{eq:iso2}
\end{equation}
By proposition 1, we know that  \eqref{eq:iso1} and \eqref{eq:iso2} admit an unique solution and therefore have the same solution, thus 
$$
\widetilde{p}_i = \overline{p},\quad \widetilde{\rho_k e_k}_i = \overline{\rho_k e_k}, ~~\forall i \in \ZZ, \quad 1\leq k \leq m.
$$ 

Now let us consider the projection step, in the present case it boils down to\begin{equation}
\left\{
\begin{array}{rclclcl}
(\rho u)^{n+1}_i &=& (\Lrhou)_i &-& \frac{\Delta t}{\Delta x} \overline{u}\Flux{ \Lrhou}_i , \\[0.15cm]
(\rho E)^{n+1}_i&=& (\Lrhoe)_i &-& \frac{\Delta t}{\Delta x}  \overline{u}\Flux{ \Lrhoe}_i, \\[0.15cm]
(\rho_k \zk)^{n+1}_i &=& (\Lrhoz)_i &-& \frac{\Delta t}{\Delta x}  \overline{u} \Flux{\Lrhoz}_i , 
\quad 1\leq k \leq m\\[0.15cm]
{\cal Z}^{n+1}_{k,i}   &=&  \widetilde{\cal Z}_{k,i}    &-& \frac{\Delta t}{\Delta x}  \overline{u} \Flux{\Lz}_i,
\quad 1\leq k \leq m. 
\end{array}
\right.
\label{eq:remap1}
\end{equation}
Assuming without any loss of generality that $\overline{u}>0$ and flux definitions  \eqref{eq: remap fluxes structure}, we have 
$$
\widetilde {\rho u}_{i+1/2} = \overline{u} \sum_{k=0}^m \overline{\rho_k} \widetilde{\zk}_{,i+1/2} \text{ and } \widetilde {\rho E}_{i+1/2} =  \sum_{k=0}^m (\overline{\rho_k e_k} +\dfrac{1}{2} \overline{\rho_k}~\overline{u}^2)\widetilde{\zk}_{,i+1/2}.
$$
The update of both momentum and energy then reads
\begin{equation}
(\rho u)^{n+1}_i - (\rho u)^{n}_i  = \overline{u} \sum_{k=1}^m \overline{\rho_k} (\zk_{,i}^{n+1}-\zk_{,i}^n),
\label{eq:mom}
\end{equation}
and 
\begin{equation}
(\rho E)^{n+1}_i - (\rho E)^{n}_i  = \overline{u} \sum_{k=1}^m (\overline{\rho_k e_k} +\dfrac{1}{2} \overline{\rho_k}~\overline{u}^2) (\zk_{,i}^{n+1}-\zk_{,i}^n), 
\label{eq:nrj}
\end{equation}
By combining both last relations of \eqref{eq:remap1}, we obtain 
$$
\rho_{k,i}^{n+1} = \overline{\rho}_k, \quad 1\leq k \leq m.
$$
Using \eqref{eq:mom}, and theorem \ref{theorem 1} we deduce that
$$
\rho_i^n(u_{i}^{n+1}-\overline{u}) = 0 \Rightarrow u_{i}^{n+1}=\overline{u}, , ~~\forall i \in \ZZ 
$$
similarly \eqref{eq:nrj} leads to
$$
 \sum_{k=1}^m ((\rho_k e_k)_i^{n+1}-\overline{\rho_k e_k})\zk_{,i}^{n+1} = 0. 
$$
Using the same uniqueness arguments as for in the Lagrange step for the definition of the pressure by isobaric closure we finally obtain
$$
p^{n+1}_i = \overline{p}, \quad (\rho_k e_k)^{n+1}_i = \overline{\rho_k e_k}, ~~\forall i \in \ZZ,
\quad 1\leq k \leq m.
$$
\end{proof}
\end{prop}

\end{document}

%% file: RES/1D/TRANSPORT/diffusion/antidiff_INT.tikz
% This file was created by matlab2tikz v0.4.4 running on MATLAB 8.0.
% Copyright (c) 2008--2013, Nico Schlömer <nico.schloemer@gmail.com>
% All rights reserved.
% 
% The latest updates can be retrieved from
%   http://www.mathworks.com/matlabcentral/fileexchange/22022-matlab2tikz
% where you can also make suggestions and rate matlab2tikz.
% 
%
% defining custom colors
\definecolor{mycolor1}{rgb}{1,0,1}%
\begin{tikzpicture}
\tikzstyle{every node}=[font=\footnotesize]
\begin{axis}[%
width=4.5cm,
height=4cm,
scale only axis,
xmin=0,
xmax=0.01,
xtick={0,0.0025,0.005,0.0075,0.01},
xticklabels={0,0.0025,0.005,0.0075,0.01},
xlabel={time ($\mathrm{s}$)},
ymin=0,
ymax=1,
ylabel={Diffusion cells ($\%$)},
title={Anti-diffusive scheme},
legend style={draw=black,fill=white,legend cell align=left}
]
\addplot [
color=black,
line width=1.5pt,
only marks,
mark=+,
mark options={solid}
]
table[row sep=crcr]{
0		0	 \\
0.0005	0.02	\\
0.001	0.02	\\
0.0015	0.02	\\
0.002	0.02	\\
0.0025	0.02	\\
0.003	0.02	\\
0.0035	0.02	\\
0.004	0.02	\\
0.0045	0.02	\\
0.005	0.02	\\
0.0055	0.02	\\
0.006	0.02	\\
0.0065	0.02	\\
0.007	0.02	\\
0.0075	0.02	\\
0.008	0.02	\\
0.0085	0.02	\\
0.009	0.02	\\
0.0095	0.02	\\
0.01		0      \\
};
\addlegendentry{$\z_1$};

\addplot [
color=red,
solid,
line width=1.5pt
]
table[row sep=crcr]{
0 0\\
0.0005 0.02\\
0.001 0.02\\
0.0015 0.02\\
0.002 0.02\\
0.0025 0.02\\
0.003 0.02\\
0.0035 0.02\\
0.004 0.02\\
0.0045 0.02\\
0.005 0.02\\
0.0055 0.02\\
0.006 0.02\\
0.0065 0.02\\
0.007 0.02\\
0.0075 0.02\\
0.008 0.02\\
0.0085 0.02\\
0.009 0.02\\
0.0095 0.02\\
0.01 0\\
};
\addlegendentry{$\z_2$};

\addplot [
color=blue,
dashed,
line width=1.5pt
]
table[row sep=crcr]{
0 0\\
0.0005 0.02\\
0.001 0.02\\
0.0015 0.02\\
0.002 0.02\\
0.0025 0.02\\
0.003 0.02\\
0.0035 0.02\\
0.004 0.02\\
0.0045 0.02\\
0.005 0.02\\
0.0055 0.02\\
0.006 0.02\\
0.0065 0.02\\
0.007 0.02\\
0.0075 0.02\\
0.008 0.02\\
0.0085 0.02\\
0.009 0.02\\
0.0095 0.02\\
0.01 0\\
};
\addlegendentry{$\z_3$};

\addplot [
color=green,
dashed,
line width=1.5pt
]
table[row sep=crcr]{
0 0\\
0.0005 0.02\\
0.001 0.02\\
0.0015 0.02\\
0.002 0.02\\
0.0025 0.02\\
0.003 0.02\\
0.0035 0.02\\
0.004 0.02\\
0.0045 0.02\\
0.005 0.02\\
0.0055 0.02\\
0.006 0.02\\
0.0065 0.02\\
0.007 0.02\\
0.0075 0.02\\
0.008 0.02\\
0.0085 0.02\\
0.009 0.02\\
0.0095 0.02\\
0.01 0\\
};
\addlegendentry{$\z_4$};

\addplot [
color=mycolor1,
dashed,
line width=1.5pt
]
table[row sep=crcr]{
0 0\\
0.0005 0.02\\
0.001 0.02\\
0.0015 0.02\\
0.002 0.02\\
0.0025 0.02\\
0.003 0.02\\
0.0035 0.02\\
0.004 0.02\\
0.0045 0.02\\
0.005 0.02\\
0.0055 0.02\\
0.006 0.02\\
0.0065 0.02\\
0.007 0.02\\
0.0075 0.02\\
0.008 0.02\\
0.0085 0.02\\
0.009 0.02\\
0.0095 0.02\\
0.01 0\\
};
\addlegendentry{$\z_5$};

\end{axis}
\end{tikzpicture}%

%% file: RES/1D/TRANSPORT/diffusion/upwind_INT.tikz
% This file was created by matlab2tikz v0.4.4 running on MATLAB 8.0.
% Copyright (c) 2008--2013, Nico Schlömer <nico.schloemer@gmail.com>
% All rights reserved.
% 
% The latest updates can be retrieved from
%   http://www.mathworks.com/matlabcentral/fileexchange/22022-matlab2tikz
% where you can also make suggestions and rate matlab2tikz.
% 
%
% defining custom colors
\definecolor{mycolor1}{rgb}{1,0,1}%
\begin{tikzpicture}
\tikzstyle{every node}=[font=\footnotesize]
\begin{axis}[%
width=4.5cm,
height=4cm,
scale only axis,
xmin=0,
xmax=0.01,
xtick={0,0.0025,0.005,0.0075,0.01},
xticklabels={0,0.0025,0.005,0.0075,0.01},
xlabel={time ($\mathrm{s}$)},
ymin=0,
ymax=1,
title={Upwind scheme},
legend style={draw=black,fill=white,legend cell align=left},
legend pos={south east}
]
\addplot [
color=black,
line width=1.5pt,
only marks,
mark=+,
mark options={solid}
]
table[row sep=crcr]{
0		0	 \\
0.0005	0.28	 \\
0.001	0.38	 \\
0.0015	0.45	 \\
0.002	0.51	 \\
0.0025	0.57	 \\
0.003	0.61	 \\
0.0035	0.65	 \\
0.004	0.69	 \\
0.0045	0.73	 \\
0.005	0.75	 \\
0.0055	0.79	 \\
0.006	0.82	 \\
0.0065	0.85	 \\
0.007	0.87	 \\
0.0075	0.91	 \\
0.008	0.93	 \\
0.0085	0.96	 \\
0.009	1	 \\
0.0095	1	 \\
0.01	1 \\
};
\addlegendentry{$\z_1$};
\addplot [
color=red,
solid,
line width=2.0pt
]
table[row sep=crcr]{
0	0	\\
0.0005	0.33	\\
0.001	0.43	\\
0.0015	0.5	\\
0.002	0.56	\\
0.0025	0.62	\\
0.003	0.66	\\
0.0035	0.7	\\
0.004	0.74	\\
0.0045	0.78	\\
0.005	0.8	\\
0.0055	0.84	\\
0.006	0.87	\\
0.0065	0.9	\\
0.007	0.92	\\
0.0075	0.96	\\
0.008	1	\\
0.0085	1	\\
0.009	1	\\
0.0095	1	\\
0.01	1\\
};
\addlegendentry{$\z_2$};
\addplot [
color=blue,
dashed,
line width=1.5pt
]
table[row sep=crcr]{
0	0	\\
0.0005	0.36	\\
0.001	0.56	\\
0.0015	0.7	\\
0.002	0.82	\\
0.0025	0.92	\\
0.003	0.96\\
0.0035 1\\
0.004 1\\
0.0045 1\\
0.005 1\\
0.0055 1\\
0.006 1\\
0.0065 1\\
0.007 1\\
0.0075 1\\
0.008 1\\
0.0085 1\\
0.009 1\\
0.0095 1\\
0.01 1\\
};
\addlegendentry{$\z_3$};
\addplot [
color=green,
dashed,
line width=1.5pt
]
table[row sep=crcr]{
0	0	\\
0.0005	0.36	\\
0.001	0.48	\\
0.0015	0.55	\\
0.002	0.61	\\
0.0025	0.67	\\
0.003	0.71	\\
0.0035	0.75	\\
0.004	0.79	\\
0.0045	0.83	\\
0.005	0.85	\\
0.0055	0.89	\\
0.006	0.92	\\
0.0065	0.95	\\
0.007	0.99\\
0.0075 1\\
0.008 1\\
0.0085 1\\
0.009 1\\
0.0095 1\\
0.01 1\\
};
\addlegendentry{$\z_4$};
\addplot [
color=mycolor1,
dashed,
line width=1.5pt
]
table[row sep=crcr]{
0 0\\
0.0005	0.28	 \\
0.001	0.38	 \\
0.0015	0.45	\\
0.002	0.51	\\
0.0025	0.57	\\
0.003	0.61	\\
0.0035	0.65	\\
0.004	0.69	\\
0.0045	0.73	\\
0.005	0.75	\\
0.0055	0.79	\\
0.006	0.82	\\
0.0065	0.85	\\
0.007	0.87	\\
0.0075	0.91	\\
0.008	0.93	\\
0.0085	0.96	\\
0.009 1\\
0.0095 1\\
0.01 1\\
};
\addlegendentry{$\z_5$};
\end{axis}
\end{tikzpicture}%

%% file: RES/1D/TAC/diffusion/antidiff_SOD.tikz
% This file was created by matlab2tikz v0.4.4 running on MATLAB 8.0.
% Copyright (c) 2008--2013, Nico Schlömer <nico.schloemer@gmail.com>
% All rights reserved.
% 
% The latest updates can be retrieved from
%   http://www.mathworks.com/matlabcentral/fileexchange/22022-matlab2tikz
% where you can also make suggestions and rate matlab2tikz.
% 
\begin{tikzpicture}
\tikzstyle{every node}=[font=\footnotesize]
\begin{axis}[%
width=4.25cm,
height=4cm,
scale only axis,
xmin=0,
xmax=0.12,
xlabel={time ($\mathrm{s}$)},
ymin=0,
ymax=0.1,
xtick={0,0.02,0.04,0.06,0.08,0.1,0.12},
xticklabels={0,0.02,0.04,0.06,0.08,0.1,0.12},
ylabel={Diffusion cells ($\%$)},
ytick={0,0.02,0.04,0.06,0.08,0.1},
yticklabels={0,0.02,0.04,0.06,0.08,0.1,0.12},
title={Anti-diffusive scheme},
legend style={draw=black,fill=white,legend cell align=left}
]
\addplot [
color=black,
solid,
line width=1.5pt
]
table[row sep=crcr]{
0 0\\
0.006 0.002\\
0.012 0.002\\
0.018 0.002\\
0.024 0.002\\
0.030 0.002\\
0.036 0.002\\
0.042 0.002\\
0.048 0.002\\
0.054 0.002\\
0.06 0.002\\
0.066 0.002\\
0.072 0.002\\
0.078 0.002\\
0.084 0.002\\
0.090 0.002\\
0.096 0.002\\
0.102 0.002\\
0.108 0.002\\
0.114 0.002\\
0.120 0.002\\
};
\addlegendentry{$\z_1$};

\addplot [
color=red,
solid,
line width=1.5pt
]
table[row sep=crcr]{
0 0\\
0.006 0.002\\
0.012 0.002\\
0.018 0.002\\
0.024 0.002\\
0.030 0.002\\
0.036 0.002\\
0.042 0.002\\
0.048 0.002\\
0.054 0.002\\
0.06 0.002\\
0.066 0.002\\
0.072 0.002\\
0.078 0.002\\
0.084 0.004\\
0.090 0.004\\
0.096 0.004\\
0.102 0.004\\
0.108 0.004\\
0.114 0.004\\
0.12 0.004\\
};
\addlegendentry{$\z_2$};

\addplot [
color=blue,
solid,
line width=1.5pt
]
table[row sep=crcr]{
0 0\\
0.006 0\\
0.012 0\\
0.018 0\\
0.024 0\\
0.03 0\\
0.036 0\\
0.042 0\\
0.048 0\\
0.054 0\\
0.06 0\\
0.066 0\\
0.072 0\\
0.078 0\\
0.084 0.002\\
0.090 0.002\\
0.096 0.002\\
0.102 0.002\\
0.108 0.002\\
0.114 0.002\\
0.120 0.002\\
};
\addlegendentry{$\z_3$};

\end{axis}
\end{tikzpicture}%

%% file: RES/1D/TAC/diffusion/upwind_SOD.tikz
% This file was created by matlab2tikz v0.4.4 running on MATLAB 8.0.
% Copyright (c) 2008--2013, Nico Schlömer <nico.schloemer@gmail.com>
% All rights reserved.
% 
% The latest updates can be retrieved from
%   http://www.mathworks.com/matlabcentral/fileexchange/22022-matlab2tikz
% where you can also make suggestions and rate matlab2tikz.
% 
\begin{tikzpicture}
\tikzstyle{every node}=[font=\footnotesize]
\begin{axis}[%
width=4.25cm,
height=4cm,
scale only axis,
xmin=0,
xmax=0.12,
xlabel={time ($\mathrm{s}$)},
ymin=0,
ymax=0.2,
xtick={0,0.02,0.04,0.06,0.08,0.1,0.12},
xticklabels={0,0.02,0.04,0.06,0.08,0.1,0.12},
ytick={0,0.04,0.08,0.12,0.16,0.2},
yticklabels={0,0.04,0.08,0.12,0.16,0.2},
title={Upwind scheme},
legend style={draw=black,fill=white,legend cell align=left},
legend pos={north west}
]
\addplot [
color=black,
solid,
line width=2.0pt
]
table[row sep=crcr]{
0 0\\
0.006 0.016\\
0.012 0.028\\
0.018 0.038\\
0.024 0.046\\
0.032 0.054\\
0.036 0.06\\
0.042 0.066\\
0.048 0.072\\
0.054 0.074\\
0.060 0.08\\
0.066 0.082\\
0.072 0.088\\
0.078 0.09\\
0.084 0.094\\
0.090 0.096\\
0.096 0.1\\
0.102 0.102\\
0.108 0.106\\
0.114 0.108\\
0.12   0.112\\
};
\addlegendentry{$\z_1$};

\addplot [
color=red,
solid,
line width=2.0pt
]
table[row sep=crcr]{
0 0\\
0.006 0.016\\
0.012 0.028\\
0.018 0.038\\
0.024 0.046\\
0.030 0.054\\
0.036 0.06\\
0.042 0.066\\
0.048 0.072\\
0.054 0.074\\
0.060 0.08\\
0.066 0.082\\
0.072 0.088\\
0.078 0.09\\
0.084 0.096\\
0.090 0.106\\
0.096 0.12\\
0.102 0.134\\
0.108 0.148\\
0.114 0.158\\
0.120 0.166\\
};
\addlegendentry{$\z_2$};

\addplot [
color=blue,
solid,
line width=2.0pt
]
table[row sep=crcr]{
0 0\\
0.006 0\\
0.012 0\\
0.018 0\\
0.024 0\\
0.030 0\\
0.036 0\\
0.042 0\\
0.048 0\\
0.054 0\\
0.060 0\\
0.066 0\\
0.072 0\\
0.078 0\\
0.084 0.002\\
0.090 0.01\\
0.096 0.02\\
0.102 0.032\\
0.108 0.042\\
0.114 0.05\\
0.12 0.054\\
};
\addlegendentry{$\z_3$};

\end{axis}
\end{tikzpicture}%